\documentclass[a4paper,10pt]{article}
\usepackage{authblk}

\usepackage{pgfplots}
\usepackage{url}

\usepackage[mathlines]{lineno}

\usepackage[backend=bibtex,sorting=nyt,sortcites]{biblatex}
\usepackage[all]{xy}
\usepackage[english]{babel}
\usepackage[margin=1in]{geometry}
\usepackage[utf8]{inputenc}
\usepackage{amsmath}
\usepackage{mathtools}
\usepackage{amsthm}
\usepackage{times}
\usepackage{hyperref}
\usepackage{pifont}
\newcommand{\cmark}{\ding{51}}
\newcommand{\xmark}{\ding{55}}
\usepackage{booktabs}  
\usepackage{multirow}

\usepackage{amsfonts}
\usepackage{graphicx}
\usepackage{amssymb}
\usepackage{arydshln}
\usepackage{tikz}
\usepackage{tikz-cd}
\usepackage{bbold}
\usepackage[percent]{overpic}

\newtheorem{theorem}{Theorem}
\newtheorem*{theorem*}{Theorem}
\newtheorem*{proposition*}{Proposition}

\newtheorem{proposition}{Proposition}[section]
\newtheorem{lemma}[proposition]{Lemma}
\newtheorem{remark}[proposition]{Remark}
\newtheorem{conjecture}[proposition]{Conjecture}

\newtheorem{question}[proposition]{Question}
\newtheorem{example}[proposition]{Example}
\newtheorem{Corollary}[proposition]{Corollary}

\newcommand{\mG}{\mathcal{G}}
\newcommand{\mH}{\mathcal{H}}

\newcommand{\mX}{\mathcal{X}}
\newcommand{\mM}{\mathcal{M}}
\newcommand{\mY}{\mathcal{Y}}
\newcommand{\mZ}{\mathcal{Z}}

\newcommand{\mU}{\mathcal{U}}
\newcommand{\mT}{\mathcal{T}}
\newcommand{\mS}{\mathcal{S}}
\newcommand{\mN}{\mathcal{V}}

\newcommand{\mE}{\mathcal{E}}

\newcommand{\R}{\mathbb{R}}

\newcommand{\catC}{\textbf{C}}

\newcommand{\globalInj}{\textbf{Glo}}
\newcommand{\homotopy}{\textbf{Hom}}
\newcommand{\spheres}{\textbf{Sph}}
\newcommand{\local}{\textbf{Loc}}

\newcommand{\dgh}{d_\mathrm{GH}}

\newcommand{\dgw}{d_{\mathrm{GW},p}}

\newcommand{\dgm}{d_{\mathrm{GM},p}}

\definecolor{darkblue}{rgb}{0.1, 0.1, 0.8}
\definecolor{darkred}{rgb}{0.8, 0.0, 0.0}
\definecolor{darkgreen}{rgb}{0.0, 0.8, 0.0}

\usepackage{mathtools}

\newcommand{\Expect}{{\rm I\kern-.3em E}}

\title{Distance Distributions and Inverse Problems for Metric Measure Spaces}

\author[1]{Facundo M\'emoli}
\author[2]{Tom Needham}

\affil[1]{Department of Mathematics and Department of Computer Science and Engineering,
	The Ohio State University\\ 
	\texttt{memoli@math.osu.edu}}
\affil[2]{Department of Mathematics,
	Florida State University\\
	\texttt{tneedham@fsu.edu}}

\pgfplotsset{compat=1.14}

\let\oldequation\equation
\let\oldendequation\endequation

\renewenvironment{equation}
  {\linenomathNonumbers\oldequation}
  {\oldendequation\endlinenomath}
  
\begin{document}
\maketitle
\sloppy

\begin{abstract}

Applications in data science, shape analysis and object classification frequently require comparison of probability distributions defined on different ambient spaces. To accomplish this, one requires a notion of distance on a given class of metric measure spaces---that is, compact metric spaces endowed with probability measures. Such distances are typically defined as comparisons between metric measure space invariants, such as distance distributions (also referred to as shape distributions, distance histograms or shape contexts in the literature). Generally, distances defined in terms of distance distributions are actually pseudometrics, in that they may vanish when comparing nonisomorphic spaces. The goal of this paper is to set up a formal framework for assessing the discrimininative power of distance distributions, i.e., the extent to which these pseudometrics fail to define proper metrics. We formulate several precise inverse problems in terms of these invariants and answer them in several categories of metric measure spaces, including the category of plane curves, where we give a counterexample to the Curve Histogram Conjecture of Brinkman and Olver, the categories of embedded and Riemannian manifolds, where we obtain sphere rigidity results, and the category of metric graphs, where we obtain a local injectivity result along the lines of classical work of Boutin and Kemper on point cloud configurations. The inverse problems are further contextualized by the introduction of a variant of the Gromov-Wasserstein distance on the space of metric measure spaces, which is inspired by the original Monge formulation of optimal transport.
\end{abstract}

\setcounter{tocdepth}{1}
\tableofcontents

\section{Introduction}

Classical optimal transport (OT) problems deal  with finding an optimal way to match two probability measures  $\mu$ and $\nu$ defined on the same ambient metric space $Z$. There are two  historical formulations of OT: the original Monge formulation \cite{monge1781memoire} and the Kantorovich formulation \cite{kantorovich1942translocation}. The former involves minimizing a certain cost over all measure preserving transformations between the measures, whereas the latter introduces a convex relaxation which enlarges the set of admissible mappings to those probability measures on the product space $Z\times Z$ whose marginals coincide with $\mu$ and $\nu$; such a measure is called a \emph{coupling} of $\mu$ and $\nu$. Optimal transport has long been an active field of study in pure mathematics (standard references are  \cite{villani2008optimal,santambrogio2015optimal}) and has recently become popular for machine learning applications, as it is a natural way to compare data distributions \cite{kolouri2017optimal,peyre2019computational}.

The convex relaxation corresponding to the Kantorovich formulation of OT has been adapted to the setting when one wishes to compare not just two probability measures defined on the same ambient space, but to the more general case when one wishes to compare triples of the form $\mathcal{X}=(X,d_X,\mu_X)$ where $(X,d_X)$ is a compact metric space and $\mu_X$ is a Borel probability measure on $X$. These triples are called \emph{metric measure spaces} (mm-spaces for short), and the resulting notion of dissimilarity between pairs of such triples is referred to as \emph{Gromov-Wasserstein (GW) distance} \cite{memoli2011gromov} (also referred to as \emph{distortion distance} in \cite{sturm2012space}). Gromov-Wasserstein distance has gained recent popularity in the machine learning community as a way to compare datasets which do not live in the same ambient space \cite{peyre2016gromov,hendrikson2016using,ezuz2017gwcnn,alvarez2018gromov,xu2019gromov,xu2019scalable,xu2020learning,bunne2019learning}.

We use the notation $\mM^m$ for the space of isomorphism classes of mm-spaces (an \emph{isomorphism} of mm-spaces being a measure-preserving isometry). The Gromov-Wasserstein distance defines a metric on $\mM^m$; however, the distance is not tractable to compute exactly and one therefore relies on pseudometric estimates of GW distance given by comparing certain distributional invariants of mm-spaces. These invariants, defined in terms of \emph{volume growth} in a mm-space, are useful in their own right, as they can be used to vectorize mm-space data in a principled manner. Informally, the goal of this paper is to characterize how badly these pseudometrics fail to define proper metrics. From a more mathematical perspective, we address the inverse problem of characterizing the extent to which volume growth data determines a space.

To make the goals of the paper more precise, define the \emph{local shape measure} of an mm-space $\mathcal{X}=(X,d_X,\mu_X)$ to be the function $dh_\mathcal{X}$ from $X$ into the set $\mathcal{P}(\mathbb{R})$ of probability measures on $\R$ defined by $x \mapsto \big(d_X(x,\cdot)\big)_\# \mu_X$, with $\big(d_X(x,\cdot)\big)_\# \mu_X$ denoting the \emph{pushforward} of $\mu_X$ by the function $d_X(x,\cdot): X \rightarrow \R$. The \emph{global shape measure} $dH_\mathcal{X}$ of $\mathcal{X}$ is the average of the local shape measure $dh_\mathcal{X}$; that is, for every measurable $A\subset \mathbb{R}$,  
\begin{equation}\label{eqn:global_shape_measure}
dH_\mathcal{X}(A) := \int_{X} dh_\mathcal{X}(x)(A)\,\mu_X(dx).
\end{equation}
One can then pose any number of problems on the discriminatory power of these invariants in  classes of mm-spaces; e.g.:
\begin{enumerate}
    \item \label{inverse_problems1} For plane curves $X$ and $Y$, considered as mm-spaces with extrinsic Euclidean distance and normalized arclength measure, does $dH_\mathcal{X} = dH_\mathcal{Y}$ imply $X$ and $Y$ differ by a rigid motion? 
    \item \label{inverse_problems_spheres} If $X$ is a Riemannian manifold, endowed with geodesic distance and normalized Riemannian volume, and $dH_\mathcal{X}$ is the same as the global shape measure of a unit sphere, must $X$ be isometric to a sphere?
    \item \label{inverse_problems2} For compact metric graphs $X$ and $Y$ with uniform measures, does the existence of a map $\phi:X \rightarrow Y$ such that $dh_\mathcal{X} = dh_\mathcal{Y} \circ \phi$ imply $X$ and $Y$ are isomorphic?
\end{enumerate}
Questions regarding the use of volume-growth information to characterize a space have been studied frequently in the Riemannian setting \cite{gray1979riemannian,calvaruso1997special,calvaruso1998semi,bokan2003geometric,csikos2012characterization,sturm2012space,balch2020distributions, balch2020expected}, where infinitesimal volume growth is related to scalar curvature, as well as for certain mm-spaces embedded in Euclidean space and endowed with extrinsic distance \cite{boutin2004reconstructing,brinkman2012invariant}. Inverse problems with a similar flavor have also been studied recently in the topological data analysis literature, where the goal is to characterize rougher spaces such as metric graphs via families of topological signatures \cite{oudot2017barcode,curry2018many,oudot2018inverse,fasy2019persistence, belton2020reconstructing}. In this paper we bridge the gap between these perspectives and establish an overarching framework for studying the recovery of general metric measure spaces from volume growth data. 

We provide a counterexample to Problem \ref{inverse_problems1}, thereby resolving the Curve Histogram Conjecture of Brinkman and Olver \cite{brinkman2012invariant}. Problem \ref{inverse_problems_spheres} can be viewed as a problem on \emph{sphere rigidity}, in the tradition of similar problems in differential geometry. We answer this question, and a similar question for embedded manifolds with \emph{extrinsic} distance,  under additional curvature assumptions. Problem \ref{inverse_problems2} is reminiscent of the classical Graph Reconstruction Conjecture of Kelly \cite{kelly1957congruence} and Ulam \cite{ulam1960collection}, and is also of applied interest as functions similar to $dh_\mathcal{X}$ have recently been incorporated into algorithms for graph classification and analysis \cite{tsitsulin2018netlsd,sato2020fast}. The answer to this problem is more subtle and depends on assumptions about the  regularity of the map $\phi$ and the topologies of the graphs. The following subsection provides more details on the main results of the paper.

At a high level, one of the main goals of this paper is to formalize the study of inverse problems in various categories of mm-spaces; e.g., Riemannian manifolds, embedded manifolds, and metric graphs equipped with measure-preserving morphisms of various prescribed regularities. To this end, we formulate precisely several specific inverse problems of this type in Section \ref{sec:distance_distributions_section}. Our inverse problems are resolved for several categories of mm-spaces in Section \ref{sec:plane_curves}, which covers smooth manifolds, and Section \ref{sec:metric_trees}, which covers metric graphs. Although these sections treat similar problems, the techniques used are fairly distinct:

\begin{itemize}
    \item Section \ref{sec:plane_curves}, treating manifolds,  relies on applying a diverse collection of previous results, including geometric constructions that have appeared as counterexamples to related but different conjectures \cite{mallows1970linear,garcia2016pairs}, Taylor expansions of volume growth functions \cite{gray1979riemannian,karp1989volume} and sphere rigidity results \cite{cheng,kokkendorff2008characterizing}. This section demonstrates that these previous works can be viewed through the common lens of inverse problems for volume growth invariants.
    \item Section \ref{sec:metric_trees}, treating metric graphs, requires more technical work and generally characterizes larger scale behavior of the volume growth invariants. Our techniques use ideas from combinatorial graph theory and computational topology.
\end{itemize}
We now provide an overview of our main results in these settings.

\subsection{Overview of Main Results}\label{sec:overview}

\paragraph{Section \ref{sec:distance_distributions_section}.} This section outlines the basic notation, concepts and main problems considered in the rest of the paper. Here we introduce the notion of subcategories of mm-spaces. A \emph{subcategory of mm-spaces} $\catC$ is a category whose object set $\mathrm{Obj}_\catC$ is a collection of isomorphism classes of mm-spaces and whose morphism set $\mathrm{Mor}_\catC$ consists of measure-preserving maps, considered up to natural equivalence. For example, one could consider the subcategory of Riemannian manifolds, whose objects are equivalence classes of Riemannian manifolds endowed with geodesic distance and normalized Riemannian volume and whose morphisms are equivalence classes of smooth measure-preserving maps. Throughout the paper, we abuse terminology slightly and use mm-spaces as representatives for elements of $\mathrm{Obj}_\catC$ (which are technically equivalence classes thereof), and likewise use maps to represent elements of $\mathrm{Mor}_\catC$.  

Given a subcategory of mm-spaces $\catC$, we pose the following inverse problems, stated here  informally:
\begin{itemize}
\item[(\globalInj)] \textbf{Global Injectivity.} For $\mX,\mY \in \mathrm{Obj}_\catC$, does the existence of $\phi \in \mathrm{Mor}_\catC$ such that $dh_\mX = dh_\mY \circ \phi$ imply $\mX \approx \mY$? More strongly, does $dH_\mX = dH_\mY$ imply $\mX \approx \mY$? 
\item[(\homotopy)] \textbf{Homotopy Type Characterization.} Does the existence of $\phi \in \mathrm{Mor}_\catC$ such that $dh_\mX = dh_\mY \circ \phi$ (more strongly, does $dH_\mX = dH_\mY$) imply the underlying spaces $X$ and $Y$ are homotopy equivalent?
\item[(\spheres)] \textbf{Sphere Rigidity.} If $\mathrm{Obj}_\catC$ has a notion of a unit sphere $\mathcal{S}$, does $dH_\mX = dH_\mathcal{S}$ imply $\mX \approx \mathcal{S}$?
\item[(\local)] \textbf{Local Injectivity.} For $\mX \in \mathrm{Obj}_\catC$, does there exist $\epsilon_\mX > 0$ such that the existence of $\phi \in \mathrm{Mor}_\catC$ with $dh_\mX = dh_\mY \circ \phi$ implies $\mX \approx \mY$, provided $Y$ is $\epsilon_\mX$-close to $X$ in Gromov-Hausdorff distance?
\end{itemize}
Problem (\globalInj) is an inverse problem in the sense that it asks whether the map taking a mm-space from a prescribed class to its global shape measure is injective, i.e., whether it has a left inverse. Problems (\homotopy) and (\spheres) are relaxations of the first problem, generalizing problems studied classically in differential geometry to handle categories whose objects are not necessarily smooth. In particular, (\spheres) can be seen as a generalization of the sphere rigidity problems which are the focus of much research in classical differential geometry (e.g., Aleksandrov's Theorem \cite{aleksandrov1962uniqueness} and Cheng's Rigidity Theorem \cite{cheng}). Problem (\local) asks a localized version of (\globalInj). 

The precise statements of these inverse problems, provided in Section \ref{sec:inverse_problems}, are given in terms of pseudometrics $L_\mathrm{h}^\catC$ and $L_\mathrm{H}^\catC$ on $\mathrm{Obj}_\catC$, defined using the local and global shape measures (or, rather, cumulative versions referred to as \emph{distance distributions}). To provide context for the pseudometrics $L_\mathrm{h}^\catC$ and $L_\mathrm{H}^\catC$ when the morphisms of $\catC$ are restricted, we introduce a variant of Gromov-Wasserstein distance which is equipped to take these restrictions into account. We define for each $p \geq 1$ a  dissimilarity measure on $\mathrm{Obj}_\catC$ called \emph{Gromov-Monge $p$-distance}, denoted $\dgm^\catC$, and outline some basic properties of the family of distances. Roughly, the Gromov-Monge $p$-distance between mm-spaces $\mathcal{X}$ and $\mathcal{Y}$ in $\mathrm{Obj}_\catC$ is given by $\dgm^\catC(\mathcal{X},\mathcal{Y}) = \inf_\phi \mathrm{cost}_p(\phi)$, where the infimum is over measure-preserving maps $\phi:X \rightarrow Y$ belonging to $\mathrm{Mor}_\catC$ and $\mathrm{cost}_p$ is an $L^p$-type cost function measuring the overall geometric distortion incurred by $\phi$ (e.g., $\phi$ is an isometry if and only if $\mathrm{cost}_p(\phi) = 0$).  The main theorem of this section, Theorem \ref{thm:dgm_is_metric}, describes the sense in which each $\dgm^\catC$ is a \emph{distance}. We then show that $L_\mathrm{h}^\catC$ and $L_\mathrm{H}^\catC$ give computationally tractable lower bounds on $\dgm^\catC$.

\paragraph{Section \ref{sec:plane_curves}.} Here we begin our study of inverse problems (\globalInj)--(\local) concerning the recovery of mm-spaces from volume growth data. This section is focused on smooth manifolds, considered both as embedded submanifolds of Euclidean space with extrinsic distance and as abstract Riemannian manifolds with geodesic distance. We state our main results here somewhat informally and in each case refer back to one of the inverse problems, to which the result is relevant. 

Section \ref{sec:plane_curves} begins with plane curves (considered as mm-spaces with extrinsic distance and normalized arclength measure). The \emph{Curve Histogram Conjecture} of Brinkman and Olver \cite{brinkman2012invariant} proposes that the global distance distribution is a complete invariant of a (sufficiently regular) plane curve $\mX$ in the sense that $dH_{\mX} = dH_{\mY}$ implies that $\mX$ and $\mY$ differ by a rigid motion; i.e., that (\globalInj) is answered in the afirmative. We adapt a counterexample of Mallows and Clark to a conjecture by Blaschke \cite{mallows1970linear} to resolve the Curve Histogram Conjecture:

\begin{theorem}[Proposition \ref{prop:curve_histogram_conjecture_false} and Example \ref{exmp:pessimistic}]
    The curves depicted in Figure \ref{fig:curve_histogram_conjecture} give a counterexample to the Curve Histogram Conjecture. Moreover, for any $\epsilon > 0$ there exists a pair of non-isometric curves $\mX$ and $\mY$ which are both $\epsilon$-close to the unit circle in Hausdorff distance for which $dH_{\mX} = dH_{\mY}$ (\globalInj).
\end{theorem}

Leveraging Taylor expansions of volume growth functions from \cite{karp1989volume,gray1979riemannian}, we obtain sphere rigidity results in the categories of embedded and smooth manifolds.

\begin{theorem}[Theorem \ref{thm:sphere_distributions}, Proposition \ref{prop:sphere_distributions_riemannian} and Examples \ref{exmp:surfaces_same_distribution} and \ref{exmp:pessimistic_riemannain}]
The unit sphere in $\R^{d+1}$ is uniquely determined by its global shape measure among:
\begin{enumerate}
    \item smooth hypersurfaces of $\R^{d+1}$ (with extrinsic distance, normalized uniform measure) satisfying certain  curvature conditions (\spheres);
    \item smooth $d$-dimensional Riemannian manifolds (with geodesic distance, normalized Riemannian measure) satisfying certain curvature conditions (\spheres).
\end{enumerate}
However, there are embedded (respectively, Riemannian) manifolds $\mX$ and $\mY$ which are arbitrarily close to the $d$-sphere in Gromov-Hausdorff distance such that $dH_\mX = dH_\mY$.
\end{theorem}

In the Riemannian setting, this result can be pushed further to a characterization result on other constant curvature manifolds.

\begin{theorem*}[Theorem \ref{thm:surface_characterization}]
A 2-dimensional Riemannian manifold with constant Gauss curvature $\kappa$ is determined by its global shape measure up to diffeomorphism (\homotopy). Moreover, if $\kappa > 0$ then the manifold is determined up to isometry (\spheres). 
\end{theorem*}

In the $\kappa = 0$ (flat torus) case of the above theorem, we provide an example of two nonisometric surfaces with the same distance distributions (Example \ref{exmp:flat_tori}).

\paragraph{Section \ref{sec:metric_trees}.}  Here we consider volume growth inverse problems for the class of \emph{metric graphs}; roughly, metric spaces which are constructed by gluing intervals together to form complex filamentary structures (see Figure \ref{fig:neuron}). Metric graphs serve as models for road networks \cite{ahmed2014local}, neurons \cite{kong2005diversity} and blood vessel systems \cite{chalopin2001modeling,charnoz2005tree}---see \cite{aanjaneya2012metric} for many more examples. While local (and hence global) distance distributions cannot distinguish metric graphs in general (Example \ref{exmp:non_injectivity_local_dist_c_trees}), we show that local distance distributions determine metric graphs in the following sense.

\begin{figure}
\begin{center}
    \includegraphics[width = 0.3\textwidth]{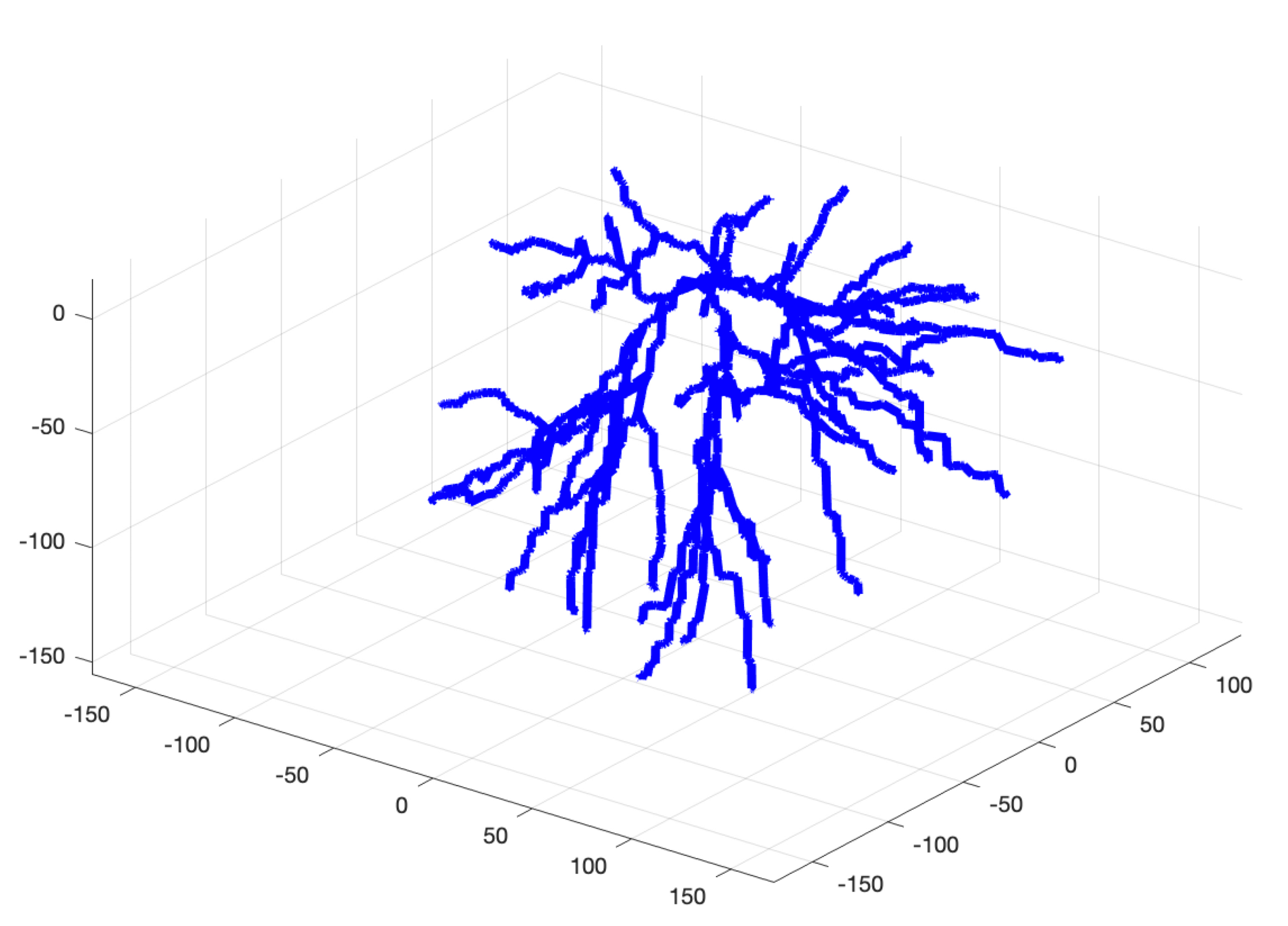}
    \qquad
    \includegraphics[width = 0.3\textwidth]{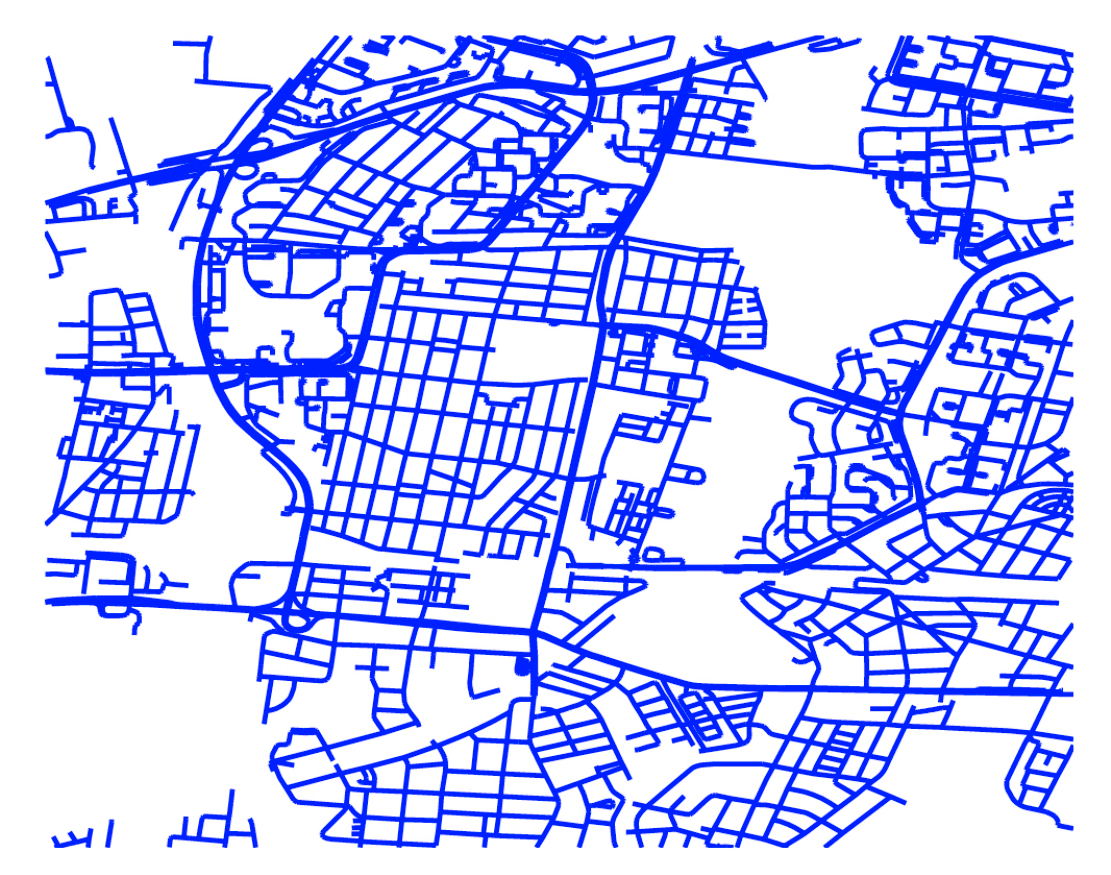}
\end{center}
\caption{Examples of metric graphs, continuous mm-spaces used to represent complex filamentary structures. (Left) Neuron M1KO28 from the NeuroMorpho.Org database \cite{ascoli2007neuromorpho,ballesteros2010alterations}. (Right) Road network from \cite{ahmed2014local}.}\label{fig:neuron}
\end{figure}

\begin{theorem*}[Theorem \ref{thm:continuous_maps_graphs}] For metric graphs $\mathcal{G}$ and $\mathcal{H}$, if there exists a continuous measure-preserving map $\phi:G \rightarrow H$ with $dh_\mG = dh_\mH \circ \phi$, then $\mG$ and $\mH$  are isomorphic (\globalInj).
\end{theorem*}

We also obtain homotopy type characterization and sphere rigidity results for metric graphs. Observe that the homotopy type characterization result drops the continuity requirement on the measure-preserving map.

\begin{theorem*}[Theorem \ref{thm:homotopy_type_metric_graphs}]
For metric graphs $\mathcal{G}$ and $\mathcal{H}$, if there exists a measure preserving map $\phi:G \rightarrow H$ such that $dh_\mG = dh_\mH \circ \phi$ then $G$ and $H$ are homotopy equivalent (\homotopy).
\end{theorem*}

\begin{theorem*}[Theorem \ref{thm:sphere_characterization_metric_graphs}]
If a metric graph $\mathcal{G}$ satisfies $dH_\mathcal{G} = dH_\mathcal{C}$, where $C$ is a circle endowed with arclength distance and normalized length measure, then $\mathcal{G} \approx \mathcal{C}$ (\spheres).
\end{theorem*}

A local injectivity result is obtained in the subcategory of \emph{metric trees} (contractible metric graphs). 

\begin{theorem*}[Theorem \ref{thm:main_theorem_mtrees}]
Local shape measures distinguish trees locally in the sense of (\local): for every metric tree $\mT$ there exists $\epsilon_\mT > 0$ such that if $\mS$ is another metric tree which is $\epsilon_\mT$-close to $\mT$ in Gromov-Hausdorff distance and such that $dh_\mT = dh_{\mS} \circ \phi$ for some measure-preserving map $\phi$, then $\mS \approx \mT$.
\end{theorem*}

The proof establishes a connection to a distance between metric trees first considered in the computational topology  community called the \emph{interleaving distance} \cite{morozov2013interleaving,agarwal2015computing}. 

\paragraph{Section \ref{sec:discussion}.} 

The paper concludes with a summary of our progress on inverse problems of for mm-spaces in Table \ref{tab:summary}. There are many open questions naturally arising from our work, which are pointed out in ``Question" environments throughout the paper. 

\section{Distance Distributions}\label{sec:distance_distributions_section}

\subsection{Preliminaries}\label{sec:preliminary_definitions}

\subsubsection*{Metric Measure Spaces}

 Let $\mM$ denote the collection of isometry classes of compact metric spaces $(X,d_X)$. In this paper, we primarily consider metric spaces which have been endowed with extra data in the form of a probability measure. A triple $(X,d_X,\mu_X)$ where $(X,d_X)$ is a compact metric space and $\mu_X$ is a Borel probability measure on $X$ is referred to as a \emph{metric measure space} (\emph{mm-space}) and will be denoted by $\mX = (X,d_X,\mu_X)$. For convenience, we always assume that the support of $\mu_X$ is all of $X$.
 
 Let $\mathcal{P}(X)$ denote the set of all fully-supported Borel probability measures on $X$. For a Borel measurable map $\phi:X \rightarrow Y$ between metric spaces and a measure $\mu_X$ on $X$, we will use the notation $\phi_\# \mu_X$ for the \emph{pushforward} of $\mu_X$ by $\phi$. This is defined on a Borel subset $A \subset Y$ by $\phi_\# \mu_X (A) = \mu_X(\phi^{-1}(A))$. An \emph{isomorphism} between mm-spaces $\mX$ and $\mY$ is a measure-preserving map $\phi:X\rightarrow Y$ which is also a metric isometry; i.e., $\phi_\# \mu_X = \mu_Y$ and $d_X = d_Y\circ(\phi \times \phi).$ When such an isomorphism exists, we write $\mX\approx \mY$. Recall that $\mM^m$ denotes the collection of isomorphism classes of mm-spaces.

Let $\pi_X:X \times Y \rightarrow X$ and $\pi_Y:X \times Y \rightarrow Y$ be coordinate projection maps. For mm-spaces $\mX$ and $\mY$, define 
\begin{linenomath}\begin{equation*}
\mU(\mu_X,\mu_Y) := \{\mu \in \mathcal{P}(X \times Y) \mid (\pi_X)_\# \mu = \mu_X \mbox{ and } (\pi_Y)_\# \mu = \mu_Y \}
\end{equation*}\end{linenomath} 
to be the set of all probability measures on $X\times Y$ with marginals $\mu_X$ and $\mu_Y$, which are called \emph{couplings} between $\mu_X$ and $\mu_Y$. The set of couplings is never empty, as it contains the \emph{product measure} $\mu_X \otimes \mu_Y$, which is uniquely determined by the property $\mu_X \otimes \mu_Y (A \times B) = \mu_X(A)\mu_Y(B)$ for all Borel sets $A\subset X$ and $B\subset Y$. In a similar vein, let 
\begin{linenomath}\begin{equation*}
\mT(\mu_X,\mu_Y):=\left\{\phi:X\rightarrow Y \mid \,\phi_\#\mu_X=\mu_Y\right\}
\end{equation*}\end{linenomath}
denote the set of all measure-preserving maps between $X$ and $Y$. Notice that it could be that $\mT=\emptyset$; for example, there is never a measure-preserving map from the mm-space containing a single point to an mm-space with larger cardinality. Given any $\phi\in \mT(\mu_X,\mu_Y)$, one can consider the probability measure $\mu_\phi$ on $X\times Y$ given by 
\begin{equation}\label{eqn:pushforward_coupling}
    \mu_\phi := (\mathrm{id}_X \times \phi)_\# \mu_X,
\end{equation}
where $\mathrm{id}_X$ is the identity map on $X$. It is straightforward to check that $\mu_\phi\in \mU(\mu_X,\mu_Y).$ 

\subsubsection*{Optimal Transport}

In this paper, we consider comparisons between mm-spaces. To build up methods for doing so, we begin with a classical notion of comparing probability distributions relative to some auxilliary cost function. Let $\mu_X$ and $\mu_Y$ be Borel probability measures on Polish spaces $X$ and $Y$, respectively, and let $c:X \times Y \rightarrow \R$ be some \emph{cost function}. The \emph{Monge optimal transport problem} seeks to solve the optimization problem
\begin{linenomath}\begin{equation*}
\inf_{\phi \in \mathcal{T}(\mu_X,\mu_Y)} \int_X c(x,\phi(x)) \; \mu_X(dx).
\end{equation*}\end{linenomath}
The problem was originally formulated by Monge in the 18th century \cite{monge1781memoire} in the context of transporting mass from one location to another, where $X = Y = \R^n$ and $c(x,y)$ is Euclidean distance. 

The Monge optimal transport problem is potentially ill-posed; as was observed above, the set $\mathcal{T}(\mu_X,\mu_Y)$ can be empty. This problem can be relaxed to the more well-behaved {Kantorovich-Rubinstein-Wasserstein optimal transport problem}, which seeks to solve
\begin{linenomath}\begin{equation}\label{eqn:wasserstein_OT}
\inf_{\mu \in \mathcal{U}(\mu_X,\mu_Y)} \int_{X \times Y} c(x,y) \; \mu (dx \times dy).
\end{equation}\end{linenomath}
The following is a basic result in optimal transport theory (see \cite{ambrosio2008gradient}).

\begin{lemma}\label{lem:realized_by_coupling}
If $c$ is lower semi-continuous and $\mU(\mu_X,\mu_Y)$ is a tight collection of measures (in the sense of Prokhorov's Theorem) then the infimum \eqref{eqn:wasserstein_OT} is always realized by some coupling $\mu \in \mU(\mu_X,\mu_Y)$.
\end{lemma}

The most commonly considered cost functions are in the setting that the probability measures are defined on the same ground space $X$, which is endowed with a metric $d_X$. The Kantorovich  optimization problem then yields a family of metrics on (a subset of) $\mathcal{P}(X)$. The \emph{Wasserstein $p$-distance} for $(X,d_X)$ is 
\begin{linenomath}\begin{equation}\label{eqn:wasserstein_distance}
d_{\mathrm{W},p}^X(\mu_1,\mu_2)  := \inf_{\mu \in \mathcal{U}(\mu_1,\mu_2)} \left(\int_{X \times X} d_X(x,y)^p \; \mu(dx \times dy) \right)^{1/p}.
\end{equation}\end{linenomath}
If $X$ is noncompact, the metric is defined on an appropriate subset of $\mathcal{P}(X)$ where the integral converges.

 \subsubsection*{Gromov-Hausdorff Distance}
 
 The classical notion of Hausdorff distance between compact subsets of an ambient metric space $Z$ can be vastly generalized to give a metric on $\mM$. The  \emph{Gromov-Hausdorff (GH) distance} $\dgh$ is defined by
\begin{equation}\label{eqn:gromov_hausdorff_defn}
\dgh(X,Y) := \inf_{Z,\phi_X,\phi_Y} d^Z_\mathrm{H}(\phi_X(X),\phi_Y(Y)),
\end{equation}
where the infimum is taken over all ambient metric spaces $(Z,d_Z)$ and isometric embeddings $\phi_X:X \rightarrow Z$ and $\phi_Y:Y \rightarrow Z$, and where $d^Z_H$ denotes Hausdorff distance in $Z$.

The GH distance can be reformulated in several ways, and we will make particular use of a reformulation in terms of the distortion of a correspondence. For compact metric spaces $X$ and $Y$, let $\Gamma_{X,Y}:X \times Y \times X \times Y \rightarrow \R$ denote the \emph{distortion map}, defined by 
\begin{linenomath}\begin{equation*}
\Gamma_{X,Y}(x,y,x',y') := |d_X(x,x') - d_Y(y,y')|.
\end{equation*}\end{linenomath}
A \emph{correspondence} between sets $X$ and $Y$ is a subset $R \subset X \times Y$ such that the coordinate projection maps $\pi_X:X \times Y \rightarrow X$ and $\pi_Y:X \times Y \rightarrow Y$ give surjections when restricted to $R$. The set of all correspondences between $X$ and $Y$ is denoted $\mathcal{R}(X,Y)$. We then have the following reformulation of GH distance \cite{burago2001course}:
\begin{equation}\label{eqn:GH_reformulation}
\dgh(X,Y) = \frac{1}{2} \inf_{R \in \mathcal{R}(X,Y)} \sup_{(x,y),(x',y') \in R} \Gamma_{X,Y}(x,y,x',y').
\end{equation}
One might notice that the righthand side of \eqref{eqn:GH_reformulation} takes the form of an infimum of instances of an $L^\infty$ norm of the function $\Gamma_{X,Y}$, suggesting a relaxation to an $L^p$ norm. 

\subsubsection*{Gromov-Wasserstein Distance}

The \emph{Gromov-Wasserstein $p$-distance} $\dgw$ gives a Kantorovich-style relaxation of the formulation \eqref{eqn:GH_reformulation} of Gromov-Hausdorff distance. It is defined on mm-spaces $\mX$ and $\mY$ by 
\begin{linenomath}\begin{align*}
&\dgw(\mX,\mY) := \inf_{\mu \in \mU(\mu_X,\mu_Y)} \|\Gamma_{X,Y}\|_{L^p(\mu \otimes \mu)} \\
\hspace{.1in} &= \inf_{\mu \in \mU(\mu_X,\mu_Y)} \left(\iint_{X \times Y \times X \times Y} \big|d_X(x,x') - d_Y(y,y')\big|^p \mu \otimes \mu (dx \times dy \times dx' \times dy') \right)^{1/p}.
\end{align*}\end{linenomath}
Gromov-Wasserstein $p$-distance defines a metric on $\mM^m$ whose topological properties and applications to shape analysis were explored in  \cite{memoli2007use, memoli2011gromov}. Geometrical properties of the metric such as its  geodesic structure and bounds on its Alexandrov curvature were explored in \cite{sturm2012space}. Gromov-Wasserstein distance has since seen many applications to data science and shape analysis, including studying network data sets \cite{hendrikson2016using}, computing barycenters of collections of shapes \cite{peyre2016gromov} and  improving  deep learning algorithms for shape classification \cite{ezuz2017gwcnn}.

Exact computation of GW distance on finite mm-spaces distance is NP-Hard, as it can be expressed as a quadratic programming problem with a nonconvex objective function. Even approximating GW distance via gradient descent is computationally taxing, incurring a computational cost of $O(n^4)$ with a naive implementation, although this can be improved to $O(n^3\log(n))$ in the $p=2$ case \cite{peyre2016gromov} and further improvements can be gained in the setting of sparse graphs \cite{xu2019gromov}. For this reason, it is convenient to consider the computable lower bounds on GW distance defined in the following subsection, which themselves serve as pseudometrics on $\mM^m$.

\subsubsection*{Notation and Terminology}

Throughout the paper, we consider various relaxations of the notion of a metric. A function $d_X:X \times X \rightarrow \R$ is 
\begin{itemize}
\item a \emph{pseudometric} if it satisfies the metric axioms, except that we allow $d_X(x,y) = 0$ for some pairs $x \neq y$;
\item a \emph{quasimetric} if it satisfies the metric axioms, except that we allow $d_X(x,y) \neq d_X(y,x)$ for some $x$ and $y$;
\item an \emph{extended metric} if it satisfies the metric axioms, except that we allow $d_X(x,y) = \infty$ for some pairs $x$ and $y$. We can similarly define extended pseudometrics and extended quasimetrics.
\end{itemize}
For a metric space $(X,d_X)$, $B_X(x,r)$ denotes the metric ball centered at $x \in X$ with radius $r$. We use $\overline{U}$ to denote the closure of a subset $U$ of a topological space.

\subsection{Distance Distributions and Associated Pseudometrics}\label{sec:distance_distributions}

Distance distributions are isomorphism invariants of mm-spaces, obtained as cumulative versions of the shape measures defined in the introduction, and we show below how they can be used to define pseudometrics on the space $\mM^m$. These lead to reformulations of the inverse problems (\globalInj)--(\local) established in the introduction. The new formulations are better contextualized as questions about the effectiveness of these pseudometrics in distinguishing mm-spaces within given categories.
 
\subsubsection*{Global Distance Distributions}

The \emph{(global) distance distribution} of $\mX \in \mM^m$ is the function $H_\mX : \R_{\geq 0} \rightarrow \R_{\geq 0}$  defined by
\begin{linenomath}\begin{equation*}
H_\mX(r) := \mu_X \otimes \mu_X \left(\{(x,x') \in X \times X \mid d_X(x,x') \leq r\}\right).
\end{equation*}\end{linenomath}
Distance distributions (sometimes referred to as \emph{shape distributions} or \emph{distance histograms}) are a standard tool for summarizing metric spaces. They have been used for classification of geometric objects \cite{osada2002shape,bonetti2005interpoint} and their mathematical properties have been studied in a variety of contexts \cite{boutin2004reconstructing,brinkman2012invariant,berrendero2016shape}. The global distance distribution is related to the global shape measure $dH_\mX$ defined in the introduction by the formula
\begin{linenomath}\begin{equation*}
H_{\mX}(r) =  \int_0^r dH_{\mX}(ds).
\end{equation*}\end{linenomath}

For $p\geq 1$, define the function $L_{\mathrm{H},p}:\mM^m \times \mM^m \rightarrow \R$ by 
\begin{linenomath}\begin{equation*}
L_{\mathrm{H},p}(\mX,\mY) := d_{\mathrm{W},p}^\R(dH_\mX,dH_\mY).
\end{equation*}\end{linenomath}
Since the global shape measures $dH_\mX$ and $dH_\mY$ are (compactly supported) probability distributions on $\R$, and Kantorovich optimal transport has a closed form solution in the 1-dimensional case given in terms of generalized inverses of cumulative distributions (see \cite[Remark 2.19]{villani2003topics}), we have the more explicit formula
\begin{linenomath}\begin{equation*}
L_{\mathrm{H},p}(\mX,\mY) = \left(\int_0^1 \left| H_\mX^{-1}(u) - H_\mY^{-1}(u)\right|^p \; du\right)^{1/p},
\end{equation*}\end{linenomath}
where, for $u \in [0,1]$, $H_\mX^{-1}(u) := \inf \{r \geq 0 \mid H_\mX(r) > u \}$ denotes the generalized inverse of $H_\mX$ at $u$. The $p=1$ case plays a special role, since one can show \cite{memoli2011gromov} that in this case the formula simplifies to
\begin{linenomath}\begin{equation*}
L_\mathrm{H}(\mX,\mY) := L_{\mathrm{H},1}(\mX,\mY) = \int_0^\infty \left|H_\mX(u) - H_\mY(u)\right| \; du.
\end{equation*}\end{linenomath}

The quantity $L_{\mathrm{H},p}(\mX,\mY)$ (particularly in the $p=1$ case) gives an intuitive notion of distance between two mm-spaces. To be precise, it is easy to check that $L_{\mathrm{H},p}$ is symmetric in its arguments, that $L_{\mathrm{H},p}(\mX,\mY) \geq 0$ for all $\mX,\mY$, that if $\mX$ and $\mY$ are isomorphic then $L_{\mathrm{H},p}(\mX,\mY) = 0$, and that $L_{\mathrm{H},p}$ satisfies the triangle inequality. The obstruction to $L_{\mathrm{H},p}$ defining a proper metric on $\mM^m$ is that it can vanish for nonisomorphic mm-spaces. 

\begin{example}\label{exmp:integer_sets}
A simple example of nonisomorphic mm-spaces with the same distance distributions appears in \cite{bloom1977counterexample} in the context of difference sets of integers. Define mm-spaces $\mX,\mY$ by taking $X = \{0,1,4,10,12,17\}$ and $Y = \{0,1,8,11,13,17\}$ and endowing each set with Euclidean distance and uniform measures. One can easily check that $\mX \not \approx \mY$, yet $L_\mathrm{H}(\mX,\mY) = 0$. 
\end{example}

Putting the above observations together, we deduce the following.

\begin{proposition}\label{prop:global_dist_pseudometric}
For each $p \geq 1$, $L_{\mathrm{H},p}$ defines a pseudometric on $\mM^m$. 
\end{proposition}

 \subsubsection*{Local Distance Distributions}

The \emph{local distance distribution} of a mm-space $\mX$  is the function $h_\mX:X\times \R_{\geq 0} \rightarrow \R_{\geq 0}$ defined by 
\begin{linenomath}\begin{equation*}
h_\mX(x,r) := \mu_X\big(\overline{B_X(x,r)}\big).
\end{equation*}\end{linenomath}
 Local distance distributions (and closely related invariants, such as \emph{shape contexts}) have also appeared frequently in the shape analysis literature (e.g., \cite{gelfand2005robust,belongie2006matching,shi2007direct}). For each $x \in X$, the function $r \mapsto h_\mX(x,r)$ measures volume growth, localized at $x$, and inverse problems about these volume growth functions similar to those considered in the present paper have appeared previously in \cite{gray1979riemannian,sturm2012space}.  The local distance distribution is related to the the global distance distribution and to the local shape distribution $dh_\mX$ defined in the introduction via, respectively, 
 \begin{linenomath}\begin{equation*}
 H_\mX(r) = \int_X h_\mX(x,r) \; \mu_X(dx) \quad \mbox{and} \quad h_{\mX}(x,r) =  \int_0^r dh_{\mX}(x)(ds).
 \end{equation*}\end{linenomath}

As in the case of global distance distributions, local distance distributions can be used to define a family of pseudometrics on $\mM^m$. These pseudometrics necessarily have a more complicated structure, since comparison of local distributions requires some registration between elements of $X$ and $Y$. For each $p \geq 1$, we define $L^{\mathrm{K}}_{\mathrm{h},p}:\mM^m \times \mM^m \rightarrow \R$ by
\begin{linenomath}\begin{equation*}
L^{\mathrm{K}}_{\mathrm{h},p}(\mX,\mY) := \inf_{\mu \in \mU(\mu_X,\mu_Y)} \left(\int_{X\times Y} d_{\mathrm{W},p}^\R(dh_\mX(x),dh_\mY(y))^p \; \mu(dx \times dy) \right)^{1/p}.
\end{equation*}\end{linenomath}
The superscript $\mathrm{K}$ indicates that the function depends on a Kantorovich-style optimization over measure couplings.

\begin{remark}
This function is well-defined, since one can show that the function $X \times Y \rightarrow \R$ defined by
\begin{linenomath}\begin{equation*}
(x,y) \mapsto d_{\mathrm{W},p}^\R(dh_\mX(x),dh_\mY(y))
\end{equation*}\end{linenomath}
is continuous, hence Borel measurable. This follows from the observation that the map $(X,d_X) \rightarrow (\mathcal{P}(\R),d_{\mathrm{W},p}^\R):x \mapsto dh_\mX(x)$ is continuous. Indeed, by triangle inequality, if $d_X(x,x') = \delta$ then 
\begin{linenomath}\begin{equation}\label{eqn:continuity_local_distributions}
h_\mX(x,r-\delta) \leq h_\mX(x',r) \leq h_\mX(x,r+\delta)
\end{equation}\end{linenomath}
for all $r \geq 0$. By \cite[Theorem 2.12]{bobkov2014one}, this implies
\begin{linenomath}\begin{equation*}
d_{\mathrm{W},\infty}^\R(dh_\mX(x),dh_\mX(x')) \leq \delta.
\end{equation*}\end{linenomath}
Since Wasserstein $\infty$-distance upper bounds Wasserstein $p$-distances for $p < \infty$, we conclude
\begin{linenomath}\begin{equation*}
d_{\mathrm{W},p}^\R(dh_\mX(x),dh_\mX(x')) \leq d_X(x,x').
\end{equation*}\end{linenomath}
\end{remark}

Once again using the special structure of optimal transport on the real line, one obtains the more explicit formula
\begin{linenomath}\begin{equation*}
L^{\mathrm{K}}_{\mathrm{h},p}(\mX,\mY) = \inf_{\mu \in \mU(\mu_X,\mu_Y)} \left(\int_{X \times Y} \left(\int_0^1 \left| h_\mX^{-1}(x,u) - h_\mY^{-1}(y,u)\right|^p \; du \right) \; \mu(dx \times dy) \right)^{1/p}.
\end{equation*}\end{linenomath}
The Kantorovich-style function $L^{\mathrm{K}}_{\mathrm{h},p}$ suggests an alternative Monge-style formulation over measure preserving maps, and we accordingly define
\begin{linenomath}\begin{equation*}
L^{\mathrm{M}}_{\mathrm{h},p}(\mX,\mY) := \inf_{\phi \in \mathcal{T}(\mu_X,\mu_Y)} \left( \int_X \left(\int_0^1 \left| h_\mX^{-1}(x,u) - h_\mY^{-1}(\phi(x),u)\right|^p \; du \right)  \; \mu_X(dx) \right)^{1/p}.
\end{equation*}\end{linenomath}

Similar to the global distribution setting, the $p=1$ versions of $L^{\mathrm{K}}_{\mathrm{h},p}$ and $L^{\mathrm{M}}_{\mathrm{h},p}$ can be simplified. Given $\mX,\mY\in\mM^m$, one defines the following \emph{local distribution cost function} $c_{\mX,\mY}:X\times Y\rightarrow \R_{\geq 0}$ by 
\begin{linenomath}
\begin{equation}\label{eqn:local_distribution_cost_function}
c_{\mX,\mY}(x,y) := \int_{0}^\infty \big|h_\mX(x,t)-h_\mY(y,t)\big|\,dt.
\end{equation}
\end{linenomath}
Following \cite{memoli2011gromov}, we can then show that
\begin{linenomath}\begin{equation*}
L^{\mathrm{K}}_{\mathrm{h}}(\mX,\mY) := L^{\mathrm{K}}_{\mathrm{h},1}(\mX,\mY) = \inf_{\mu \in \mU(\mu_X,\mu_Y)} \int_{X \times Y} c_{\mX,\mY}(x,y) \; \mu(dx \times dy)
\end{equation*}\end{linenomath}
and
\begin{linenomath}\begin{equation*}
L^{\mathrm{M}}_{\mathrm{h}}(\mX,\mY) := L^{\mathrm{M}}_{\mathrm{h},1}(\mX,\mY) = \inf_{\phi \in \mT(\mu_X,\mu_Y)} \int_X c_{\mX,\mY}(x,\phi(x)) \; \mu_X(dx).
\end{equation*}\end{linenomath}
Thus the $p=1$ functions compare mm-spaces by searching for a coupling or a transport map, respectively, which optimally preserve mm-space structure at the level of local volume growth. 

\begin{example}
The functions $L^{\mathrm{K}}_{\mathrm{h}}$ and $L^{\mathrm{M}}_{\mathrm{h}}$ have better distinguishing power than $L_{\mathrm{H}}$. For example, the spaces in $\mX$ and $\mY$ with $X = \{0,1,4,10,12,17\}$ and $Y = \{0,1,8,11,13,17\}$ from Example \ref{exmp:integer_sets}, which confound the pseudometric $L_\mathrm{H}$, are distinguished by both $L^{\mathrm{K}}_\mathrm{h}$ and $L^{\mathrm{M}}_\mathrm{h}$. This is easy to conclude from the observation that $h_\mX(0,r)$ is piecewise constant, with discontinuities at $r = 1,4,10,12,17$, and that this pattern is not replicated by  $h_\mY(y,r)$ for any $y \in Y$. 
\end{example}

It is straightfoward to show that  $L^{\mathrm{K}}_{\mathrm{h},p}$ (respectively, $L^{\mathrm{M}}_{\mathrm{h},p}$) satisfies the axioms of a metric (respectively, extended quasimetric) on $\mM^m$ except that it is nonvanishing for nonisomorphic spaces. Indeed,  the counterexample below shows that this is not the case in general.

\begin{example}
Example 5.6 from \cite{memoli2011gromov} shows that $L^{\mathrm{M}}_\mathrm{h}$ (and hence $L^{\mathrm{K}}_\mathrm{h}$) does not distinguish nonisomorphic mm-spaces in general. A variant of this example is shown below in Example \ref{exmp:non_injectivity_local_dist_c_trees}.
\end{example}

We therefore have the following.

\begin{proposition}
For $p \geq 1$, $L^{\mathrm{K}}_{\mathrm{h},p}$ defines a pseudometric and $L^{\mathrm{M}}_{\mathrm{h},p}$ defines an extended quasi-pseudometric on $\mM^m$. 
\end{proposition}

The $L$ in our notation $L_{\mathrm{H},p}, L^{\mathrm{K}}_{\mathrm{h},p}$ is chosen because these pseudometrics serve as lower bounds for Gromov-Wasserstein $p$-distances, as was shown in \cite{memoli2011gromov}. The following proposition refines the original lower bound results.

\begin{proposition}\label{prop:gw_lower_bound_hierarchy}
For $p \geq 1$ and mm-spaces $\mX$ and $\mY$,
\begin{linenomath}\begin{equation*}
L_{\mathrm{H},p}(\mX,\mY) \leq L^{\mathrm{K}}_{\mathrm{h},p}(\mX,\mY) \leq d_{\mathrm{GW},p}(\mX,\mY).
\end{equation*}\end{linenomath}
\end{proposition}

\begin{proof}
 The righthand inequality is implicit in Corollary 6.3 of \cite{memoli2011gromov}, so it remains to prove the lefthand inequality, which is novel---it was not obtained in the original paper \cite{memoli2011gromov}, or in the recent related paper \cite{chowdhury2019Stable}. Indeed, both of these papers deduce that $L_{\mathrm{H},p}$ and $L^{\mathrm{K}}_{\mathrm{h},p}$ give lower bounds on $d_{\mathrm{GW},p}$, but they do not give a relationship between the lower bounds themselves.
 
 For each $(x,y) \in X \times Y$, via \cite[Theorem 2.18]{villani2003topics}, choose $\gamma_{x,y} \in \mU(\mu_X,\mu_Y)$ realizing $d_{\mathrm{W},p}^\R(dh_\mX(x),dh_\mY(y))$ explicitly  as that probability measure on $\R^2$ with cumulative distribution function equal to $F_{\gamma_{x,y}}(s,t):=\min\{h_\mX(x,s),h_\mY(y,t)\}$ for each $(s,t)\in \R^2$. Also, using Lemma \ref{lem:realized_by_coupling}, choose $\mu \in \mU (\mu_X,\mu_Y)$ realizing 
\begin{linenomath}\begin{equation*}
\inf_{\mu \in \mU(\mu_X,\mu_Y)} \int_{X \times Y} \Big(d_{\mathrm{W},p}^\R(dh_\mX(x),dh_\mY(y))\Big)^p \mu(dx \times dy).
\end{equation*}\end{linenomath}
 Then we have 
\begin{linenomath}\begin{align}
    \int_{X \times Y} \big(d_{\mathrm{W},p}^\R(dh_\mX(x),dh_\mY(y))\big)^p \mu(dx \times dy) &= \int_{X \times Y} \int_{\R \times \R} |s-t|^p\, \gamma_{x,y}(ds \times dt)\, \mu(dx \times dy) \nonumber \\
    &= \int_{\R \times \R} |s-t|^p \; \overline{\mu}_F(ds \times dt), \label{eqn:hierarchy_claim}
\end{align}\end{linenomath}
by Fubini's Theorem, where $\overline{\mu}_F \in \mathcal{P}(\R \times \R)$ is defined as follows. We begin by defining a cumulative distribution $F:\R \times \R \rightarrow \R$ by 
\begin{linenomath}\begin{equation*}
F(s,t) := \int_{X \times Y} F_{\gamma_{x,y}}(s,t) \; \mu(dx \times dy).
\end{equation*}\end{linenomath}
Since the integrand is continuous, as follows from \eqref{eqn:continuity_local_distributions}, and $\mu$ is Borel, this function is well-defined. We then define $\overline{\mu}_F$ to be the probability measure associated to $F$. We claim that this is given by the formula 
\begin{linenomath}\begin{equation*}
\overline{\mu}_F(A) = \int_{X \times Y} \gamma_{x,y}(A) \, \mu (dx \times dy).
\end{equation*}\end{linenomath}
for measurable $A \subset \R \times \R$, hence \eqref{eqn:hierarchy_claim} holds. One can check this claim easily on rectangles of type $(s,s']\times (t,t']$ using the fact that the integrand of $F(s,t)$ is equal to the cumulative distribution $F_{\gamma_{x,y}}$ for $\gamma_{x,y}$  and the general formula holds by standard arguments.

Next we claim that $\overline{\mu}_F \in \mU(dH_\mX,dH_\mY)$. Indeed, for measurable $S \subset \R$, we have that $\overline{\mu}_F(S \times \R)$ is equal to
\begin{linenomath}\begin{equation*}
 \int_{X \times Y} \gamma_{x,y}(S \times \R) \mu(dx \times dy) = \int_{X \times Y} dh_\mX(x)(S)\, \mu(dx \times dy) = \int_X dh_\mX(x)(S) \,\mu_X(dx) = dH_\mX(S),
\end{equation*}\end{linenomath}
where we have applied marginal constraints for $\gamma_{x,y}$ and $\mu$. A similar computation shows that the other necessary marginal constraint for $\overline{\mu}_F$ is satisfied. We conclude that
\begin{linenomath}\begin{align*}
     L^{\mathrm{K}}_{\mathrm{h},p}(\mX,\mY)^p &= \int_{X \times Y} \big(d_{\mathrm{W},p}^\R(dh_\mX(x),dh_\mY(y))\big)^p \mu(dx \times dy) \\
     &= \int_{\R \times \R} |s -t |^p\, \overline{\mu}_F(ds \times dt) \geq d_{\mathrm{W},p}(dH_\mX,dH_\mY)^p = L_{\mathrm{H},p}(\mX,\mY)^p.
\end{align*}\end{linenomath}
\end{proof}

\subsubsection*{Subcategories of Metric Measure Spaces}

The main question we aim to address in this paper is the extent to which the pseudometrics defined above fail to be true metrics. Since the space $\mM^m$ is so large, this question is intractable in full generality. We therefore restrict our attention to various subcategories of mm-spaces. To make this precise, define a \emph{subcategory of metric measure spaces} to be a category $\catC$ whose objects $\mathrm{Obj}_\catC$ are isomorphism classes of mm-spaces from some fixed family and whose morphisms $\mathrm{Mor}_\catC$ are some appropriately restricted class of measure preserving maps, considered up to pre- and post-composition with isomorphisms. We will frequently abuse terminology and denote elements of $\mathrm{Obj}_\catC$ by representatives of isomorphism classes $\mX$ or $\mY$ and elements of $\mathrm{Mor}_\catC$ by representative measure-preserving maps $\phi:X \rightarrow Y$. For mm-spaces $\mX$ and $\mY$, let $\mathrm{Mor}_\catC(\mu_X,\mu_Y)$ denote the set of admissible measure-preserving morphisms from $X$ to $Y$. 

We provide below several examples of relevant categories $\catC$.

\begin{example}
The category $\catC$ of all mm-spaces with $\mathrm{Mor}_\catC$ containing \emph{continuous} measure-preserving maps.
\end{example}

\begin{example}
The category $\catC$ of all mm-spaces whose morphisms $\mathrm{Mor}_\catC$ are \emph{bijective} measure-preserving mappings. In this case, the set of morphisms between $\mX$ and $\mY$ is empty if $X$ and $Y$ have different cardinality.
\end{example}

Throughout the paper, we use the term \emph{full subcategory} to denote a subcategory with no extra restrictions placed on its morphisms.

\begin{example}
The full subcategory $\catC$ whose objects $\mathrm{Obj}_\catC$ are \emph{finite} mm-spaces, perhaps with a fixed cardinality. For a fixed metric space $(X,d_X)$, one could similarly consider the full category of finite subsets of $X$ endowed with the restricted distance $d_X$ and uniform measure.
\end{example}

\begin{example}
The category $\catC$ whose objects $\mathrm{Obj}_\catC$ are compact Riemannian manifolds, perhaps with fixed dimension or diffeotype, endowed with geodesic distance and normalized Riemannian volume, and whose morphisms $\mathrm{Mor}_\catC$ are measure-preserving maps with fixed regularity $C^0, C^1, C^\infty$, et cetera.
\end{example}

A motivation for considering the Monge-style pseudometric $L^{\mathrm{M}}_{\mathrm{h},p}$ (as opposed to only considering the Kantorvich-style $L^{\mathrm{K}}_{\mathrm{h},p}$) is that it naturally restricts to a given subcategory $\catC$ of mm-spaces. Indeed, define
\begin{linenomath}\begin{equation*}
L^\catC_{\mathrm{h},p}: \mathrm{Obj}_\catC \times \mathrm{Obj}_\catC \rightarrow \R
\end{equation*}\end{linenomath}
by 
\begin{linenomath}\begin{equation*}
L^{\catC}_{\mathrm{h},p}(\mX,\mY) := \inf_{\phi \in \mathrm{Mor}_\catC(\mu_X,\mu_Y)} \left(\int_X \int_0^1 \left| h_\mX^{-1}(x,u) - h_\mY^{-1}(\phi(x),u)\right|^p \; du  \; \mu_X(dx) \right)^{1/p}.
\end{equation*}\end{linenomath}
In the $p=1$ case, this reduces to 
\begin{linenomath}\begin{equation*}
L^\catC_\mathrm{h}(\mX,\mY) := L^\catC_{\mathrm{h},1}(\mX,\mY) =  \inf_{\phi \in \mathrm{Mor}_\catC(\mu_X,\mu_Y)} \int_X c_{\mX,\mY}(x,\phi(x)) \; \mu_X(dx).
\end{equation*}\end{linenomath}
We similarly define $L_{\mathrm{H},p}^\catC$ to be the restriction of $L_{\mathrm{H},p}$ to $\mathrm{Obj}_\catC \times \mathrm{Obj}_\catC$ and use the notation $L_\mathrm{H}^\catC : = L_{\mathrm{H},1}^\catC$. Note that, since $L_{\mathrm{H},p}$ does not depend on any registration between mm-spaces, the formulas for $L_{\mathrm{H},p}^\catC$ and $L_{\mathrm{H},p}$ are exactly the same---we only use the notation $L^\catC_{\mathrm{H},p}$ to emphasize restriction to the subcategory and for the sake of symmetry with the notation $L_{\mathrm{h},p}^\catC$. 

\subsection{Inverse Problems for Metric Measure Spaces}\label{sec:inverse_problems}

We now have several pseudometrics defined on $\mM^m$, or on subcategories thereof. Moreover, these pseudometrics are computationally feasible for finite (or finite approximations of) mm-spaces: $L_{\mathrm{H},p}$ can be expressed as a straightforward comparison between distance histogram vectors,  $L^{\mathrm{K}}_{\mathrm{h},p}$ can be computed via optimization of a linear function over a convex domain and $L^{\mathrm{M}}_{\mathrm{h},p}$ can be computed via the Hungarian algorithm. Given a subcategory of mm-spaces $\catC$, we wish to determine the extent to which the pseudometrics $L_{\mathrm{H},p}^\catC$ and $L_{\mathrm{h},p}^\catC$ fail to be true metrics. Since the formulas for these pseudometrics have simpler, more interpretable forms when $p=1$, we will generally restrict our attention to $L_\mathrm{H}^\catC = L_{\mathrm{H},1}^\catC$ and $L_\mathrm{h}^\catC = L_{\mathrm{h},1}^\catC$. The main problem we consider in this paper is:

\medskip
\noindent \textbf{Main Problem.} For a fixed subcategory of mm-spaces $\catC$, give conditions which guarantee that mm-spaces $\mX,\mY \in \mathrm{Obj}_\catC$ satisfy $L_\mathrm{H}^\catC(\mX,\mY) = 0 \Leftrightarrow \mX \approx \mY$ or $L_\mathrm{h}^\catC(\mX,\mY) = 0 \Leftrightarrow \mX \approx \mY$. 
\medskip

We specifically aim to address the following inverse problems, which can be viewed as refinements of the Main Problem, for various subcategories of mm-spaces $\catC$:

\begin{enumerate}
\item[(\globalInj)] \textbf{Global Injectivity.} For $\mX,\mY \in \mathrm{Obj}_\catC$, does $L_\mathrm{h}^\catC(\mX,\mY) = 0$ imply $\mX \approx \mY$? More strongly, does $L_\mathrm{H}^\catC(\mX,\mY) = 0$ imply $\mX \approx \mY$?
\end{enumerate}

\begin{enumerate}
\item[(\homotopy)] \textbf{Homotopy Type Characterization.} For $\mX,\mY \in \mathrm{Obj}_\catC$, does $L_\mathrm{h}^\catC(\mX,\mY) = 0$ (or, more strongly, $L_\mathrm{H}^\catC(\mX,\mY) = 0$) imply that the metric spaces $(X,d_X)$ and $(Y,d_Y)$ are homotopy equivalent?

\item[(\spheres)] \textbf{Sphere Rigidity.} Suppose that $\catC$ includes an object which can be characterized as a unit sphere (for example, the category of Riemannian manifolds of fixed dimension), denoted $\mathcal{S}$. Does $L_\mathrm{h}^\catC(\mX,\mathcal{S}) = 0$ (or, more strongly, $L_\mathrm{H}^\catC(\mX,\mathcal{S}) = 0$) imply $\mX \approx \mathcal{S}$?

\item[(\local)] \textbf{Local Injectivity.} For $\mX \in \mathrm{Obj}_\catC$, does there exist $\epsilon_\mX > 0$ such that for all $\mY \in \mathrm{Obj}_\catC$ with $d_\mathrm{GH}(X,Y) < \epsilon_\mX$, $L_\mathrm{h}^\catC(\mX,\mY) = 0$ (or, more strongly, $L_\mathrm{H}^\catC(\mX,\mY) = 0$) implies $\mX \approx \mY$?
\end{enumerate}

These problems, and several related problems, are treated for embedded and abstract Riemannian manifolds in Section \ref{sec:plane_curves} and for metric graphs in Section \ref{sec:metric_trees}. For many results throughout the paper, we indicate which of these problems (\globalInj)--(\local) to which the result is relevant. Before considering these questions, we provide more context for the pseudometrics $L^\catC_{\mathrm{h},p}$ in the following subsection by showing that they lower bound a Monge-style variant of Gromov-Wasserstein distance.

\subsection{Gromov-Monge Quasi-Metrics}\label{sec:gromov_monge}

To contextualize the category-restricted pseudometrics $L_{\mathrm{H},p}^\catC$ and $L_{\mathrm{h},p}^\catC$ as lower bounds in their own right, we introduce a variant of Gromov-Wasserstein distance which restricts the constraint set to only consider measure-preserving mappings. The \emph{Gromov-Monge $p$-distance} between mm-spaces $\mX$ and $\mY$ is the quantity
\begin{linenomath}\begin{align*}
\dgm(\mX,\mY) &:= \inf_{\phi \in \mT(\mu_X,\mu_Y)} \|\Gamma_{X,Y}\|_{L^p(\mu_\phi \otimes \mu_\phi)} \qquad \mbox{($\mu_\phi$ as in \eqref{eqn:pushforward_coupling})}\\
&=\inf_{\phi \in \mT(\mu_X,\mu_Y)} \left(\iint_{X \times X} \big|d_X(x,x') - d_Y(\phi(x),\phi(x'))\big|^p \mu_X \otimes \mu_X (dx \times dx') \right)^{1/p},
\end{align*}\end{linenomath}
with the understanding that if $\mT(\mu_X,\mu_Y)=\emptyset$ then $\dgm (\mX,\mY):=\infty$. 

\begin{remark}\label{rmk:gromov_wasserstein_bound}
Clearly, the bound $\dgw(\mX,\mY) \leq \dgm(\mX,\mY)$ always holds.
\end{remark}

Beyond their relationship to the pseudometrics $L_{\mathrm{H},p}^\catC$ and $L_{\mathrm{h},p}^\catC$, elucidated below, Gromov-Monge distances may be of independent interest. The aim of many applications in imaging and shape classification is to obtain a \emph{registration} of two objects via a mapping from one object to another. While Gromov-Wasserstein distance defines a metric on $\mM^m$, it is generally realized by a measure coupling which is not of the form \eqref{eqn:pushforward_coupling} and the vital registration map between spaces is not obtained. Beyond the motivation of preferring a matching to a coupling in applications, the Gromov-Monge distance formulation has several other benefits over the Gromov-Wasserstein formulation. For finite spaces $X$ and $Y$, say of cardinality $n$, the computational complexity of evaluating the objective function for Gromov-Monge is lower than that of Gromov-Wasserstein---reducing the complexity from $O(n^3\log(n))$ to $O(n^2)$ in the $p=2$ case  \cite{vayer2019sliced}. At a theoretical level, it was observed in \cite{sturm2012space} that the geodesics (in the sense of metric geometry) of $\mM^m$ with respect to $d_{\mathrm{GW},2}$ have a simpler structure when optimal couplings are realized by measure-preserving maps; i.e., when the distances $d_{\mathrm{GW},2}$ and $d_{\mathrm{GM},2}$ coincide. This leads to the following open question.

\begin{question}\label{ques:optimal_transport_maps}
 For which classes of mm-spaces is it possible to guarantee that $d_{\mathrm{GW},2} = d_{\mathrm{GM},2}$?
\end{question}

\begin{remark}\label{rmk:optimal_transport_maps}
We remark here on progress on and further motivation for this question. The question of realizing an optimal coupling through an optimal mapping was already raised by Sturm in \cite[Challenge 3.6]{sturm2012space}, where it is specified to smooth manifolds. In this setting, an answer to the problem would be a quadratic version of famous results of Brenier \cite{brenier1991polar} and McCann \cite{mccann2001polar} in the classical optimal transport setting. Sturm is able to show that optimal Gromov-Wasserstein couplings between rotationally symmetric distributions in a  Euclidean space are always realized by a transport map which is  unique up to composition with an isometry. 

This question was recently tackled for discrete spaces in \cite{vayer2019sliced}, with a view toward applications in data science. In this work the authors consider the  special category of uniformly weighted $n$-point configurations on the real line. Even in this simple situation, the proof that Gromov-Monge and Gromov-Wasserstein coincide is nontrivial. A proof for the general version of the discrete problem, as stated in \cite[Challenge 5.27]{sturm2012space}, is still open.

It was recently observed empirically in \cite{chowdhury2019gromov} that numerically approximated optimal couplings between discrete spaces tend to at least be relatively sparse, which can be leveraged to implement the geodesic formula in Gromov-Wasserstein space described by Sturm in \cite{sturm2012space}. A theoretical sparsity bound was derived in \cite{chowdhury2020generalized} in a slightly different setting, where heat kernel representations of graphs,  rather than mm-spaces, are compared. A better theoretical understanding of the optimal coupling landscape would guide principled improvements to computation of Gromov-Wasserstein geodesics and barycenters in applications.
\end{remark}

The following basic result clarifies the sense in which $\dgm$ is a ``distance".

\begin{theorem}\label{thm:dgm_is_metric}
For any $p\geq 1$ the function $\dgm$ defines an extended quasi-metric on $\mM^m$.
\end{theorem}

\begin{proof}
Let $\mX,\mY$ and $\mZ$ be mm-spaces. From the definition of $\dgm $, Remark \ref{rmk:gromov_wasserstein_bound} and the fact that $\dgw$ is a metric on $\mM^m$, we easily see that $\dgm(\mX,\mY) \geq 0$, with $\dgm (\mX,\mY) = 0$ if and only if $\mX \approx \mY$. 

It only remains to show that $\dgm $ satisfies the triangle inequality $\dgm (\mX,\mZ) \leq \dgm (\mX,\mY) + \dgm (\mY,\mZ)$. If $\dgm(\mX,\mZ) = \infty$, then $\mT(\mu_X,\mu_Z) = \emptyset$ and it follows that either $\mT(\mu_X, \mu_Y) = \emptyset$ or $\mT(\mu_Y, \mu_Z) = \emptyset$, whence one of $\dgm(\mX,\mY)$ or $\dgm(\mY,\mZ)$ is infinity. The triangle inequality therefore holds in this case. The triangle inequality follows trivially if $\dgm(\mX,\mY)$ or $\dgm(\mY,\mZ)$ is infinite, so let us assume that all distances are finite. In this case we have
\begin{linenomath}\begin{align}
\dgm (\mX,\mZ) &= \inf_{\phi \in \mathcal{T}(\mu_X,\mu_Y)} \|\Gamma_{\mX,\mZ}\|_{L^p(\mu_\phi \otimes \mu_\phi)} \leq \inf_{\substack{\phi_1 \in \mathcal{T}(\mu_X,\mu_Y)\\ \phi_2 \in \mathcal{T}(\mu_Y,\mu_Z)}} \{ \|\Gamma_{\mX,\mZ}\|_{L^p(\mu_\phi \otimes \mu_\phi)} \mid \phi = \phi_2 \circ \phi_1\} \nonumber \\
&\leq \inf_{\substack{\phi_1 \in \mathcal{T}(\mu_X,\mu_Y)\\ \phi_2 \in \mathcal{T}(\mu_Y,\mu_Z)}} \left(\|\Gamma_{\mX,\mY}\|_{L^p(\mu_{\phi_1} \otimes \mu_{\phi_1})} + \|\Gamma_{\mY,\mZ}\|_{L^p(\mu_{\phi_2} \otimes \mu_{\phi_2})} \right) \label{eqn:dist_metric_1} \\
&= \dgm(\mX,\mY) + \dgm(\mY,\mZ) \nonumber.
\end{align}\end{linenomath}
The estimate (\ref{eqn:dist_metric_1})  follows from the fact that 
\begin{linenomath}\begin{equation*}
\|\Gamma_{\mX,\mZ}\|_{L^p(\mu_\phi \otimes \mu_\phi)} \leq \|\Gamma_{\mX,\mY}\|_{L^p(\mu_{\phi_1} \otimes \mu_{\phi_1})} + \|\Gamma_{\mY,\mZ}\|_{L^p(\mu_{\phi_2} \otimes \mu_{\phi_2})} 
\end{equation*}\end{linenomath} 
for any fixed $\phi=\phi_2 \circ \phi_1$ with $\phi_1 \in \mathcal{T}(\mu_X,\mu_Y)$ and $\phi_2 \in \mathcal{T}(\mu_Y,\mu_Z)$. This holds by the definition of $\mu_\phi$, the Minkowski inequality and the fact that 
\begin{linenomath}\begin{equation*}
\Gamma_{\mX,\mZ}(x,z,x',z') \leq \Gamma_{\mX,\mY}(x,y,x',y') + \Gamma_{\mY,\mZ}(y,z,y',z')
\end{equation*}\end{linenomath}
for all $x,x',y,y',z,z'$.
\end{proof}

\begin{example}[$\dgm$ is Not Symmetric] \label{exmp:asymmetry}
Let $X=\{u,v\}$ be endowed with empirical measure $\delta_u = \delta_v = \frac{1}{2}$ and  metric determined by $d_X(u,v)=1$. Let $Y=\{y\}$ with measure $\delta_y = 1$. Then the constant map $X \rightarrow Y$ is measure-preserving, but neither of the two possible maps $Y \rightarrow X$ preserves measure (we cannot split mass). Thus $\dgm(\mX,\mY)=1/2^{1/p}$ while $\dgm(\mX,\mY) =\infty$
\end{example}

\begin{remark}
Additional theoretical properties of Gromov-Monge distances will be explored in a forthcoming supplementary paper \cite{supplementary}. For now, we are mainly interested in Gromov-Monge distances because they contextualize the category-specific inverse problems of Section \ref{sec:inverse_problems}, as we explain below.
\end{remark}

\subsubsection*{Restricting to Subcategories of Metric Measure Spaces}

A benefit of using mappings rather than couplings in the definition of the Gromov-Monge $p$-distances is that it makes the definition amenable to restricting to various convenient subclasses of spaces and maps. For a given subcategory $\catC$ of mm-spaces, let $d_\mathrm{GM,p}^\catC:\mathrm{Obj}_\catC \times \mathrm{Obj}_\catC \rightarrow \R$ denote the function
\begin{linenomath}\begin{equation*}
\dgm^\catC(\mX,\mY) := \inf_{\phi \in \mathrm{Mor}_\catC(\mu_X,\mu_Y)} \left(\iint_{X \times X} \big|d_X(x,x') - d_Y(\phi(x),\phi(x'))\big|^p \mu_X \otimes \mu_X (dx \times dx') \right)^{1/p}.
\end{equation*}\end{linenomath}

The following gives a refinement of Proposition \ref{prop:gw_lower_bound_hierarchy} to the setting of subcategories of mm-spaces. 

\begin{proposition}\label{prop:gw_lower_bound_hierarchy_GM}
For $p \geq 1$ and mm-spaces $\mX$ and $\mY$,
\begin{linenomath}\begin{equation*}
L_{\mathrm{H},p}^\catC(\mX,\mY) \leq L_{\mathrm{h},p}^\catC(\mX,\mY) \leq d_{\mathrm{GM},p}^\catC(\mX,\mY).
\end{equation*}\end{linenomath}
\end{proposition}

\begin{proof}
The righthand inequality follows by adapting the proof of \cite[Corollary 6.3]{memoli2011gromov}, replacing couplings by transport maps as necessary. The lefthand inequality follows from Proposition \ref{prop:gw_lower_bound_hierarchy} and the observations that $L_{\mathrm{h},p}^{\mathrm{K}}(\mX,\mY) \leq L_{\mathrm{h},p}^\catC(\mX,\mY)$ and $L_{\mathrm{H},p}^\catC(\mX,\mY) = L_{\mathrm{H},p}(\mX,\mY)$.
\end{proof}

\section{Characterization Results for  Manifolds}\label{sec:plane_curves}

We now return to the main goal of the paper, which is to address the inverse problems (\globalInj)--(\local) laid out in Section \ref{sec:inverse_problems}, beginning with embedded manifolds and later moving on to abstract manifolds and metric graphs.

We begin by reviewing a result on arguably the simplest subclass of mm-spaces: point clouds in Euclidean space. Let $\catC_{N,k}$ denote the full subcategory of mm-spaces whose objects are isomorphism classes of mm-spaces $\mathcal{X}=(X,d_X,\mu_X)$ such that $X$ is a set of $N$ points in $\R^k$, $d_X$ is Euclidean distance and $\mu_X$ is uniform measure. The problem of reconstructing $\mX$ from its collection of interpoint distances (i.e., from its distance distribution $H_\mX$) is classical and has applications to DNA sequencing and X-ray crystallography \cite{lemke2003reconstructing}.

It is well known that $L_\mathrm{H}^{\catC_{N,k}}$ does not distinguish elements of $\catC_{N,k}$ for any $(N,k)$ with $N > 3$. Example \ref{exmp:integer_sets} gave an example of nonisomorphic mm-spaces $\mX,\mY$ in $\catC_{1,6}$ with $L_\mathrm{H}^{\catC_{1,6}}(\mX,\mY) = 0$. Counterexamples in higher dimensions were given by Boutin and Kemper \cite[Section 1.1]{boutin2004reconstructing}, who refined the question and proved the following theorem (which we state using our terminology). 

\begin{theorem}[\cite{boutin2004reconstructing}]\label{thm:boutin_kemper}
The pseudometric $L_\mathrm{H}^{\catC_{N,k}}$ is locally injective on $\catC_{N,k}$. That is, for every point cloud $\mX$, there exists $\epsilon_\mX > 0$ such that if a point cloud $\mY$ is $\epsilon_\mX$-close to $\mX$ in Hausdorff distance and $L_\mathrm{H}^{\catC_{N,k}}(\mX,\mY)=0$ then $\mX$ and $\mY$ differ by a rigid motion.
\end{theorem}

This result gives an affirmative answer for the category of Euclidean point clouds to the local injectivity question (\local) posed in Section \ref{sec:inverse_problems}. We remark that (\local) is formulated using Gromov-Hausdorff distance, while Theorem \ref{thm:boutin_kemper} uses Hausdorff distance. This can be reconciled via \cite[Theorem 2]{memoli2008gromov}, which precisely relates an isometry-invariant variant of Hausdorff distance to Gromov-Hausdorff distance between Euclidean point clouds.

\subsection{Plane Curves and the Curve Histogram Conjecture}

We now move on to study the distinguishing power of distance distribution invariants in subcategories of manifolds. We begin by studying 1-dimensional manifolds embedded in the plane.

\subsubsection*{The Curve Histogram Conjecture}

In \cite{brinkman2012invariant}, Brinkman and Olver consider the category of mm-spaces $\mX=(X,d_X,\mu_X)$ with $X$ a plane curve, $d_X$ extrinsic Euclidean distance between points, and $\mu_X$ normalized arclength measure along the curve. They study global distance distributions (which they refer to as \emph{distance histograms}) for elements of this class, where they obtain convergence results for discrete approximations of plane curves and pose the following conjecture.

\begin{conjecture}[The Curve Histogram Conjecture \cite{brinkman2012invariant}]\label{con:curve_histogram}
The distance histogram $H_\mX$ determines a regular plane curve up to isometry.
\end{conjecture}

In \cite{brinkman2012invariant}, the class of \emph{regular} plane curves includes both polygons and smooth curves---see \cite{brinkman2012invariant} for details. Using the terminology of this paper, the Curve Histogram Conjecture says that $L_\mathrm{H}^\catC$ induces a metric on the full subcategory $\catC$ of isomorphism classes of regular plane curves, which would provide an affirmative answer to question (\globalInj) from Section \ref{sec:inverse_problems}.

Consider the curves shown in Figure \ref{fig:curve_histogram_conjecture}. Each curve is obtained as the boundary of a polygonal region, each constructed by appending four congruent isosceles triangles to a regular octagon. These curves are clearly non-isometric and will serve as our counterexample to the Curve Histogram Conjecture. 

\begin{remark}
These curves were first used by Mallows and Clark in \cite{mallows1970linear} as a counterexample to a conjecture of Blaschke that the distribution of lengths of chords intersecting a convex polygon determines the polygon up to isometry (see \cite{santalo1953introduction})---note that one can make the curves in Figure \ref{fig:curve_histogram_conjecture} convex by taking the heights of the appended triangles to be sufficiently small. The distribution of chord lengths for a convex body is determined by its distribution of distances \cite{piefke1978beziehungen, vlasov2007signed}. In \cite{garcia2016pairs}, Garc\'{i}a-Pelayo shows that the convex bodies bounded by the polygons in Figure \ref{fig:curve_histogram_conjecture} have the same distance distributions, with a view toward constructing other counterexamples to the Blaschke conjecture. It may be that the distance distributions of the boundary curves can be obtained from those of the convex bodies via a tube formula as in \cite{gray2003tubes}, but we were unable to work out a precise connection. We instead give a direct proof using a general procedure for comparing distance distributions of arbitrary mm-spaces.
\end{remark}

\begin{figure}
\begin{center}
\begin{tikzpicture}[scale=1.5, transform shape]
\draw[gray!10, thick] (1,0) -- (.707,.707);
\draw[black, thick] (1,0) -- (0.924*1.4,0.383*1.4);
\draw[black, thick] (0.924*1.4,0.383*1.4) -- (.707,.707);
\draw[black, thick, dotted] (0.924*0.93,0.383*0.93) -- (0.924*1.4,0.383*1.4) node[above] {\tiny $T_3$};
\draw[black, thick] (.707,.707) -- (0,1) node[midway, label={[xshift=0.1cm, yshift=-0.2cm] \tiny $S_1$}] {};
\draw[black, thick] (0,1) -- (-.707,.707) node[midway, label={[xshift=-0.2cm, yshift=-0.21cm] \tiny $S_2$}] {};
\draw[black, thick] (-.707,.707) -- (-1,0) node[midway, label={[xshift=-0.22cm, yshift=-0.27cm] \tiny $S_3$}] {};
\draw[gray!10, thick] (-1,0) -- (-.707,-.707);
\draw[black, thick] (-1,0) -- (-0.924*1.4,-0.383*1.4);
\draw[black, thick] (-0.924*1.4,-0.383*1.4) -- (-.707,-.707);
\draw[black, thick, dotted] (-0.924*0.93,-0.383*0.93) -- (-0.924*1.4,-0.383*1.4) node[below] {\tiny $T_4$};
\draw[gray!10, thick] (-.707,-.707) -- (0,-1);
\draw[black, thick] (0,-1) -- (-0.383*1.4,-0.924*1.4);
\draw[black, thick] (-0.383*1.4,-0.924*1.4) -- (-.707,-.707);
\draw[black, thick, dotted] (-0.383*0.93,-0.924*0.93) -- (-0.383*1.4,-0.924*1.4) node[below] {\tiny $T_1$};
\draw[black, thick] (0,-1) -- (.707,-.707) node[midway, label={[xshift=0.1cm, yshift=-0.6cm] \tiny $S_4$}] {};
\draw[gray!10, thick] (.707,-.707) -- (1,0);
\draw[black, thick] (1,0) -- (0.924*1.4,-0.383*1.4);
\draw[black, thick] (0.924*1.4,-0.383*1.4) -- (.707,-.707);
\draw[black, thick, dotted] (0.924*0.93,-0.383*0.93) -- (0.924*1.4,-0.383*1.4) node[below] {\tiny $T_2$};
\end{tikzpicture} \hspace{0.2in} \begin{tikzpicture}[scale=1.5, transform shape]
\draw[gray!10, thick] (1,0) -- (.707,.707);
\draw[black, thick] (1,0) -- (0.924*1.4,0.383*1.4);
\draw[black, thick] (0.924*1.4,0.383*1.4) -- (.707,.707);
\draw[black, thick, dotted] (0.924*0.93,0.383*0.93) -- (0.924*1.4,0.383*1.4) node[above] {\tiny $T_3'$};
\draw[gray!10, thick] (.707,.707) -- (0,1);
\draw[black, thick] (0,1) -- (0.383*1.4,0.924*1.4);
\draw[black, thick] (0.383*1.4,0.924*1.4) -- (.707,.707);
\draw[black, thick, dotted] (0.383*0.93,0.924*0.93) -- (0.383*1.4,0.924*1.4) node[above] {\tiny $T_1'$};
\draw[black, thick] (0,1) -- (-.707,.707) node[midway, label={[xshift=-0.2cm, yshift=-0.21cm] \tiny $S_2'$}] {};
\draw[black, thick] (-.707,.707) -- (-1,0) node[midway, label={[xshift=-0.22cm, yshift=-0.27cm] \tiny $S_3'$}] {};
\draw[gray!10, thick] (-1,0) -- (-.707,-.707);
\draw[black, thick] (-1,0) -- (-0.924*1.4,-0.383*1.4);
\draw[black, thick] (-0.924*1.4,-0.383*1.4) -- (-.707,-.707);
\draw[black, thick, dotted] (-0.924*0.93,-0.383*0.93) -- (-0.924*1.4,-0.383*1.4) node[below] {\tiny $T_4'$};
\draw[black, thick] (-.707,-.707) -- (0,-1) node[midway, label={[xshift=-0.1cm, yshift=-0.6cm] \tiny $S_1'$}] {};
\draw[black, thick] (0,-1) -- (.707,-.707) node[midway, label={[xshift=0.1cm, yshift=-0.6cm] \tiny $S_4'$}] {};
\draw[gray!10, thick] (.707,-.707) -- (1,0);
\draw[black, thick] (1,0) -- (0.924*1.4,-0.383*1.4);
\draw[black, thick] (0.924*1.4,-0.383*1.4) -- (.707,-.707);
\draw[black, thick, dotted] (0.924*0.93,-0.383*0.93) -- (0.924*1.4,-0.383*1.4) node[below] {\tiny $T_2'$};
\end{tikzpicture}
\end{center}
\caption{A counterexample to Conjecture \ref{con:curve_histogram}.}\label{fig:curve_histogram_conjecture}
\end{figure}

To prove that the curves give a counterexample to the Curve Histogram Conjecture, we will derive a simple sufficient condition for mm-spaces to have the same global distance distribution. For a mm-space $\mathcal{X}$, let $X=X_1 \cup X_2 \cup \cdots \cup X_n$ be a partition into measurable subsets. For each $i,j\in\{1,\ldots,n\}$  let $H_\mX^{i,j}$ be the function defined by
\begin{linenomath}\begin{equation*}
H_\mX^{i,j}(r) = \mu_X \otimes \mu_X \left(\{(x,x') \in X_i \times X_j \mid d_X(x,x') \leq r \}\right).
\end{equation*}\end{linenomath}
Denoting by $\mathbb{1}_r:X \times X \rightarrow \R$ the indicator function for the set $\{(x,x') \in X \times X \mid d_X(x,x') \leq r\}$,
we have
\begin{linenomath}\begin{align*}
H_{\mX}(r) &= \int_{X \times X} \mathbb{1}_r(x,x') \mu_X \otimes \mu_X (dx \times dx') \\
&= \sum_{i,j=1}^n \int_{X_i \times X_j} \mathbb{1}_r|_{X_i \times X_j}(x,x') \mu_X \otimes \mu_X (dx \times dx') = \sum_{i,j=1}^n H_{\mathcal{X}}^{i,j}(r).
\end{align*}\end{linenomath}
The next theorem then follows immediately.

\begin{theorem}\label{thm:main_theorem}
Let $\mX$ and $\mathcal{Y}$ be mm-spaces. If there exist measurable partitions $X=X_1 \cup X_2 \cup \cdots \cup X_n$ and $Y=Y_1 \cup Y_2 \cup \cdots \cup Y_n$ such that for each pair $(X_i,X_j)$, there is a pair $(Y_k,Y_\ell)$ with $H_\mX^{i,j} = H_\mY^{k,\ell}$, then $H_\mX = H_\mY$. 
\end{theorem} 

\begin{proposition}[Noninjectivity of Global Distance Distributions for Plane Curves (\globalInj)]\label{prop:curve_histogram_conjecture_false}
The curves in Figure \ref{fig:curve_histogram_conjecture} are not isomorphic, but satisfy the hypotheses of Theorem \ref{thm:main_theorem} and therefore have the same distance distribution. 
\end{proposition}

\begin{proof}
Denote the curve on the left by $X$ and the curve on the right by $Y$. We partition each curve into triangular pieces $T_i$ (respectively, $T_i'$) and straight pieces $S_j$ (respectively, $S_j'$) according to the figure. For each pair $(X_i,X_j)$ of pieces $X$ (here $X_j$ stands as a placeholder for either an $S_j$ or a $T_j$), we find a corresponding pair $(Y_k,Y_\ell)$ in $Y$ so that there is a rigid motion taking the first pair onto the second. Since the functions $H_\mathcal{X}^{i,j}$ and $H_\mY^{k,\ell}$ are invariant under rigid isometries, the hypotheses of Theorem \ref{thm:main_theorem} are then satisfied. 

Pairs of the form $(S_i,S_i)$ and $(T_i,T_i)$ are matched with $(S_i',S_i')$ and $(T_i',T_i')$, respectively. Moreover, note that curve $Y$ is obtained form curve $X$ by ``swapping" $S_1$ with $T_1$, it is clear that $(S_1,T_1)$ should be matched with $(S_1',T_1')$. It also follows that we have obvious matchings $(S_i,S_j) \leftrightarrow (S_i',S_j')$, $(T_i,T_j) \leftrightarrow (T_i',T_j')$ and $(S_i,T_j) \leftrightarrow (S_i',T_j')$ for $i,j \neq 1$.  It remains to show that the desired  matchings exist for pairs which include $S_1$ or $T_1$. These are given by

\begin{center}
\begin{tabular}{ c c c c }
$(S_1,S_2) \leftrightarrow (S_1',S_4')$ & $(S_1,T_2)\leftrightarrow (S_1',T_2')$ & $(T_1,T_2) \leftrightarrow (T_1',T_2')$ & $(T_1,S_2) \leftrightarrow (T_1',S_4')$\\ 
$(S_1,S_3) \leftrightarrow (S_1',S_3')$ & $(S_1,T_3)\leftrightarrow (S_1',T_4')$ & $(T_1,T_3) \leftrightarrow (T_1',T_4')$ &$(T_1,S_3) \leftrightarrow (T_1',S_3')$ \\  
$(S_1,S_4)\leftrightarrow (S_1',S_2')$ &$(S_1,T_4)\leftrightarrow (S_1',T_3')$ & $(T_1,T_4) \leftrightarrow (T_1',T_3')$ &$(T_1,S_4) \leftrightarrow (T_1',S_2').$   \\
\end{tabular}
\end{center}
\end{proof}

\begin{remark}
Consider the point $y$ at the tip of pyramid $T_1'$ in curve $Y$ on the right hand side of Figure \ref{fig:curve_histogram_conjecture}. It is easy to see that the local distance distribution $h_\mathcal{Y}(y,\cdot)$ is different from the local distribution $h_\mathcal{X}(x,\cdot)$ of any point $x$ in the curve $X$ on the left hand side of Figure \ref{fig:curve_histogram_conjecture}. This can be used to show that these curves \emph{are} distinguished by the lower bound $L^{\mathrm{K}}_{\mathrm{h}}$. We show below in Proposition \ref{prop:local_distribution_curves} that, in an appropriately restricted category, plane curves are always distinguished by local distributions. 
\end{remark}

Pushing this construction further, we now give an example demonstrating a lack of distinguishing power of $L_\mathrm{H}$ locally (cf.\ the main result of \cite{bauer2013local} demonstrating injectivity in a neighborhood of the unit circle for the \emph{circular integral invariant} for plane curves).

\begin{example}\label{exmp:pessimistic}
For any $\epsilon > 0$, we can construct a pair of simple closed planar curves $\mX_1$ and $\mX_2$ such that the Hausdorff distance between the unit circle $\mathbb{S}^1$ and each $X_j$ is less than $\epsilon$, $X_1$ is not isometric to $X_2$, and $H_{\mX_1} = H_{\mX_2}$.  Begin by circumscribing $\mathbb{S}^1$ by a regular octagon. Partition the points of tangency into two groups according to the edge partitions used in Proposition \ref{prop:curve_histogram_conjecture_false}. Construct one curve $X_1$ by adding a small, symmetric bump at each point in one of the partition groups, and construct $X_2$ by adding bumps to the complementary points of tangency (see Figure \ref{fig:pessimistic}). It follows from Theorem \ref{thm:main_theorem} and the proof of Proposition \ref{prop:curve_histogram_conjecture_false} that $H_{\mX_1} = H_{\mX_2}$ and taking the bumps to be sufficiently small yields Hausdorff distance from $\mathbb{S}^1$ to $X_j$ less than $\epsilon$.
\end{example}

\begin{figure}
\begin{center}
\includegraphics[width = 0.4\textwidth]{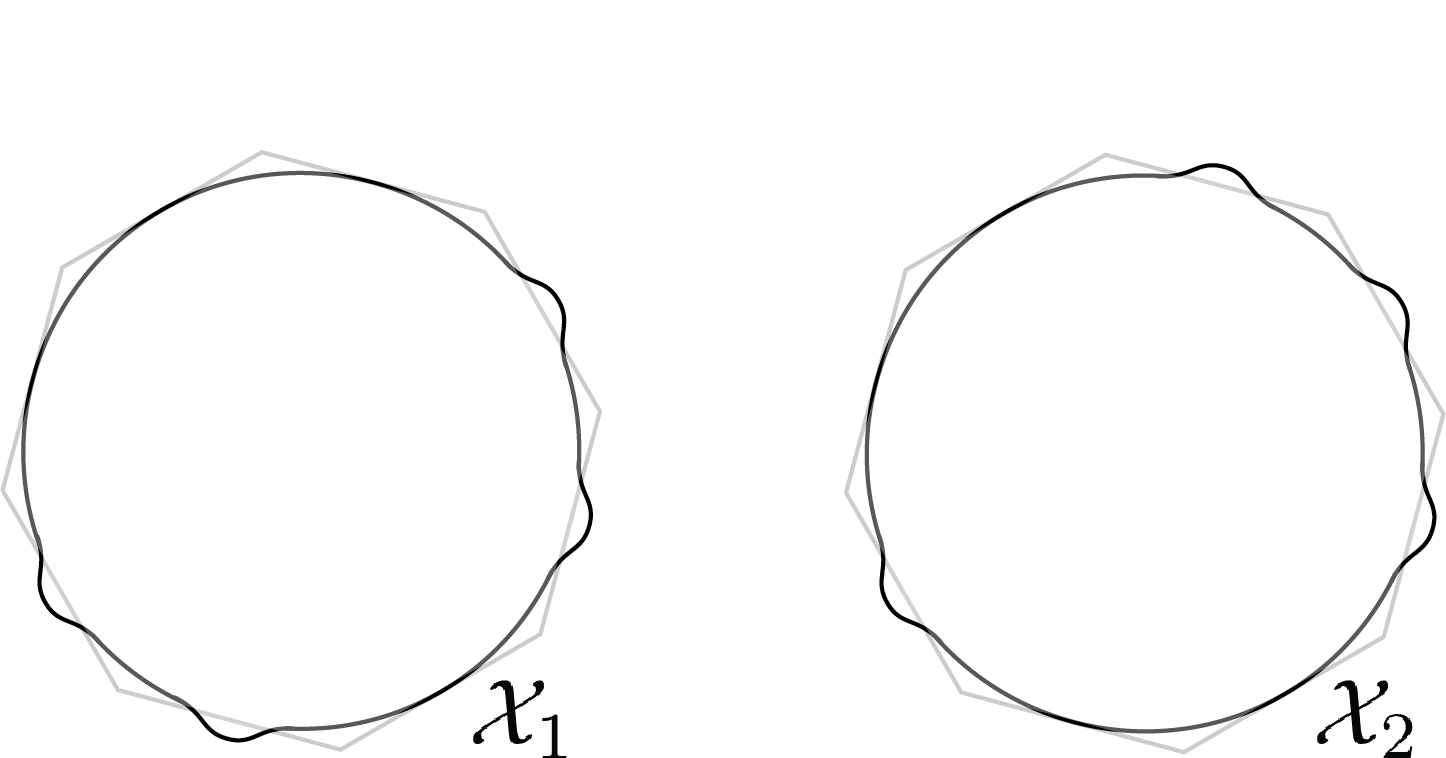}
\end{center}
\caption{Counterexample to the Curve Histogram Conjecture with both curves close to the unit circle. Circumscribed octagons used to guide the construction are shown with lower opacity.}\label{fig:pessimistic}
\end{figure}

\begin{example} In contrast to Example \ref{exmp:pessimistic}, we now give an example of an infinite family of curves containing the circle whose members are completely distinguished from one another by their distance distributions. given $a \geq b \geq 0$, let $X_{a,b}$ be an ellipse with semi-major axis $a$ and semi-minor axis $b$. We treat each such $X_{a,b}$ as a mm-space $\mX_{a,b}$ by endowing it with extrinsic Euclidean distance and normalized arclength measure. We claim that if ellipses $X_{a,b}$ and $X_{a',b'}$ have the same distance distributions $H_{\mX_{a,b}} = H_{\mX_{a',b'}}$ then $a=a'$ and $b=b'$, whence we conclude that the curves are related by a rigid motion. Indeed, if the ellipses have the same distance distributions, then they have the same diameter, as this can be read off from the distance distribution by the formula $\mathrm{diam}(X_{a,b}) = \inf\{r \geq 0 \mid H_{\mX_{a,b}}(r) = 1\}$. The diameter of $X_{a,b}$ is $2a$ and this implies that $a = a'$. Using the degree-1 term of the expansion  \eqref{eqn:global_distribution_hypersurface} in the following subsection, we also see that the ellipses have the same length. For fixed $a$, the length of an ellipse $X_{a,b}$ is a strictly increasing function of $b$. Thus $a = a'$ and equality of lengths together imply that $b = b'$. 
\end{example}

Proposition \ref{prop:curve_histogram_conjecture_false} shows that the global distance distribution is not globally injective on the category of embedded plane curves, answering inverse problem (\globalInj) in the negative for plane curves. The question of whether it is locally injective remains open (i.e., the answer to (\local) for plane curves), and is stated here more formally.

\begin{question}
 If $X_\tau$ is a smooth one-parameter family of embedded plane curves such that the family of global distance distributions $H_{\mathcal{X}_\tau}$ is constant in $\tau$, must each pair of curves in the family differ by a rigid motion?
\end{question}

\subsubsection*{Injectivity of the Local Distribution for Curves}

One could ask whether global injectivity in the sense of (\globalInj) holds for \emph{local} distance distributions of plane curves. We consider the subcategory of smooth plane curves whose objects have the generic condition that their signed curvature functions have isolated zeros and whose morphisms are twice continuously differentiable and \emph{infinitesimally measure-perserving}: that is, maps whose derivatives preserve norms of tangent vectors.

\begin{proposition}[Global Injectivity of Local Distance Distributions for Plane Curves (\globalInj)]\label{prop:local_distribution_curves}
Let $X_1$, $X_2$ be twice continuously differentiable simple closed plane curves whose signed curvature functions have isolated zeros. If there exists a twice continuously differentiable  infinitesimally measure-preserving map $\phi:X_1 \rightarrow X_2$ such that $h_{\mX_1}(x,r) = h_{\mX_2}(\phi(x),r)$ for all $(x, r) \in X_1 \times \R_{\geq 0}$, then $X_1$ is isometric to $X_2$. 
\end{proposition}

\begin{proof}
Pick an arclength parameterization $\gamma_1:S^1 \rightarrow X_1$ of $X_1$ and let $\gamma_2  = \phi \circ \gamma_1$. By our assumptions, $\gamma_2$ is an arclength parameterization of $X_2$. It follows from \cite[Section 7]{brinkman2012invariant} (see also the theorem of Karp and Pinksy \cite{karp1989volume}, given as Theorem \ref{thm:sphere_distributions} below) that we have Taylor expansions
\begin{linenomath}\begin{equation*}
h_{\mX_j}(\gamma_j(s),r) = \frac{2r}{\ell(X_j)} + \frac{r^3}{12 \ell(X_j)} \kappa_j(s)^2 + O(r^5),
\end{equation*}\end{linenomath}
where $\ell(X_j)$ is the length and $\kappa_j$ is the signed curvature of curve $X_j$. For each $s \in \mathbb{S}^1$, 
\begin{linenomath}\begin{equation*}
\kappa_1(s)^2 = \kappa_2(s)^2,
\end{equation*}\end{linenomath}
where we are expressing the $\kappa_j$ as functions on $\mathbb{S}^1$ via our parameterizations $\gamma_j$. 

So far we have that $\kappa_1(s) = \pm \kappa_2(s)$ for each $s$ and we claim that the functions must agree up to a \emph{global} choice of sign; i.e., $\kappa_1 = \pm \kappa_2$. Since plane curves with the same signed curvature up to a global sign must differ by a rigid Euclidean motion, establishing this claim will complete the proof. To obtain a contradiction, assume that this is not the case. By our genericity assumption, it must be that there exist points $\gamma_1(s_0) \in X_1$ and $\gamma_2(s_0) \in X_2$ of vanishing curvature such that $\kappa_1 (s) = \pm \kappa_2(s)$ for $s$ sufficiently close to $s_0$ with $s < s_0$ and $\kappa_1(s) = \mp \kappa_2(s)$ for $s$ sufficiently close to $s_0$ with $s > s_0$ (choosing an ordering of $\mathbb{S}^1$ locally). We will show that this forces inequality in the local distance distributions. To do so, we restrict to one of many possible (but essentially equivalent) cases.

Without loss of generality, assume that there is an $\epsilon > 0$ and a point $ \gamma_1(s_0) \in X_1$ of vanishing curvature such that $\kappa_1(s) > 0$ for all $s \in (-\epsilon,\epsilon) \setminus\{0\}$ such that $\kappa_2(s) > 0$ for $s \in (-\epsilon,0)$ and $\kappa_2(s) < 0$ for $s \in (0,\epsilon)$. After applying some rigid motions, we can view $X_1$ near $\gamma_1(s_0)$ as the graph of a function $f:(-\delta,\delta) \rightarrow \R$
\begin{linenomath}\begin{equation*}
f(x) = \left\{\begin{array}{cc}
f_1(x) & x \in (-\delta,0] \\
f_2(x) & x \in (0,\delta),\end{array}\right.
\end{equation*}\end{linenomath}
where $\delta > 0$ is sufficiently small and $f_1,f_2$ are both strictly convex. The curve $X_2$ is then expressed near $\gamma_2(s_0)$ as the graph of
\begin{linenomath}\begin{equation*}
g(x) = \left\{\begin{array}{cc}
f_1(x) & x \in (-\delta,0] \\
-f_2(x) & x \in (0,\delta). \end{array}\right.
\end{equation*}\end{linenomath}
(See Figure \ref{fig:generic_curve_picture} for an example.)
Now fix $x_0 \in (-\delta,0)$. It follows from the definitions of these functions that for any $x \in (0,\delta)$,
\begin{equation}\label{eqn:distance_inequality}
\| (x,f(x)) - (x_0,f(x_0)) \| < \| (x,g(x)) - (x_0,g(x_0)) \|.
\end{equation}
Choose $|x_0|$ small enough so that $B_{\R^2}((x_0,f_1(x_0)),2x_0)$ only intersects the curves $X_j$ in the graphs of $f$ and $g$, respectively. Consider the arclength $\ell_j$ of $B_{\R^2}((x_0,f_1(x_0)),2x_0)\cap X_j$. For $j=1$, this length breaks into two pieces: $\ell_1 = \ell_1^- + \ell_1^+$, 
where $\ell_1^-$ is the length of the intersection of the ball with the graph of $f_1$ and $\ell_1^+$ is the length of the intersection with the graph of $f_2$. We also have $\ell_2 = \ell_2^- + \ell_2^+$, with $\ell_2^\pm$ defined similarly. By construction, we have $\ell_1^- = \ell_2^-$. On the other hand \eqref{eqn:distance_inequality} implies that $\ell_1^+ > \ell_2^+$. This implies that the local distance distributions of $X_1$ and $X_2$ are not the same, and we have arrived at the desired contradiction.
\end{proof}
\begin{figure}
    \centering
    \includegraphics[width = 0.7\textwidth]{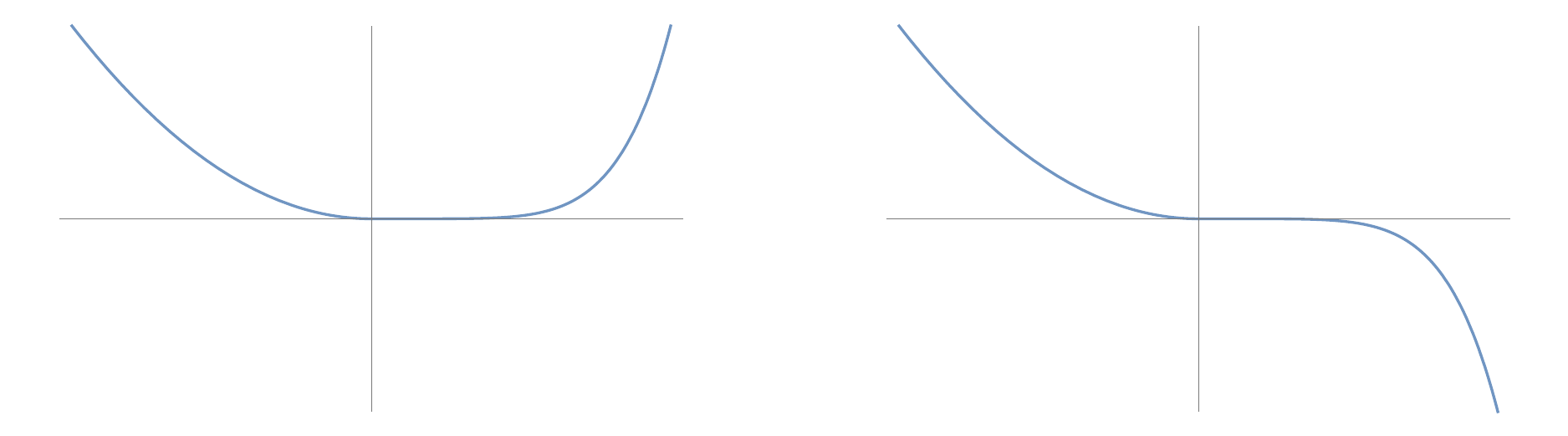}
    \caption{General picture of the local behavior of curves $X_1$ and $X_2$ from the proof of Proposition \ref{prop:local_distribution_curves}.}
    \label{fig:generic_curve_picture}
\end{figure}

\subsection{Sphere Characterization for Embedded Manifolds}

Example \ref{exmp:pessimistic} shows that inside any arbitrarily small Hausdorff neighborhood of the circle there exist two non-isometric curves with the same global distance distribution. This construction can be generalized to give a similar example for higher-dimensional manifolds.

\begin{figure}
\begin{center}
\includegraphics[width = 0.8\textwidth]{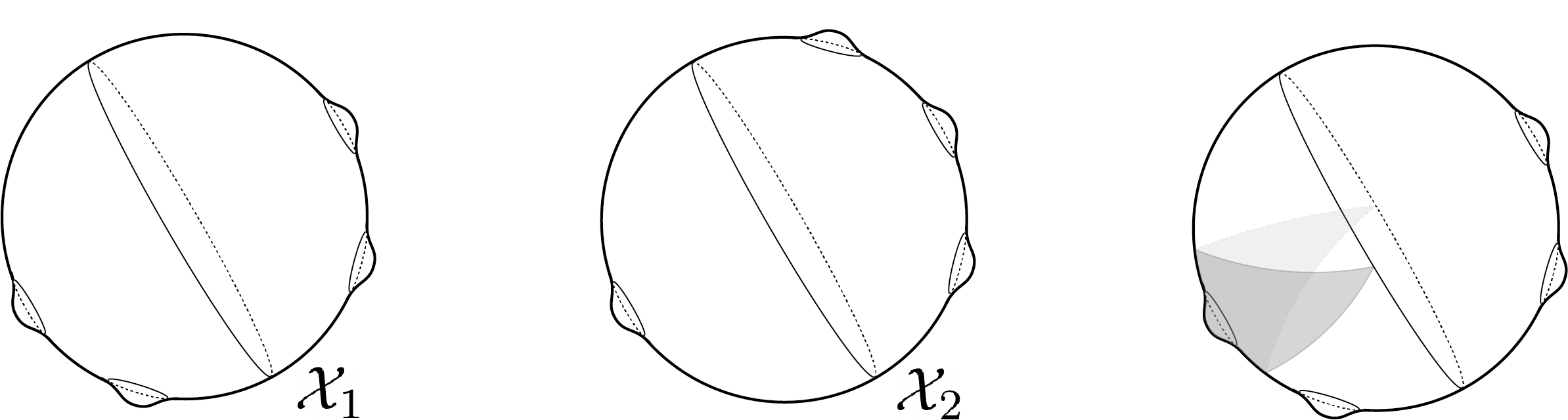}
\end{center}
\caption{Nonisometric surfaces $\mX_1$ and $\mX_2$ with the same distance distributions, as described in Example \ref{exmp:surfaces_same_distribution}. The figure on the right shows an example lune region of $\mX_1$ used to prove that $H_{\mX_1} = H_{\mX_2}$.}\label{fig:surface_counterexample}
\end{figure}

\begin{example}\label{exmp:surfaces_same_distribution}
We can extend Example \ref{exmp:pessimistic} to construct two surfaces $\mX_1$ and $\mX_2$, each of which is arbitrarily close to the unit sphere in Hausdorff distance, such that $\mX_1 \not \approx \mX_2$ but $H_{\mX_1} = H_{\mX_2}$. Instead of adding small bumps to a unit circle which mimic the pattern of the polygons in Figure \ref{fig:curve_histogram_conjecture}, we instead add bumps centered along a great circle of a unit sphere. The construction is shown in Figure \ref{fig:surface_counterexample}. To show that $H_{\mX_1} = H_{\mX_2}$, we apply the strategy of the proof of Proposition \ref{prop:curve_histogram_conjecture_false}. In this case, the decompositions of the surfaces necessary for the application of Theorem \ref{thm:main_theorem} are obtained dividing each surface into 8 lunes---an example of such a lune is shown on the right in Figure \ref{fig:surface_counterexample}. Doing so in analogy with Proposition \ref{prop:curve_histogram_conjecture_false} allows us to apply the same matchings of isomorphic pairs.

This construction generalizes to arbitrary dimension, by adding small, disjointly supported bumps centered on a great circle of a $d$-sphere. The concept of a lune makes sense in arbitrary dimensions, as a spherical region bounded between two great semicircles and the proof strategy of Proposition \ref{prop:curve_histogram_conjecture_false} still goes through.
\end{example}

On the other hand, the global distance distribution does distinguish the unit circle (respectively, unit sphere) from all other sufficiently smooth simple closed curves (resp., closed surfaces), as we show in the next result. We extend the result to arbitrary dimension, assuming some curvature bounds. Here we consider an embedded, smooth, closed codimension-1 manifold $X$ in $\R^{d+1}$ as a mm-space $\mathcal{X}$ with extrinsic Euclidean distance and normalized $d$-dimensional Hausdorff measure. The unit $d$-sphere $\mathbb{S}^d$ is considered as a mm-space in this manner. Let $\{\kappa_j(p)\}_{j=1}^d$ denote the principal curvatures at a point $p \in X$ with respect to a chosen orientation.

\begin{theorem}[Sphere Characterization for Embedded Hypersurfaces (\spheres)]\label{thm:sphere_distributions}
Let $X$ be an embedded, closed, $d$-dimensional submanifold of $\R^{d+1}$ with principal curvatures  $\kappa_j$. 
\begin{enumerate}
    \item When $d \leq 2$, if $H_{\mathcal{X}}(r) = H_{\mathbb{S}^d}(r)$ for all $r \geq 0$, then $X$ is isometric to $\mathbb{S}^d$.
    \item When $d > 2$, further assume that all principal curvatures satisfy $|\kappa_j(p)| \leq 1$ at every point $p \in X$. If $H_\mathcal{X}(r)=H_{\mathbb{S}^d}(r)$ for all $r\geq 0$ then $X$ is isometric to $\mathbb{S}^d$.
\end{enumerate}
\end{theorem}

The proof will use the following result of Karp and Pinsky, which we express using our notation. For an embedded hypersurface $X$, let $\mathrm{vol}_X$ denote its induced volume form and denote its total volume as $\mathrm{Vol}(X)$.

\begin{theorem}[\cite{karp1989volume}]
Let $p \in X$ and let $\kappa_1(p),\ldots,\kappa_d(p)$ denote the principal curvatures of $X$ at $p$. Then for sufficiently small $r \geq 0$, we have 
\begin{equation}\label{eqn:Karp_Pinksy}
h_\mathcal{X}(p,r) = \frac{1}{\mathrm{Vol}(X)} \left(\frac{\sigma_{d-1}}{d} r^d + \frac{\sigma_{d-1}}{8 d (d+2)} \left(2 A - B \right) r^{d+2} + O\left(r^{d+3}\right)\right),
\end{equation}
where 
\begin{linenomath}\begin{equation}\label{eqn:curvature_terms}
A := \sum_{j=1}^d \kappa_j^2(p), \;\; \;\; B := \left(\sum_{j=1}^d \kappa_j(p)\right)^2
\end{equation}\end{linenomath}
and $\sigma_k = \frac{2 \pi^{\frac{k+1}{2}}}{\Gamma\left(\frac{k+1}{2}\right)}$ is the volume of the unit $k$-sphere.
\end{theorem}

\begin{remark}
The quantities $A$ and $B$ have geometric significance: $A$ is the \emph{Casorati curvature} \cite{casorati1890mesure} of $X$ at $p$, while $B$ is the square of $d$-times the mean curvature of $X$ at $p$. 
\end{remark}

The proof of Theorem \ref{thm:sphere_distributions} will use the following elementary lemma.

\begin{lemma}\label{lem:polynomial}
Let $d > 2$ and let $G(x_1,\ldots,x_d) = 2 \sum_{j=1}^d x_j^2 - \left(\sum_{j=1}^d x_j\right)^2$.
Then, on the domain $[-1, 1]^d \subset \R^d$, $G$ takes its minimum value $2d-d^2$ only at the points $\pm(1,1,\ldots,1)$.
\end{lemma}

\begin{proof}
For any point $(x_1,\ldots,x_d)$ in the domain of $G$ with $x_i\neq x_j$, we have that the point  \begin{linenomath}\begin{equation*}
\left(x_1,\dots,\frac{x_i+x_j}{2},\ldots,\frac{x_i+x_j}{2},\ldots,x_d\right)
\end{equation*}\end{linenomath}
is still in the domain of $G$ and $G(x_1,\ldots,x_d)>G(x_1,\dots,\frac{x_i+x_j}{2},\ldots,\frac{x_i+x_j}{2},\ldots,x_d)$. Indeed, the quantity $\sum_{j=1}^d x_j$ does not change under this substitution, but $\sum_{j=1}^d x_j^2$ becomes smaller due to the  inequality  $x^2+y^2>2(\frac{x+y}{2})^2$, valid for all $x\neq y$. This means that we need to consider the minimum of $g(x):=G(x,\ldots,x) = (2d-d^2)x^2$ in $[-1,1]$. But since $d>2$, the  minimizers are $x=\pm 1$.
\end{proof}

\begin{proof}[Proof of Theorem \ref{thm:sphere_distributions}]
Integrating \eqref{eqn:Karp_Pinksy} over $X$ with respect to the normalized Riemannian volume yields
\begin{equation}\label{eqn:global_distribution_hypersurface}
H_\mathcal{X}(r) = \frac{1}{\mathrm{Vol}(X)} \frac{\sigma_{d-1}}{d} r^d + \frac{1}{\mathrm{Vol}(X)^2} \frac{\sigma_{d-1}}{8d(d+1)} r^{d+2} \int_X \big(2A - B\big) \; \mathrm{vol}_X(dp) + O\left(r^{d+3}\right),
\end{equation}
for sufficiently small $r$, with $A$ and $B$ as in \eqref{eqn:curvature_terms}. Since the principal curvatures of $\mathbb{S}^d$ are all constantly equal to one, we derive from \eqref{eqn:global_distribution_hypersurface} that, for sufficiently small $r$,
\begin{linenomath}\begin{equation*}
H_{\mathbb{S}^d}(r) = \frac{1}{\sigma_d} \frac{\sigma_{d-1}}{d} r^d + \frac{1}{\sigma_d}\frac{\sigma_{d-1}}{8d(d+1)} r^{d+2} (2d - d^2) + O\left(r^{d+3}\right).
\end{equation*}\end{linenomath}
Combining \eqref{eqn:global_distribution_hypersurface} with the assumption $H_\mathcal{X}(r) = H_{\mathbb{S}^d}(r)$, we see that $\mathrm{Vol}(X) = \sigma_d$ and
\begin{linenomath}\begin{equation}\label{eqn:sphere_characterization_integral_equality}
\int_X \big(2A - B\big) \; \mathrm{vol}_X = (2d - d^2)\sigma_d.
\end{equation}\end{linenomath}

We now treat the cases $d=1$, $d=2$ and $d > 2$ separately. In the $d=1$ case, we note that \eqref{eqn:global_distribution_hypersurface} was derived specifically for curves in  \cite[Section 7]{brinkman2012invariant}. Here, $\mathrm{vol}_X = ds$ (arclength measure), $\mathrm{Vol}(X)$ (is the length of the curve) and condition \eqref{eqn:sphere_characterization_integral_equality} reads $\oint_X \kappa^2(s)\, ds = 2\pi$. By assumption, since the curve is simple its index \cite{kuhnel2006differential} is $1$, so that
\begin{linenomath}\begin{equation*}
\oint_X \kappa(s)\,ds = 2\pi.
\end{equation*}\end{linenomath}
Applying the Cauchy-Schwarz inequality, one has
\begin{linenomath}\begin{equation*}
4\pi^2 = \left(\oint_X \kappa(s)\cdot 1\,ds\right)^2 \leq \left(\oint_X \kappa^2(s)\,ds\right)\left(\oint_X 1\,ds\right)=2\pi\cdot \mathrm{Vol}(X) = 2\pi \cdot 2\pi= 4\pi^2.
\end{equation*}\end{linenomath}
The inequality is then forced to be an equality, which implies that the two involved functions ($\kappa$ and $1$) must be proportional, i.e. $\kappa$ is constant. Since $\mathrm{Vol}(X)=2\pi$, then $\kappa(s) = 1$ for all $s\in[0,2\pi]$. It follows that $\mX \approx \mathbb{S}^1$.

When $d=2$, we have $2d-d^2 = 0$, hence
\begin{linenomath}\begin{equation*}
0 = \int_X \big(2A - B\big) \; \mathrm{vol}_X = \int_X (\kappa_1(p) - \kappa_2(p))^2 \; \mathrm{vol}_X(dp).
\end{equation*}\end{linenomath}
This implies $\kappa_1(p) = \kappa_2(p)$ for all $p \in X$. A closed surface in which all points are umbillic is a sphere \cite[Section 3-2, Proposition 4]{do2016differential}, and $\mathrm{Vol}(X) = \sigma_2$ implies that $X$ is isometric to $\mathbb{S}^2$.

Finally, suppose $d > 2$. With $G$ denoting the polynomial function from Lemma \ref{lem:polynomial}, we have
\begin{linenomath}\begin{equation*}
(2d - d^2)\sigma_d = \int_X G(\kappa_1(p),\ldots,\kappa_d(p)) \; \mathrm{vol}_X(dp) \geq \mathrm{min}(G) \mathrm{Vol}(X) = (2d-d^2) \sigma_d.
\end{equation*}\end{linenomath}
It follows that the inequality must actually be an equality, so that $G$ is constantly equal to its minimum value. By Lemma \ref{lem:polynomial}, this minimum is unique and we have $\kappa_1 = \kappa_2 = \cdots = \kappa_d = \pm 1$. This implies that $\mX \approx  \mathbb{S}^d$. 
\end{proof}

\subsection{Sphere Characterization for Riemannian Manifolds}

For the rest of this section, we consider a $d$-dimensional Riemannian manifold $(X,g)$ as a mm-space $\mathcal{X}=(X,d_g,\mu_g)$ with geodesic distance $d_g$ and normalized Riemannian volume measure $\mu_g$. We denote the Riemannian volume form by $\mathrm{vol}_g$ and the total volume of $X$ with respect to $g$ by
\begin{linenomath}\begin{equation*}
\mathrm{Vol}_g(X) = \int_X \mathrm{vol}_g(dp),
\end{equation*}\end{linenomath}
so that  $\mu_g = \frac{1}{\mathrm{Vol}_g(X)} \mathrm{vol}_g$. In this subsection we denote by $\mathbb{S}^d$ the unit sphere with its canonical Riemannian mm-space structure. Since the metric is canonical here, we use specialized notation $d_{\mathbb{S}^1}$, $\mathrm{vol}_{\mathbb{S}^d}$ and $\mathrm{Vol}(\mathbb{S}^d)$ for geodesic distance, volume form and total volume, respectively.

\begin{example}\label{exmp:pessimistic_riemannain}
Consider the construction of Example \ref{exmp:surfaces_same_distribution}. Endowing the surfaces $\mX_1$ and $\mX_2$ with geodesic distance induced by their embedding in $\R^3$ yields Riemannian surfaces $\mX_1$ and $\mX_2$ which are not isomorphic, can be made arbitrarily close to the 2-sphere in Gromov-Hausdorff distance, and which have the same global distance distribution. The last point can once again be proved by applying Theorem \ref{thm:main_theorem} using partitions into lunes (since the pairs of lunes are related by rigid motions, they are isometric with respect to their Riemannian structures). As was the case in Example \ref{exmp:surfaces_same_distribution}, this example generalizes to arbitrary dimension.
\end{example}

We now give a characterization of the unit sphere in the category of Riemannian mm-spaces. The statement of the theorem uses notation introduced earlier: the global shape measure \eqref{eqn:global_shape_measure} and the Wasserstein 1-distance \eqref{eqn:wasserstein_distance}.

\begin{proposition}[Sphere Characterization for Riemannian Manifolds (\spheres)]\label{prop:sphere_distributions_riemannian}
Let $(X,g)$ be a $d$-dimensional Riemannian manifold with  Ricci curvatures bounded below by $d-1$. Let $dH_\mathcal{X}$ and $dH_{\mathbb{S}^{d}}$ denote the  global shape measure of $\mathcal{X}$ and the $d$-dimensional round sphere with its standard Riemannian structure, respectively.
\begin{itemize}
    \item[(1)] There exists an $\epsilon = \epsilon(d)>0$ such that if the Wasserstein $1$-distance between the shape measures satisfies
    \begin{linenomath}\begin{equation*}
    L_\mathrm{H}(\mX,\mathbb{S}^d) = d_{\mathrm{W},1}^\R(dH_\mathcal{X},dH_{\mathbb{S}^d}) < \epsilon
    \end{equation*}\end{linenomath}
    then $X$ is diffeomorphic to $\mathbb{S}^d$. 
    \item[(2)] If the Wasserstein distance is zero---equivalently,  $H_\mathcal{X}(r)=H_{\mathbb{S}^d}(r)$ for all $r\geq 0$, or $L_\mathrm{H}(\mathcal{X},\mathbb{S}^d) = 0$---then $X$ is isometric to $\mathbb{S}^d$. 
    \end{itemize}
\end{proposition}

This proposition and its proof illustrate the the connection between distance distribution inverse problems and classical results in Riemannian geometry. The proof follows from a result of Kokkendorff \cite{kokkendorff2008characterizing} (the second part also follows from a classic result of Cheng \cite{cheng}).  We will use the concept of $1$-diameters, together with a lemma. For a mm-space $\mathcal{X}$, the \emph{$1$-diameter of $\mathcal{X}$}, $\mathrm{diam}_1(\mathcal{X})$, is the number
\begin{linenomath}\begin{equation*}
\mathrm{diam}_1(\mathcal{X}) := \iint_{X \times X} d_X(x,x') \mu_X \otimes \mu_X (dx \times dx').
\end{equation*}\end{linenomath}
Properties of this invariant (and more general $p$-diameters of a mm-space) are described in  \cite[Section 5.2]{memoli2011gromov}.

\begin{lemma}\label{lem:1-diameter}
For mm-spaces $\mathcal{X}$ and $\mathcal{Y}$ we have the lower bound
\begin{linenomath}\begin{equation*}
\left|\mathrm{diam}_1(\mathcal{X}) - \mathrm{diam}_1(\mathcal{Y}) \right| \leq d_{\mathrm{W},1}^\R (dH_\mathcal{X},dH_{\mathcal{Y}}).
\end{equation*}\end{linenomath}
\end{lemma}

\begin{proof}
Let $\mathcal{Z} = \left(\{z\},0,\delta_z\right)$ denote the one point mm-space. One can check that $dH_\mathcal{Z} = \delta_0$, the $\delta$  probability measure on $\R$ supported on $\{0\}$. Thus the only coupling between the probability measures $dH_\mathcal{X}$ and $dH_\mathcal{Z}$ is the product measure $dH_\mathcal{X} \otimes \delta_0$. From this observation, one is able to show $d_{\mathrm{W},1}^\mathbb{R}(dH_\mathcal{X},dH_\mathcal{Z}) = \mathrm{diam}_1(\mathcal{X})$.
The claim then follows from the reverse triangle inequality.
\end{proof}

\begin{proof}[Proof of Proposition \ref{prop:sphere_distributions_riemannian}]

To prove the first statement, we appeal to the previously mentioned sphere characterization theorem of Kokkendorff \cite[Theorem 4]{kokkendorff2008characterizing}. Translating Part 3 of this theorem into our terminology, it says that there exists an $\epsilon = \epsilon(d)$ such that, under our curvature assumptions, $\left|\mathrm{diam}_1(\mathcal{X}) - \mathrm{diam}_1(\mathbb{S}^d)\right| < \epsilon$ implies that $X$ is diffeomorphic to $\mathbb{S}^d$. Together with Lemma \ref{lem:1-diameter}, this proves the first claim of the theorem.

The second claim of the theorem follows directly from Part 2 of \cite[Theorem 4]{kokkendorff2008characterizing}. Alternatively, we can observe that $H_\mathcal{X} = H_{\mathbb{S}^d}$ implies that $X$ and $\mathbb{S}^d$ have the same diameter. It then follows immediately from Cheng's Rigidity Theorem \cite{cheng} that $X$ and $\mathbb{S}^d$ are isometric.
 \end{proof}

Observe that in both the embedded and Riemannian categories, our sphere characterization theorems use assumptions on curvature for technical reasons. On the other hand, we do not have counterexamples showing that these assumptions are necessary. The proofs of these results only use data about the distributions coming from the infinitesimal radius regime (in the form of Taylor expansions) or rather coarse long-range information (diameter of the spaces), so it is possible that using more detailed long-range information from the distributions could allow one to drop the curvature assumptions.

\begin{question}
 Can the curvature assumptions be removed in Theorem \ref{thm:sphere_distributions} and Proposition \ref{prop:sphere_distributions_riemannian}?
\end{question}

\subsection{Riemannian Surfaces}

The results above on general Riemannian manifolds can be strengthened if we restrict to the 2-dimensional case.

\subsubsection*{Constant Curvature Surface Characterization}

Proposition \ref{prop:sphere_distributions_riemannian} on sphere characterization for Riemannian manifolds can be improved and generalized in the category of Riemannian surfaces, where one can not only remove the curvature assumption in the case of spheres but also one can obtain classification up to diffeomorphism for other constant curvature model spaces.

\begin{theorem}[Constant Curvature Surface Characterization for Riemannian Surfaces (\spheres)]\label{thm:surface_characterization}
Let $(X_0,g_0)$ and $(X,g)$ be any two oriented 2-dimensional closed Riemannian manifolds  such that $X_0$ has constant Gaussian curvature $\kappa \in \mathbb{R}$ and $H_{\mathcal{X}} = H_{\mathcal{X}_0}$. Then $X$ is diffeomorphic to $X_0$ and also has constant curvature $\kappa$. In particular,  if $\kappa > 0$ then $X$ and $X_0$ are both isometric to a round sphere of radius $\kappa^{-1/2}$.
\end{theorem}

The proof uses the following lemma (stated for arbitrary dimension, because it will be useful later on).

\begin{lemma}\label{lem:riemannian_distributions}
Let $\mathcal{X}$ be a $d$-dimensional Riemannian manifold with its Riemannian mm-space structure. The local distribution of distances of $\mathcal{X}$ admits the following Taylor expansion: for $r>0$ small enough and all $p \in X$,
\begin{linenomath}\begin{equation}\label{eqn:local_surface_taylor_expansion}
h_\mathcal{X}(p,r) = \frac{\omega_d(r)}{\mathrm{Vol}_g(X)}\left(1-\frac{S_g(p)}{6(d+2)}r^2 + O(r^4)\right),
\end{equation}\end{linenomath}
where $\omega_d(r)$ is the volume of a ball of radius $r$ in $\R^d$ and  $S_g(p)$ is the \emph{scalar curvature} of $X$ at the point $p$. It then follows that 
\begin{equation}\label{eqn:global_distribution_riemannian}
H_\mathcal{X}(r) = \frac{\omega_d(r)}{\mathrm{Vol}_g(X)}\left(1-\frac{\int_X S_g(p)\,\mathrm{vol}_g(dp)}{6(d+2) \mathrm{Vol}_g(X)}r^2 + O(r^4)\right).
\end{equation}
Moreover, if $X$ is a closed surface, then
\begin{linenomath}\begin{equation*}
    H_{\mathcal{X}}(r) = \frac{\pi r^2}{\mathrm{Vol}_g(X)}\left(1-\frac{r^2}{12\,\mathrm{Vol}_g(X)} \chi(X) +\frac{r^4}{360 \, \mathrm{Vol}_g(X)}\int_X K_g^2(p)\, \mathrm{vol}_g(dp)+O(r^6)\right),
\end{equation*}\end{linenomath}
where $\chi(X)$ is the Euler characteristic and $K_g$ is the Gauss curvature of the surface $X$.
\end{lemma}

\begin{proof}
The local distance distribution expansion can be found in \cite[p. 446]{memoli2011gromov}. The global distance distribution expansion then follows by integrating. To get the expression for surfaces, we first observe that scalar curvature is twice Gauss curvature and apply the Gauss-Bonnet theorem---this explains the first two terms in the expansion. The $r^6$ term involving squared curvature follows from the expansion in \cite[Corollary 3.4]{gray1979riemannian}.
\end{proof}

Because the Euler characteristic determines the diffeotype of a closed orientable surface, we have the following immediate corollary, answering a strengthening of (\homotopy).

\begin{Corollary}[Diffeotype Characterization for Riemannian Surfaces (\homotopy)]\label{cor:diffeotype_surfaces}
Let $\catC$ denote the full subcategory of mm-spaces whose objects are closed Riemannian surfaces. Then $L_\mathrm{H}^\catC(\mX,\mY)=0$ implies $X$ and $Y$ are diffeomorphic.
\end{Corollary}

We are now prepared to prove the theorem.

\begin{proof}[Proof of Theorem \ref{thm:surface_characterization}]
 By Lemma \ref{lem:riemannian_distributions}, the condition that $H_{\mathcal{X}_0} = H_{\mathcal{X}}$ gives that $\mathrm{Vol}_g(X) = \mathrm{Vol}_{g_0}(X_0)$, $\int_X K_g = \int_{X_0} K_{g_0}$, and $\int_X K_g^2 = \int_{X_0} K_{g_0}^2$. It follows that $\chi(X) = \chi(X_0)$, hence $X$ is diffeomorphic to $X_0$ by the classification of oriented surfaces. Moreover, we have that
\begin{linenomath}\begin{equation*}
\int_X K_g^2 \cdot \int_X 1^2 = \int_{X_0} K_{g_0}^2 \cdot \int_{X_0} 1^2 = \left(\int_{X_0} K_{g_0} \right)^2 = \left(\int_X K_g\right)^2 \leq \int_X K_g^2 \cdot \int_X 1^2,
\end{equation*}\end{linenomath}
where we have invoked the equalities observed above and applied Cauchy-Schwarz . The inequality is then forced to be equality, so that $K_g$ is constant and $\int_X K_g = \int_{X_0} K_{g_0}$ forces $K_g = K_{g_0} = \kappa$. If $\kappa > 0$, then the Uniformization Theorem implies that $X$ and $X_0$ are isometric to a round sphere of radius $\kappa^{-1/2}$.
\end{proof}

\begin{example}[Flat Tori]\label{exmp:flat_tori}
This example shows that the global distance distribution is not able to distinguish closed surfaces of constant curvature $\kappa = 0$ (i.e., flat tori) up to isometry. Consider the flat torus given as the quotient of $\R^2$ by the lattice generated by translations along the vectors $u = [1,0]^T$ and $v = [1/3,1]^T$. Denote this torus by $\mathcal{T} = (T,d_T,\mu_T)$, where $d_T$ is geodesic distance and $\mu_T$ is area measure; that is, $T = \R^2/\sim$, where $x \sim y \Leftrightarrow y = x + ku + \ell v$ for some $k,\ell \in \mathbb{Z}$. Let $D$ denote the fundamental domain for $T$ with lower left corner at $0$, and adjacent sides given by $u$ and $v$---we consider $D$ to be a half-open parallelogram so that any point in $T$ has a unique lift in $D$. Let $p \in T$ denote the image of $0$ under the quotient map. By the homogeneity of $T$, $H_\mathcal{T}(\cdot)$ is determined by $h_\mathcal{T}(p,\cdot)$. Computations in $T$ can be done by lifting to $D$: in particular, the volume of the ball $B_T(p,r)$ can be computed as 
\begin{linenomath}\begin{equation*}
\mu_T(B_T(p,r)) = \mu\left(\bigcup \left\{B_{\R^2}(w,r) \cap D \mid w = 0,u,v,u+v \right\}\right),
\end{equation*}\end{linenomath}
where, for this example, $\mu$ is Lebesgue measure. In the following, we condense notation by writing $w_1 = 0, w_2 = u, w_3 = v, w_4 = u+v$ and $B_i(r) = B_{\mathbb{R}^2}(w_i,r)$. The local distance distribution satisfies the following (note that the calculations in this example are summarized in Figure \ref{fig:torus_distributions}):
\begin{itemize}
    \item For $r \leq \frac{1}{2} = \frac{\|u\|}{2}$, the balls $B_i(r)$ are disjoint, so that $h_\mathcal{T}(p,r) = \pi r^2$.
    \item For $\frac{1}{2} < r \leq \frac{\sqrt{10}}{6} = \frac{\|v\|}{2}$, $B_1(r) \cap B_2(r)$ and $B_3(r) \cap B_4(r)$ are nonempty and we have
    \begin{linenomath}\begin{equation*}
    h_\mathcal{T}(p,r) = \pi r^2 - \mu(B_1(r) \cap B_2(r) \cap D) - \mu(B_3(r) \cap B_4(r) \cap D)
    = \pi r^2 - \mu(B_1(r) \cap B_2(r)).
    \end{equation*}\end{linenomath}
    \item For $\frac{\sqrt{10}}{6} \leq r \leq \frac{\sqrt{13}}{6} = \frac{\|v-u\|}{2}$, $B_1(r) \cap B_2(r)$ and $B_3(r) \cap B_4(r)$ are nonempty, as above, and we have $B_1(r)\cap B_3(r)$ and $B_2(r) \cap B_4(r)$ nonempty as well. It follows that
    \begin{align*}
    h_\mathcal{T}(p,r) &= \pi r^2 - \mu(B_1(r) \cap B_2(r) \cap D) - \mu(B_3(r) \cap B_4(r) \cap D) \\
    &\hspace{2in} - \mu(B_1(r) \cap B_3(r) \cap D) - \mu(B_2(r) \cap B_4(r) \cap D) \\
    &= \pi r^2 - \mu(B_1(r) \cap B_2(r)) - \mu(B_1(r) \cap B_4(r)).
    \end{align*}
    \item The diameter of $T$ is the circumradius of the triangle formed by $u$, $v$ and $v-u$; specifically, $\mathrm{diam}(T) = \frac{\sqrt{130}}{18}$. When $\frac{\sqrt{13}}{6} < r \leq \frac{\sqrt{130}}{18}$, we get an intersection of $B_2(r)$ and $B_3(r)$, so that 
    \begin{linenomath}\begin{equation*}
    h_\mathcal{T}(p,r) = \pi r^2 - \mu(B_1(r) \cap B_2(r)) - \mu(B_1(r) \cap B_4(r)) - \mu(B_2(r) \cap B_3(r)).
    \end{equation*}\end{linenomath}
    \item Finally, when $r > \frac{\sqrt{130}}{18}$, $\bigcup_{i=1}^4 B_i$ covers $D$, so that $h_\mathcal{T}(p,r) = 1$.
\end{itemize}

We now construct another flat torus $\mathcal{T}'$ such that $\mathcal{T}' \not \approx \mathcal{T}$ and such that, $h_{\mathcal{T}'}(p',r) = h_\mathcal{T}(p,r)$ for all $r$, where $p'$ is the image of $0$ under the quotient map $\R^2 \to T'$. It then follows that $H_{\mathcal{T}} = H_{\mathcal{T}'}$, so that global distributions cannot distinguish flat tori up to isomorphism. Let $T'$ be the flat torus whose lattice is generated by the vectors $u' = u = [1,0]^T$ and $v' = [2/3,1]^T$---the key feature of $T'$ is that the triangles with side lengths $u,v,v-u$ and $u',v',v'-u'$ are isometric (differing by a reflection over a vertical line), while the fundamental domains $D$ and $D'$ are not isometric or related by an integral shearing transformation (hence $\mathcal{T} \not \approx \mathcal{T}'$). Due to the former property, it is straightforward to show, by arguments similar to the above, that $h_{\mathcal{T}'}(p',r) = h_\mathcal{T}(p,r)$ for all $r$.
\end{example}

\begin{figure}
    \centering
    \includegraphics[width = \textwidth]{Figures/distanceDistribution.png}
    \caption{This figure illustrates the computations of Example \ref{exmp:flat_tori}. The top and bottom rows show the local distance distribution at $0$ for the tori $T$ and $T'$, respectively. Each row shows a sequence of metric balls in the respective torus within the various radius regimes explained in the example, together with a rearrangement of the intersection region which demonstrates the corresponding volume formula from the example. The last entry in each row shows that the torus is covered by a ball with radius equal to a circumradius.}
    \label{fig:torus_distributions}
\end{figure}

\subsubsection*{Injectivity of the Local Distribution for Riemannian Surfaces}

In the following, we fix a 2-dimensional closed manifold $X$. For a chosen Riemannian metric $g$ on $X$, let $h_g$ denote the local distance distribution for $(X,g)$. We will adapt some ideas of Cartan \cite{Cartan1925}, borrowing terminology and notation from \cite{berger2012panoramic}. A metric $g$ on $X$ is called \emph{generic} if the differentials $dI_1$ and $dI_2$ of the functions $I_1:=K_g$ and $I_2:=\|dI_1 \|_g^2$ ($\|\cdot\|_g$ being the norm on $1$-forms induced by $g$) are pointwisely linearly independent on an open dense subset of $X$ (see \cite[Proposition 42, p. 234]{berger2012panoramic}). In fact, if $dI_1$ and $dI_2$ are  independent, then the metric tensor of the surface $X$ can be (pointwisely) determined on a  open dense set as explained in the proof of the following proposition.

\begin{proposition}\label{prop:local_dist_riem_surf}
Assume that $X$ is a smooth closed surface and $g_1$ and $g_2$ are two generic Riemannian metrics on $X$ such that $h_{g_1}=h_{g_2}$ over $X \times \R$, then $g_1=g_2$. 
\end{proposition}
\begin{proof} The proof is divided into two steps.

\smallskip
\noindent
\textbf{First step.} Assume that $g$  is any generic Riemannian metric  on a closed surface $X$.  Consider the following 4 functions $I_1,I_2,I_3,I_4:X\rightarrow \R$ induced from the Gauss curvature $K_g$:
\begin{linenomath}\begin{equation*}
    I_1:=K_g, \qquad
    I_2:=\|d I_1\|_g^2, \qquad
    I_3:=\langle d I_1, d I_2\rangle_{g}, \qquad
    I_4:= \|d I_2\|_g^2,
\end{equation*}\end{linenomath}
where  $dI_1$ and $dI_2$ denote the differentials (which are 1-forms on $X$) of the functions $I_1$ and $I_2$, respectively. Also, for $1$-forms $\alpha$ and $\beta$ on $X$,  $\langle\alpha,\beta\rangle_{g}$ denotes the induced inner product $g^{-1}(\alpha,\beta)$ and $\|\alpha\|_g^2$ denotes the squared norm $g^{-1}(\alpha,\alpha).$ These conventions together with the fact that the metric $g$ was assumed to be generic imply that, locally in an open dense subset $S \subset X$, the vector fields $\{\partial_{I_1},\partial_{I_2}\}$ are independent and that the inverse metric tensor field in the  local  coordinate system induced by $(I_1,I_2)$ can be written as 
\begin{linenomath}\begin{equation*}
g^{-1} = [\partial_{I_1},\partial_{I_2}]\,\Gamma\, [\partial_{I_1},\partial_{I_2}]^T = I_2\, \partial_{I_1} \partial_{I_1} +2 I_3\, \partial_{I_1} \partial_{I_2} + I_4\, \partial_{I_2} \partial_{I_2}
\end{equation*}\end{linenomath}
where $\Gamma:=\begin{psmallmatrix}I_2&I_3\\I_3&I_4\end{psmallmatrix}$.
But the genericity assumption guarantees that the determinant $I_2I_4-I_3^2$ of the matrix  $\Gamma$ is positive at points in $S$. That it is non-negative follows from the Cauchy-Schwarz inequality. From the conditions for equality in the Cauchy-Schwarz inequality we know that $I_2I_4-I_3^2$  will vanish if and only if $d I_1$ and $d I_2$ are linearly dependent. However, this condition is precluded by the genericity assumption, i.e. because $dI_1$ and $dI_2$ are independent inside $S$. Therefore we conclude that  $I_2I_4-I_3^2$ is positive, so that the matrix $\Gamma$ is indeed invertible (it is positive definite, actually). 
 From this it follows that the metric tensor $g$ at points in $S$ can be recovered explicitly as
\begin{linenomath}\begin{equation*}
g = [dI_1,dI_2]\,\Gamma^{-1}\,[dI_1,dI_2]^T = \frac{I_4 \, dI_1 \, dI_1 - 2 I_3 \, dI_1 \, dI_2 + I_2 \, dI_2 \, dI_2}{I_2 I_4 - I_3^2}.
\end{equation*}\end{linenomath}
See  \cite[p. 323]{Cartan1925} or \cite{mathOverflow} for more  details.

\smallskip
\noindent
\textbf{Second step:} Assume now that two generic Riemannian metrics $g_1$ and $g_2$ on $X$ are such that their respective local distance distributions agree: $h_{g_1}(x,r) = h_{g_2}(x,r)$ for all $x\in X$ and $r \in \R$. Then, by the Taylor expansion (\ref{eqn:local_surface_taylor_expansion}) and the observation that scalar curvature is twice Gauss curvature for surfaces, we have that $K_{g_1}(x) = K_{g_2}(x)$ for all $x\in X$. It then follows that the open dense subsets in the definition of generic metrics are the same for $g_1$ and $g_2$. Let $S$ denote this common open dense subset. Then, by the First Step above, we have that $g_1=g_2$ on $S$ and it follows by continuity that $g_1 = g_2$ on the entire surface.
\end{proof}

\subsection{Injectivity for Weighted Riemannian Manifolds}

Let $(X,g)$ be a compact Riemannian manifold. In this final subsection, we extend the class of mm-spaces derived from $X$ beyond those we have considered so far. As above, let $d_g$ denote geodesic distance and let $\mu_g$ denote the normalized Riemannian volume measure on $X$. A \emph{weighted Riemannian manifold} is a metric measure space $\mathcal{X} = (X,d_g,f\cdot\mu_g)$, where $f$ is a smooth probability density function with respect to $\mu_g$. Such spaces are frequently considered, for example, in the context of the  concentration of measure phenomenon \cite{gromov1983topological, funano2013concentration}. The next result shows that the normalized Riemannian volume is distinguished from other densities by the global distance distribution. For this section, when a base Riemannian manifold is fixed, we denote the global distance distribution of the weighted Riemannian manifold with density $f$ by $H_f$. We denote the global distance distribution with respect to normalized Riemannian volume by $H_g$. The following proposition addresses (\globalInj) in the category of densities over a fixed Riemannian manifold.

\begin{proposition}
Let $(X,g)$ be a compact Riemannian manifold and let $f$ be a probability density. If $H_f = H_g$, then $f$ is constantly equal to one.
\end{proposition}

\begin{proof}
The proof follows from the estimate
\begin{equation}\label{eqn:distribution_estimate}
    \int_X f(x)^2 \mu_g(dx) \leq 1.
\end{equation}
Assuming this claim, we deduce that
\begin{linenomath}\begin{equation*}
1 \leq \left(\int_X f(x) \mu_g(dx) \right)^2 \leq \int_X f(x)^2 \mu_g(dx) \cdot \int_X \mu_g(dx) \leq 1,
\end{equation*}\end{linenomath}
hence the Cauchy-Schwartz inequality must be an equality and $f$ is linearly dependent on the constant function $1$. It follows that $f =1$.

It remains to show \eqref{eqn:distribution_estimate}. Let $\epsilon > 0$. By the metric space Lebesgue Differentiation Theorem \cite[Theorem 1.8]{heinonen2012lectures}, for $r$ sufficiently small we have
\begin{linenomath}\begin{equation*}
f(x) < \frac{1}{\mu_g(B_X(x,r))} \int_{B_X(x,r)} f(y) \mu_g(dy) + \epsilon,
\end{equation*}\end{linenomath}
for $\mu_g$-almost every $x \in X$. This implies
\begin{linenomath}\begin{align*}
&\int_X f(x)^2 \mu_g(B_X(x,r)) \mu_g(dx) < \int_X f(x) \left(\int_{B_X(x,r)} f(y) \mu_g(dy) \right) \mu_g(dx) 
 + \epsilon \int_X f(x) \mu_g(B_X(x,r)) \mu_g(dx).
\end{align*}\end{linenomath}
Observing that the first term on the righthand side is $H_f(r)$, which is equal to $H_g(r)$, we obtain
\begin{linenomath}\begin{equation*}
    \int_X f(x)^2 \mu_g(B_X(x,r)) \mu_g(dx) < H_g(r) + \epsilon \int_X f(x) \mu_g(B_X(x,r)) \mu_g(dx).
\end{equation*}\end{linenomath}
From Lemma \ref{lem:riemannian_distributions}, we can replace both $\mu_g(B_X(x,r)) = h_g(x,r)$ and $H_g(r)$ with $\frac{\omega_d(r)}{\mathrm{Vol}_g(X)} + O(r^2)$, resulting in the estimate
\begin{linenomath}\begin{equation*}
\int_X f(x)^2 \mu_g(dx) < 1+\epsilon + O(r^2).
\end{equation*}\end{linenomath}
Taking $r \rightarrow 0$ proves the claim and completes the proof.
\end{proof}

\subsubsection*{Densities on the Circle}

In the case that $X$ is the unit circle $\mathbb{S}^1$, we are able to use tools from Fourier analysis to understand the space of densities in more detail. In the following, let $d_{\mathbb{S}^1}$ denote geodesic (arclength) distance on $\mathbb{S}^1$ and let $\mu_{\mathbb{S}^1}$ denote normalized arclength measure. For a density $f$ with respect to $\mu_{\mathbb{S}^1}$, we use the simplified notation $H_f$ for the global distance distribution of the mm-space $(\mathbb{S}^1,d_{\mathbb{S}^1},f\mu_{\mathbb{S}^1})$. We parameterize $\mathbb{S}^1$ with parameter $s \in [0,2\pi)$ (that is, $\mathbb{S}^1 = \{e^{is} \mid s \in [0,2\pi)\}$) and use additive notation for the Lie group structure on $\mathbb{S}^1$.

\begin{proposition}\label{prop:circle_densities}
Let $f_1$ and $f_2$ be two densities with respect to $\mu_{\mathbb{S}^1}$. The global distance distributions $H_{f_1}$ and $H_{f_2}$ agree if and only if $f_1$ and $f_2$ have the same autocorrelation functions:
\begin{linenomath}\begin{equation*}
\int_{\mathbb{S}^1} f_1(s)f_1(s+t) \mu_{\mathbb{S}^1}(ds) = \int_{\mathbb{S}^1} f_2(s)f_2(s+t) \mu_{\mathbb{S}^1}(ds)
\end{equation*}\end{linenomath}
for all $t$.
\end{proposition}

\begin{proof}
For $r \geq \pi$, $H_f(r) = 1$ for any distribution $f$. For $r < \pi$, 
\begin{linenomath}\begin{equation*}
H_{f_j}(r) = \int_{\mathbb{S}^1} \left( \int_{B_{\mathbb{S}^1}(s,r)} f_j(t) \mu_{\mathbb{S}^1}(dt) \right) f_j(s) \mu_{\mathbb{S}^1}(ds)  = \int_{\mathbb{S}^1} \left( \int_{s-r}^{s+r} f_j(t) \mu_{\mathbb{S}^1}(dt) \right) f_j(s) \mu_{\mathbb{S}^1}(ds).
\end{equation*}\end{linenomath}
It follows that 
\begin{linenomath}\begin{equation*}
\frac{d}{dr} H_{f_j}(r) = \int_{\mathbb{S}^1} f_j(s) \left(f_j(s+r) + f_j(s-r)\right) \mu_{\mathbb{S}^1}(ds) = 2 \int_{\mathbb{S}^1} f_j(s) f_j(s+r) \mu_{\mathbb{S}^1}(ds).
\end{equation*}\end{linenomath}
Since $H_f(0) = 0$ for any $f$, $H_{f_1} = H_{f_2}$ if and only if $\frac{d}{dr}H_{f_1} = \frac{d}{dr}H_{f_2}$, so this completes the proof.
\end{proof}

\begin{Corollary}\label{cor:circle_densities}
Let $f_1$ and $f_2$ be densities with respect to $\mu_{\mathbb{S}^1}$ with Fourier expansions
\begin{linenomath}\begin{equation*}
f_1(s) = \sum_n a_n e^{ins}, \qquad f_2(s) = \sum_n b_n e^{ins}
\end{equation*}\end{linenomath}
Then $H_{f_1} = H_{f_2}$ if and only if $|a_n| = |b_n|$ for all $n$.
\end{Corollary}

\begin{proof}
Let $q_j$ denote the autocorrelation function for $f_j$; that is,
\begin{linenomath}\begin{equation*}
q_j(t) = \int_{\mathbb{S}^1} f_j(s) f_j(s+t) \mu_{\mathbb{S}^1}(ds).
\end{equation*}\end{linenomath}
By Proposition \ref{prop:circle_densities}, $H_{f_1} = H_{f_2}$ if and only if $q_1 = q_2$. Letting $\hat{q}_j$ denote the Fourier transform of $q_j$, we have that $q_1 = q_2$ if and only if $\hat{q}_1(n) = \hat{q}_2(n)$ for all $n$, which in turn holds if and only if $|a_n| = |b_n|$ for all $n$.
\end{proof}

\pgfplotsset{width=6cm,compat=1.14}
\usepgfplotslibrary{polar}

\begin{figure}
\centering
\begin{tikzpicture}
	\begin{polaraxis}[xticklabels = \empty, yticklabels=\empty]
	\addplot+[mark=none,domain=0:720,samples=600] 
		{(1+.1 * cos(x)  + .1*cos(3*x) + .1*cos(5*x))/(2*pi)}; 
	\addplot+[mark=none,domain=0:720,samples=600] 
		{1/(2*pi)}; 
	\end{polaraxis}
	\end{tikzpicture}
	\begin{tikzpicture}
	\begin{polaraxis}[xticklabels = \empty, yticklabels=\empty]
	\addplot+[mark=none,domain=0:720,samples=600] 
	{(1+.1 * cos(x+90)  + .1*cos(3*x) + .1*cos(5*x) )/(2*pi)}; 
	\addplot+[mark=none,domain=0:720,samples=600] 
		{1/(2*pi)}; 
	\end{polaraxis}
	\end{tikzpicture}
	\begin{tikzpicture}
	\begin{polaraxis}[xticklabels = \empty, yticklabels=\empty]
	\addplot+[mark=none,domain=0:720,samples=600] 
	{(1+.1 * cos(x+180) + .1*cos(3*x) + .1*cos(5*x) )/(2*pi)}; 
	\addplot+[mark=none,domain=0:720,samples=600] 
		{1/(2*pi)}; 
	\end{polaraxis}
	\end{tikzpicture}
	\caption{Examples from a one-parameter family of densities on the circle $f_\tau$ such that $H_{f_\tau}$ is constant in $\tau$, but no two distinct densities produce isomorphic mm-spaces (see the text for the formula for $f_\tau$ plotted here). The densities plotted are $f_0$, $f_{\frac{\pi}{2}}$ and $f_\pi$, from left to right. In each figure, the blue curve is a graph of the density $f_\tau$ and the red curve is a graph of the constant density, for reference.}\label{fig:circle_densities_example}

	\end{figure}
	
\begin{example}[Global Distance Distribution is Not Locally Injective for Weighted $\mathbb{S}^1$ (\local)]\label{exmp:global_distribution_circle_densities}
We now show how to construct a one-parameter family of probability densities $f_\tau$ on $\mathbb{S}^1$ such that $H_{f_\tau}$ is constant in $\tau$ but so that the corresponding mm-spaces are not isomorphic for any distinct parameter pairs. Let $f$ be any density with Fourier expansion $f(s) = \sum_n a_n e^{ins}$. Because $f$ is a real-valued density, $a_0 = 1$ and $a_{-n} = \overline{a_n}$ (complex conjugate of $a_n$). Suppose that $a_n$ is nonzero for at least two values of $n>0$, say $n_1$ and $n_2$. Our one-parameter family of densities is then defined as
\begin{linenomath}\begin{equation*}
f_\tau(s) = 1 + e^{i\tau} a_{n_1} e^{in_1 s} + e^{-i\tau} a_{-n_1} e^{-in_1 s} + \sum_{n \neq 0,\pm n_1} a_n e^{ins}.
\end{equation*}\end{linenomath}
For every value of $\tau \in [0,2\pi)$, $f_\tau$ is a real-valued density. Moreover, for each $\tau$, the Fourier coefficients of $f_\tau$ agree in magnitude with the Fourier coefficients of $f_0$, by construction. It follows from Corollary \ref{cor:circle_densities} that $H_{f_\tau}$ is constant in $\tau$. Finally, we claim that the mm-spaces $(\mathbb{S}^1,d_{\mathbb{S}^1},f_{\tau_1} \mu_{\mathbb{S}^1})$ and $(\mathbb{S}^1,d_{\mathbb{S}^1},f_{\tau_2} \mu_{\mathbb{S}^1})$ are not isomorphic for any $\tau_1 \neq \tau_2$ in $[0,2\pi)$. Indeed, such an isomorphism would consist of an isometry $\phi$ of $\mathbb{S}^1$ satisfying $f_{\tau_2} \circ \phi^{-1} = f_{\tau_1}$. If $\phi$ is a rotation, we can parametrically express it as  $\phi^{-1}(s) = s + \theta$ for some $\theta \in [0,2\pi)$. Comparing terms of $f_{\tau_2} \circ \phi^{-1}(s)$ and $f_{\tau_1}(s)$, one is easily able to deduce from the assumption that there exist distinct nonzero  coefficients $a_{n_1}$ and $a_{n_2}$ that $\tau_1 = \tau_2$. A similar analysis works when $\phi$ is a reflection. 

Examples from a specific one-parameter family are shown in Figure \ref{fig:circle_densities_example}, with
\begin{linenomath}\begin{equation*}
f_\tau(s) = 1 + 0.1 \left( e^{i\tau} e^{is} +  e^{-i\tau} e^{-is} + e^{i3s} + e^{-i3s} + e^{i5s} + e^{-i5s} \right).
\end{equation*}\end{linenomath}
The densities plotted in the figure are $f_0$, $f_{\frac{\pi}{2}}$ and $f_{\pi}$, from left to right.
\end{example}

The following example  exhibits a 2-parameter family of densities on $\mathbb{S}^1$ which can be fully characterized by the global distance distribution.
 \begin{example}\label{exmp:bumpy_circles}  For each $A\in[0,1)$ and $n\in\mathbb{N}$ consider the density function $f_{A,n}:\mathbb{S}^1\rightarrow \R_+$ given by $f_{A,n}(s) := 1+ A\,\sin(n\cdot s)$, for $s\in[0,2\pi)$. Consider the
 mm-space $\mathcal{S}_{A,n} = (\mathbb{S}^1,d_{\mathbb{S}^1},f_{A,n}\cdot \mu_{\mathbb{S}^1})$ (so $\mathcal{S}_{0,n}$ is isomorphic to $\mathbb{S}^1$ for any $n$). By direct calculation we obtain that the autocorrelation function of $f_{A,n}$ is $[0,2\pi)\ni t\mapsto 1+\frac{A^2}{2} \cos(n\cdot t)$. Via Proposition \ref{prop:circle_densities} we immediately conclude that the global distribution of distances is able to discriminate these ``bumpy circles". More precisely, if for some $A,A'\in [0,1)$  and $n,n'\in\mathbb{N}$ it holds that  $H_{\mathcal{S}_{A,n}} = H_{\mathcal{S}_{A',n'}}$, then it must be that  $A = A'$ and (if $A,A' \neq 0$) $n=n'$.
See Figure \ref{fig:bumpy_circles}.
 \end{example}


\pgfplotsset{width=6cm,compat=1.14}
\usepgfplotslibrary{polar}
\begin{figure}
\centering
\begin{tikzpicture}
	\begin{polaraxis}[xticklabels = \empty, yticklabels=\empty]
	\addplot+[mark=none,domain=0:720,samples=600] 
		{(1+.1 * sin(5*x))/(2*pi)}; 
	\addplot+[mark=none,domain=0:720,samples=600] 
		{1/(2*pi)}; 
	\end{polaraxis}
	\end{tikzpicture}
	\begin{tikzpicture}
	\begin{polaraxis}[xticklabels = \empty, yticklabels=\empty]
	\addplot+[mark=none,domain=0:720,samples=600] 
		{(1+.05 * sin(40*x))/(2*pi)}; 
	\addplot+[mark=none,domain=0:720,samples=600] 
		{1/(2*pi)}; 
	\end{polaraxis}
	\end{tikzpicture}
		\caption{Densities $f_{A,n}$ over $\mathbb{S}^1$ from Example \ref{exmp:bumpy_circles}, plotted in blue, with the uniform density plotted in red for reference. The densities shown have parameters $A=0.1$, $n=5$ (left) and $A=0.05$, $n=40$ (right).}\label{fig:bumpy_circles}
	\end{figure}
	

\section{Characterization Results for Metric Graphs}\label{sec:metric_trees}

We now transition to consider the inverse problems of Section \ref{sec:inverse_problems} in categories of compact 1-dimensional stratified spaces called metric graphs. Roughly, a  \emph{metric graph} is constructed by gluing together a collection of intervals  \cite{kuchment2008quantum}.  To define this more precisely, we recall some basic terminology from graph theory. A \emph{(combinatorial) graph} $(V,E)$ consists of a finite set $V$ of \emph{vertices} and a collection of \emph{edges} $E \subset V \times V$. We take the convention that \emph{self-loops} $(v,v)$ are allowed to belong to $E$ and that for each pair of vertices $v,w \in V$, at most one of the pairs $(v,w)$ and $(w,v)$ belongs to $E$. We only consider connected graphs, meaning that for any $v,w \in V$, there is a \emph{path from $v$ to $w$}---that is, a sequence of vertices $v = v_1,v_2,\ldots,v_n = w$ such that for each $i=1,\ldots,n-1$, either $(v_i,v_{i+1})$ or $(v_{i+1},v_i)$ belongs to $E$.  It will sometimes be convenient to forget the order of vertices in an edge $e = (v,w)$, in which case we write $\overline{e} := \{v,w\}$.

Let $(V,E)$ be a fixed combinatorial graph. To construct a metric graph from this data, we associate each edge $e = (v,w)$ with an interval $I_e = [0,\ell_e]$ for some choice of $\ell_e > 0$: if $x_e$ is the parameter for $I_e$, we identify $v$ with $x_e = 0$ and $w$ with $x_e = \ell_e$. The underlying topological space for the metric graph $G$ is the quotient space $\bigcup_{e \in E} I_e/\sim$ under this identification of endpoints with vertices. The metric on $G$ is the \emph{geodesic distance}, defined as follows. For $v,w \in V \subset G$, the distance is given by
\begin{linenomath}\begin{equation*}
d_G(v,w) = \min \left\{\sum_{i=1}^{n-1} \ell_{e_i} \mid v=v_1,\ldots,v_n=w \mbox{ is a path from $v$ to $w$ and $\overline{e_i} = \{v_i,v_{i+1}\}$} \right\}.
\end{equation*}\end{linenomath}
For general points $x,y \in G$, $d_G(x,y)$ is defined to be the length of the shortest path joining $x$ to $y$. Suppose $x \in I_e$ and $y \in I_{e'}$; a shortest path joining $x$ to $y$ can always be realized as a path from $x$ to a vertex in $\overline{e}$, followed by a sequence of paths between vertices, as above, followed by a path from a vertex in $\overline{e}'$ to $y$. The distance is therefore given by
\begin{equation}\label{eqn:geodesic_distance_formula}
d_G(x,y) = \min \{|x-v| + |y-w| + d_G(v,w) \mid x \in I_e, \, v \in \overline{e}, \mbox{ and } y \in I_{e'},\, w \in \overline{e}'\}.
\end{equation}
 Finally, we endow $G$ with a uniform probability measure $\mu_G$---explictly, the interior of each $I_e$ is endowed with Lebesgue measure and this induces a measure on $G$. To turn this into a probability measure, we could then normalize; alternatively we take the simplifying convention that the total length always satisfies $\sum_{e \in E} \ell_e = 1$. Observe that, with these definitions, for any connected measurable subset $A \subset G$ with $x,x' \in A$, we have 
 \begin{equation}\label{eqn:measurable_subset_inequality}
 d_G(x,x') \leq \mu_G(A),
 \end{equation}
 with equality if and only if $A$ is a path realizing \eqref{eqn:geodesic_distance_formula}.
 
A mm-space $\mathcal{G} = (G,d_G,\mu_G)$ which can be constructed as above is called a \emph{metric graph}. If $\mathcal{G}$ can be constructed from a combinatorial graph $(V,E)$, we say that $(V,E)$ is an \emph{underlying combinatorial graph} of $\mathcal{G}$.

\begin{remark}
The assumption that $\sum_{e \in E} \ell_e = 1$ does not lose generality---one can show that, if using  normalized measures rather than this convention, the total length of a metric graph $\mG$ can be recovered from $H_\mG$.
\end{remark}

For each $x \in G$, there exists $\epsilon_0 > 0$ such that for all $\epsilon \in (0,\epsilon_0)$, the set $B_G(x,\epsilon) \setminus \{x\}$ has a constant number of connected components, which we refer to as the \emph{degree} of $x$, denoted $\mathrm{deg}(x)$. Note that if $(V,E)$ is an underlying combinatorial graph for $\mathcal{G}$ and $x \in V \subset G$, then $\mathrm{deg}(x)$ agrees with usual graph-theoretic notion of degree (number of edges containing $x$). For each $k=1,2,\ldots$, let $\mN_k(\mG) \subset G$ denote the set of degree-$k$ points of $\mG$. By the definition of a metric graph, $\mN_k(\mG)$ is finite for all $k \neq 2$ and $\mN_2(\mG)$ is a set of full measure in $\mG$. In analogy with the terminology used in the setting of combinatorial graphs, an element of $\mN_k(\mG)$ for $k \neq 2$ is called a \emph{vertex} of $\mG$. We define $\mN(\mG) := \cup_{k \neq 2} \mN_k(\mG)$ to be the \emph{set of vertices} of $\mG$. A connected component of $G \setminus \mN(\mG)$ is called an \emph{edge} of $\mG$. Each edge $I$ is homeomorphic to an open interval and its closure $\overline{I}$ in $G$ is homeomorphic to a closed interval with \emph{endpoints} belonging to $\mN(\mG)$---intuitively, the edges of $\mG$ are the interiors of the intervals which were glued together in the construction of $G$. Let $\mE(\mG)$ denote the finite set of edges of $\mG$; note that $\mN_2(\mG) = \cup_{I \in \mE(\mG)} I$.

\begin{figure}
    \centering
    \includegraphics[scale=0.5]{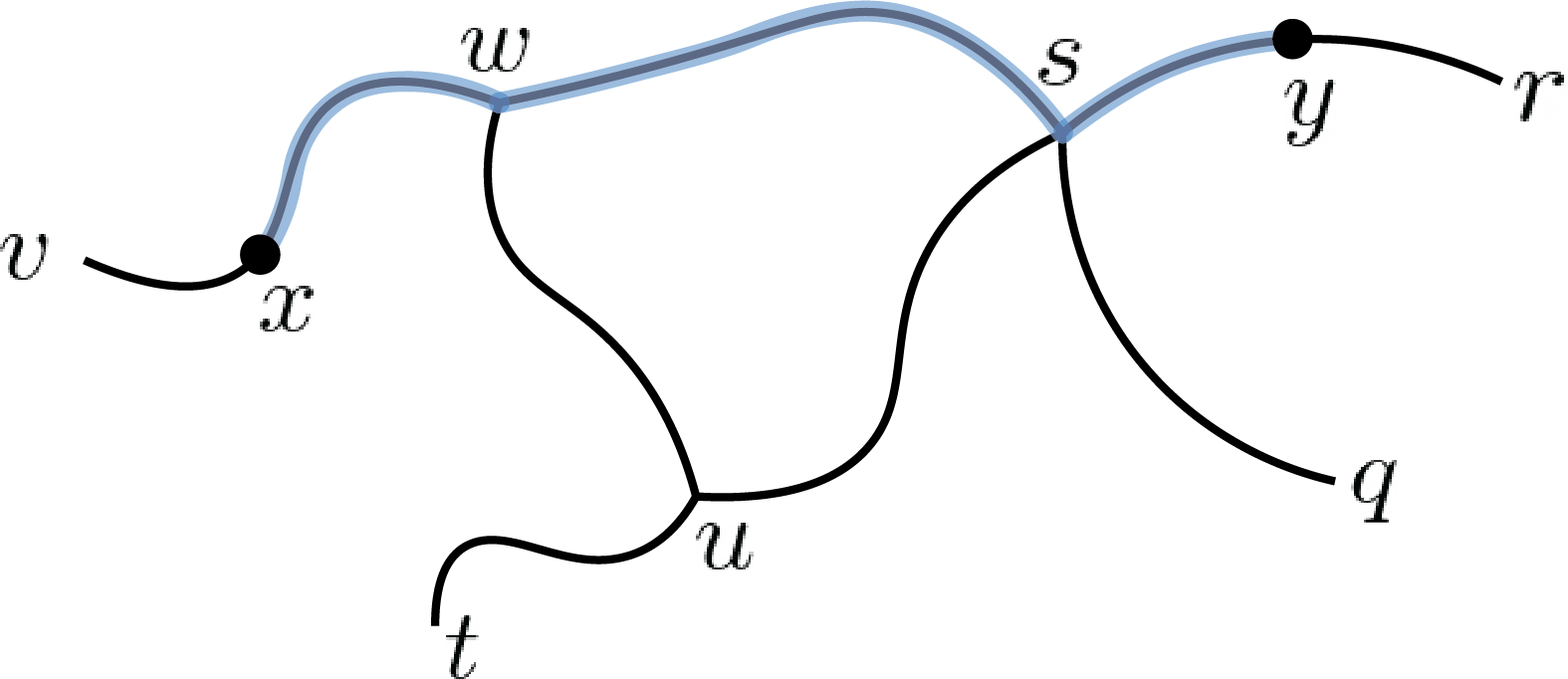}
    \caption{The metric graph from Example \ref{exmp:metricGraph}.}
    \label{fig:metric_graph}
\end{figure}

\begin{example}\label{exmp:metricGraph}
Figure \ref{fig:metric_graph} shows an example of a metric graph $\mG$. An underlying combinatorial graph for $\mG$ is given by $V = \{q,r,s,t,u,v,w\}$ and $E = \{(q,s),(r,s),(s,u),(s,w),(w,v),(w,u),(u,t)\}$. We have, for example, $\mathrm{deg}(r) = 1$, $\mathrm{deg}(x) = 2$ and $\mathrm{deg}(s) = 4$. The vertex set of $\mG$ is $\mN(\mG) = V$ and the edge set $\mE(\mG)$ is the collection of open interval segments joining vertices. The blue highlighted path joining $x$ and $y$ realizes the distance $d_G(x,y) = |x-w| + |y-s| + \ell_{(w,s)}$.
\end{example}

Metric graphs arise naturally in several applications, serving as models for systems of earthquake fault lines, GPS traces of vehicles and stress cracks in materials (see \cite{aanjaneya2012metric}).  Metric trees (i.e., metric graphs which are contractible) are of particular interest, as they appear in geometric group theory as the simplest 0-hyperbolic spaces in the sense of Gromov \cite{gromov1987hyperbolic}, in data science as targets for low-dimensional embeddings of datasets \cite{indyk20178}, in computational anatomy as models for blood vessels \cite{chalopin2001modeling,charnoz2005tree} and in shape analysis as merge trees \cite{gueunet2017task,morozov2013interleaving}. We are particularly interested in our inverse problems in this setting because metric graphs arguably provide the simplest categories of continuous mm-spaces which are not manifolds.

\subsection{Homotopy Type Characterization}

We begin by observing that the global distribution pseudometric $L_{\mathrm{H}}$ cannot distinguish nonisomorphic metric graphs. Indeed, this is illustrated by a simple example in \cite[Figure 2]{martin2008distinguishing}. This also follows from the stronger observation that even the local distribution pseudometric $L_{\mathrm{h}}$ is not able to distinguish nonisomorphic metric graphs, as shown in the following example.

\begin{example}[$L_\mathrm{h}$ Does Not Distinguish Trees (\globalInj)]\label{exmp:non_injectivity_local_dist_c_trees}

Consider the tree structure shown in Figure~\ref{fig:tree_local}, where we are free to assign the number of branches $A_j$, $B_j$ and $C_j$ in each lobe. Let $T_1$ and $T_2$ be the combinatorial trees with numbers of branches given, respectively, by 
\begin{equation}\label{eqn:tree_matrices}
\left(\begin{array}{ccc}
A_1 & A_2 & A_3 \\
B_1 & B_2 & B_3 \\
C_1 & C_2 & C_3 \end{array}\right)= \left(\begin{array}{ccc}
5 & 10 & 5 \\
3 & 3 & 14 \\
1 & 7 & 12 \end{array}\right) \;\; \mbox{ and } \;\; 
\left(\begin{array}{ccc}
A_1 & A_2 & A_3 \\
B_1 & B_2 & B_3 \\
C_1 & C_2 & C_3 \end{array}\right)= \left(\begin{array}{ccc}
5 & 14 & 1 \\
3 & 7 & 10 \\
5 & 3 & 12 \end{array}\right).
\end{equation}
We consider these as metric trees by setting all edge lengths to be equal. The metric trees $\mathcal{T}_1$ and $\mathcal{T}_2$ are nonisomorphic. However, one can check explicitly that $L^{\mathrm{M}}_{\mathrm{h}}(\mT_1,\mT_2)$ (the Monge variant of the local distribution lower bound) is zero by finding a measure preserving mapping realizing this value. One such map is encoded by
\begin{linenomath}\begin{equation*}
\left(\begin{array}{ccc}
A_1 & A_2 & A_3 \\
B_1 & B_2 & B_3 \\
C_1 & C_2 & C_3 \end{array}\right) \mapsto 
\left(\begin{array}{ccc}
A_1 & B_3 & C_1 \\
B_1 & C_2 & A_2 \\
A_3 & B_2 & C_3 \end{array}\right).
\end{equation*}\end{linenomath}
For example, all edges in lobe $A_2$ of $T_1$ are mapped isometrically to the edges in lobe $B_3$ of $T_2$. One can check that this map $\phi$ satisfies $h_{\mT_2}(\phi(x),r) = h_{\mT_1}(x,r)$ for all $r$, with the key point being that the row sums of the matrices in \eqref{eqn:tree_matrices} are all equal.

\begin{figure}
\centering
\begin{overpic}[width=0.40\textwidth,tics=5]{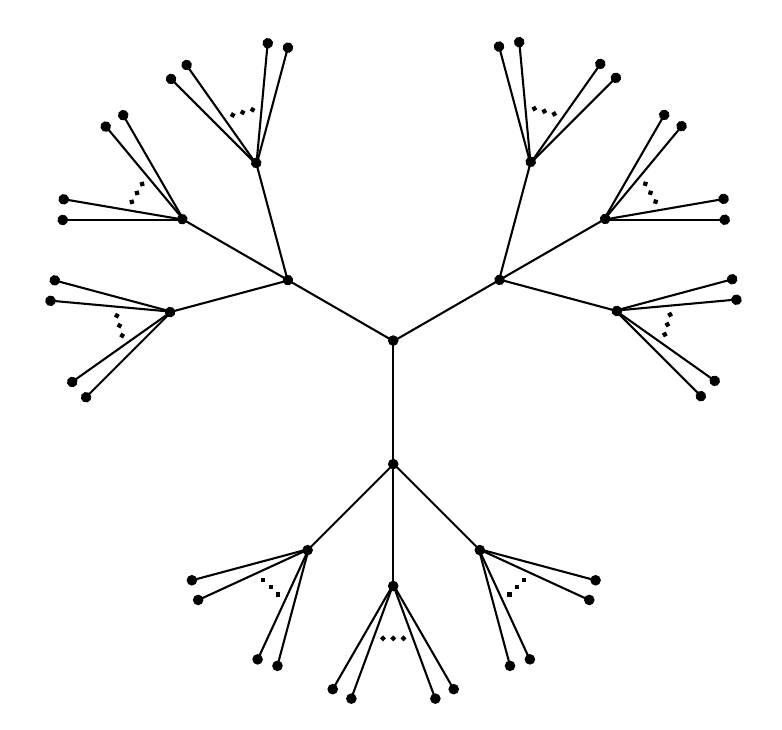}
 \put (0,45) {$A_1$}
    \put (3,71) {$A_2$}
  \put (25,88) {$A_3$}
  
   \put (95,45) {$B_1$}
    \put (92,71) {$B_2$}
  \put (70,88) {$B_3$}
  
     \put (25,5) {$C_1$}
    \put (46,-1) {$C_2$}
  \put (70,5) {$C_3$}

\end{overpic}
\smallskip
\caption{Common tree structure used in Example \ref{exmp:non_injectivity_local_dist_c_trees}.}\label{fig:tree_local}
\end{figure}
\end{example}

Since the metric graphs in Example \ref{exmp:non_injectivity_local_dist_c_trees} are both trees, they are homotopy equivalent. Our first result for this section shows that, in general, $L^{\mathrm{K}}_\mathrm{h}$ is able to detect homotopy type of metric graphs.

\begin{theorem}[Local Distance Distribution Characterizes Homotopy Type of Metric Graphs (\homotopy)]\label{thm:homotopy_type_metric_graphs}
Let $\mG$ and $\mH$ be metric graphs. If $L^{\mathrm{K}}_{\mathrm{h}}(\mG,\mH) = 0$, then $G$ and $H$ are homotopy equivalent. 
\end{theorem}

The proof of Theorem \ref{thm:homotopy_type_metric_graphs} will use a technical lemma. 

\begin{lemma}\label{lem:vertex_cardinalities}
Let $\mG$ and $\mH$ be metric graphs. If $L^{\mathrm{K}}_{\mathrm{h}}(\mG,\mH) = 0$ then $|\mN_k(\mG)| = |\mN_k(\mH)|$ for all $k \neq 2$.
\end{lemma}

Before proving the lemma, we proceed with the proof of Theorem \ref{thm:homotopy_type_metric_graphs}, after a short discussion on metric graph topology. Recall that the \emph{first Betti number} $\beta_1(X)$ of a topological space $X$ is the rank of its first singular homology group \cite{hatcher2002algebraic}. For a metric graph $\mG$ the first Betti number has a simple formula:
\begin{linenomath}\begin{equation}\label{eqn:betti_1}
\beta_1(\mG) = |\mE(\mG)| - |\mN(\mG)| + 1.
\end{equation}\end{linenomath}
The first Betti number of $G$ determines the homotopy type of $G$. Indeed, any connected metric graph $\mG$ is homotopy equivalent to a bouquet of $\beta_1(\mG)$ circles, glued together along a common point---explicitly, such a homotopy equivalence is obtained by choosing a metric tree $T \subset G$ such that $\mN(\mG) \subset T$ (i.e., a spanning tree) and homotoping $T$ to a point. We also have the following formula:
\begin{equation}\label{eqn:sum_of_degrees}
\sum_{v \in \mN(\mG)} \mathrm{deg}(v) = 2|\mE(\mG)|.
\end{equation}
This is well-known for combinatorial graphs (see, e.g., \cite[Page 4]{bollobas1998modern}) and holds for metric graphs by the same reasoning: each edge in $\mE(\mG)$ has exactly 2 vertices associated to it.

\begin{proof}[Proof of Theorem \ref{thm:homotopy_type_metric_graphs}]

Let $\mG$ and $\mH$ be metric graphs with $L^{\mathrm{K}}_{\mathrm{h}}(\mG,\mH) = 0$. Then Lemma \ref{lem:vertex_cardinalities} implies that $|\mN_k(\mG)| = |\mN_k(\mH)|$ for all $k \neq 2$---in particular, $|\mN(\mG)| = |\mN(\mH)|$. Moreover,
\begin{linenomath}\begin{equation*}
2|\mE(\mG)| = \sum_{v \in \mN(\mG)} \mathrm{deg}(v) = \sum_{k\neq 2} k |\mN_k(\mG)| = \sum_{k\neq 2} k|\mN_k(\mH)| = \sum_{w \in \mN(\mH)} \mathrm{deg}(w) = 2|\mE(\mH)|.
\end{equation*}\end{linenomath}
It follows that $\beta_1(\mG) = \beta_1(\mH)$ and the metric graphs are homotopy equivalent.
\end{proof}

\begin{question}
The above results show that the local distance distribution characterizes homotopy type of metric graphs, but is the homotopy type of a metric graph characterized by its global distance distribution? That is, if $L_\mathrm{H}(\mG,\mH) = 0$ for metric graphs $\mG$ and $\mH$, must it be that $G$ and $H$ are homotopy equivalent?
\end{question}

\begin{remark}
We conjecture that the answer to the question is ``yes". This is supported by numerical experiments and potential proof strategies. One can show that the global shape measure $dH_\mG$ is piecewise linear, with jump discontinuities corresponding to geodesic loops---or isometric embeddings of circles---in $G$. We can use this to show that if $\mG$ and $\mH$ both have the property that their geodesic loops give a basis for their respective first homology vector spaces over some field, then $H_\mG = H_\mH$ implies $G$ and $H$ are homotopy equivalent. However, this property is not enjoyed by arbitrary metric graphs, so the general result cannot be obtained from this argument without more work. See Figure \ref{fig:distance_distributions_metric_graphs} for some examples.

An alternative proof strategy would be to construct Riemannian surfaces as tubular neighborhoods of $G$ and $H$ after embedding them into some Euclidean space. These tubular neighborhood surfaces would then have arbitrarily close (by taking the tube radii arbitrarily small) distance distributions with respect to geodesic distance, and one could attempt to apply Corollary \ref{cor:diffeotype_surfaces} to conclude that the surfaces (hence the original metric graphs) have the same Euler characteristic.
\end{remark}

\begin{figure}
    \centering
    \includegraphics[width = 0.8\textwidth]{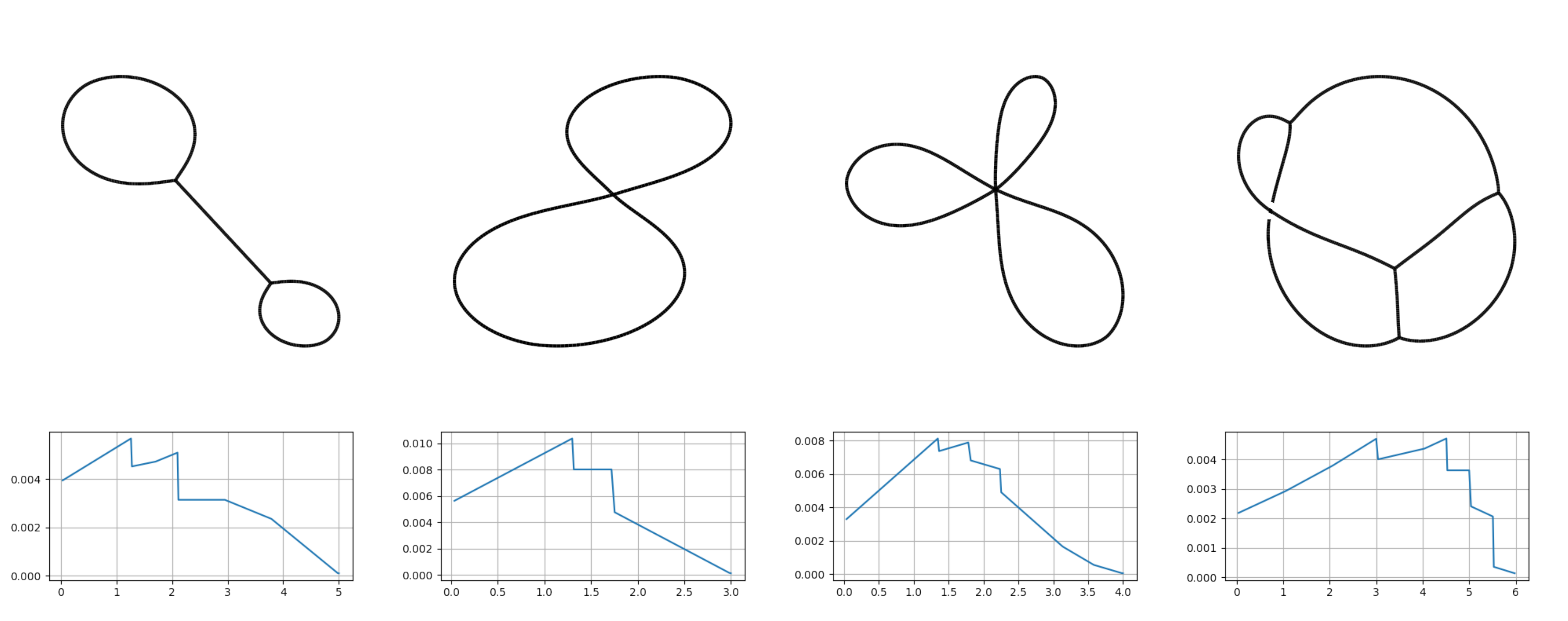}
    \caption{Distance distributions for metric graphs. The top row shows several examples of simple metric graphs and the bottom row shows their global shape measures. Observe that jump discontinuities correspond to geodesic loops in the graphs. The first two examples are homotopy equivalent and their distributions share the same number of jump discontinuities. The last two examples are homotopy equivalent, but have a different number of jump continuities in their distributions: the last example (which is homeomorphic to the 1-skeleton of a tetrahedron) has 4 geodesic loops, even though its first Betti number is 3. We see that counting discontinuities is not enough to show that metric graphs are distinguished up to homotopy type by their global distance distributions.}
    \label{fig:distance_distributions_metric_graphs}
\end{figure}

\subsubsection*{Proof of Lemma \ref{lem:vertex_cardinalities}}

The proof of the lemma will require some preliminary results.

\begin{lemma}\label{lem:wasserstein_bound}
Let $\mX$ and $\mY$ be mm-spaces and let $f:X \rightarrow \R$ and $g:Y \rightarrow \R$ be continuous maps. For each $\mu \in \mM(\mu_X,\mu_Y)$, we have
\begin{linenomath}\begin{equation*}
\int_{X \times Y} \left|f(x) - g(y)\right| \mu(dx \times dy) \geq d_{\mathrm{W},1}^\R(f_\#\mu_X,g_\# \mu_Y).
\end{equation*}\end{linenomath}
\end{lemma}

\begin{proof}
Let $f \times g: X \times Y \rightarrow \R \times \R$ be the product map $(f \times g)(x,y) = (f(x),g(y))$ and consider the measure $(f \times g)_\# \mu \in \mathcal{P}(\R^2)$. It is easy to check that $(f \times g)_\# \mu$ is a measure coupling of $f_\# \mu_X$ and $g_\# \mu_Y$. The change of variables formula then implies that
\begin{linenomath}\begin{equation*}
\int_{X \times Y} \left|f(x) - g(y)\right| \mu(dx \times dy) = \int_{\R \times \R} \left|u-v\right| \big((f \times g)_\# \mu\big)(du \times dv) \geq d_{\mathrm{W},1}^\R(f_\#\mu_X,g_\# \mu_Y).
\end{equation*}\end{linenomath}
\end{proof}

For a mm-space $\mX$ and a fixed $r > 0$, let $h_\mX^r:X \rightarrow \R$ denote the function $h_\mX^r(x) = h_\mX(x,r)$. We will consider the one-parameter family of measures $\{(h_\mX^r)_\# \mu_X\}_{r \geq 0} \subset \mathcal{P}(\R)$. We note that this family of measures appears in the definition of the  \emph{modulus of mass distribution} studied in \cite{greven2009convergence}.

\begin{lemma}\label{lem:equality_of_pushforwards}
For mm-spaces $\mX$ and $\mY$, $L^{\mathrm{K}}_{\mathrm{h}}(\mX,\mY) = 0$ implies that $(h_\mX^r)_\# \mu_X = (h_\mY^r)_\# \mu_Y$ for almost every $r > 0$ (with respect to Lebesgue measure). 
\end{lemma}

\begin{proof}
If $L^{\mathrm{K}}_{\mathrm{h}}(\mX,\mY) = 0$ then for any $\epsilon > 0$ there exists a coupling $\mu \in \mM(\mu_X,\mu_Y)$ such that 
\begin{linenomath}\begin{align}
\epsilon &> \int_{X \times Y} \left(\int_0^\infty \left|h_\mX(x,r) - h_\mY(y,r)\right| \; dr \right) \mu(dx \times dy) \nonumber \\
&= \int_0^\infty \left(\int_{X \times Y} \left|h_\mX^r(x) - h_\mY^r(y)\right| \mu(dx \times dy) \right) \; dr \label{eqn:fubini}\\
&\geq \int_0^\infty d_{\mathrm{W},1}^\R\big((h_\mX^r)_\# \mu_X,(h_\mY^r)_\# \mu_Y\big) \; dr \label{eqn:apply_lemma},
\end{align}\end{linenomath}
where $\eqref{eqn:fubini}$ follows by Fubini's Theorem and \eqref{eqn:apply_lemma} follows from Lemma \ref{lem:wasserstein_bound}. We deduce that 
\begin{linenomath}\begin{equation*}
d_{\mathrm{W},1}^\R((h_\mX^r)_\#\mu_X,(h_\mY^r)_\#\mu_Y) = 0
\end{equation*}\end{linenomath}
for almost every $r > 0$. This completes the proof, since Wasserstein distance is a metric on the space of compactly supported probability measures on $\R$.
\end{proof}

\begin{figure}
    \centering
    \begin{overpic}[unit=1mm,scale=.6]{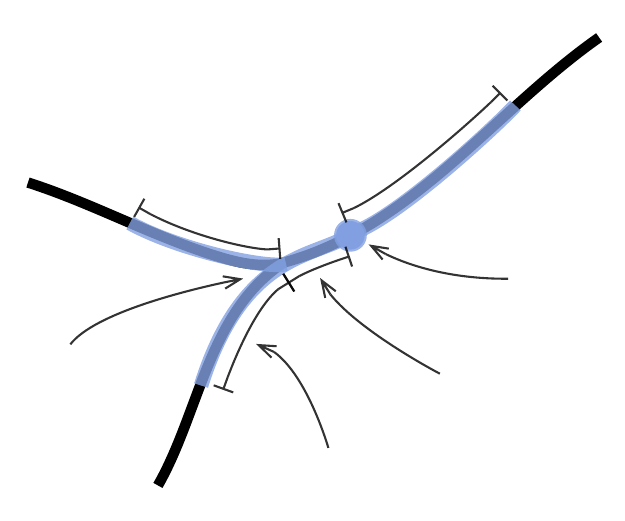}
      \put(8,26){\small $v$}
      \put(83,37){\small $x$}
      \put(72,20){\small $d_G(x,v)$}
      \put(33,47){\small $\delta$}
      \put(53,8){\small $\delta$}
      \put(63,59){\small $r$}
      \end{overpic}
    \caption{The figure shows a portion of a metric graph $G$, a point $x$ in the graph and the metric ball $B_G(x,r)$, highlighted in blue. The point $x$ is within distance $r$ from a vertex $v$. The indicated distances, with $\delta = r - d_G(x,v)$ illustrate the proof of Lemma \ref{lem:degree_formula_for_local_distribution}.}
    \label{fig:local_distribution_illustration}
\end{figure}

\begin{lemma}\label{lem:degree_formula_for_local_distribution}
Let $\mG$ be a metric graph, let $x \in G$ and let $r > 0$. Suppose that there exists exactly one vertex $v$ in $B_G(x,r)$. Then
\begin{linenomath}\begin{equation*}
h_\mG(x,r) = r + d_G(x,v) + \left(\mathrm{deg}(v)-1\right)(r - d_G(x,v)).
\end{equation*}\end{linenomath}
In particular, it holds that
\begin{linenomath}\begin{equation*}
h_\mG(x,r) = \mathrm{deg}(x) \cdot r
\end{equation*}\end{linenomath}
for any $x \in G$, provided $B_G(x,r)\setminus \{x\}$ contains no vertices (see Figure \ref{fig:local_distribution_illustration}).
\end{lemma}

\begin{proof}
The set $B_G(x,r) \setminus \{x,v\}$ is a disjoint union of intervals: $(\mathrm{deg}(v)-1)$ intervals contained in those edges which are incident on $v$ and which which do not contain $x$, each of which has length $r - d_G(x,v)$, one interval of length $d_G(x,v)$ connecting $v$ to $x$, and one interval of length $r$ on the opposite side of $x$ from $v$---see Figure \ref{fig:local_distribution_illustration}. The formula follows: keeping in mind our convention in the construction of $\mu_G$ that all metric graphs have unit total edge length, we have
\begin{linenomath}\begin{equation*}
h_\mG(x,r) = \mu_G(\overline{B_G(x,r)}) = \mu_G(B_G(x,r) \setminus \{x,y\})= r + d_G(x,v) + \left(\mathrm{deg}(v)-1\right)(r - d_G(x,v)).
\end{equation*}\end{linenomath}
The last statement of the lemma holds in the case that $x$ is a vertex by setting $x=v$. On the other hand, if $x$ is not a vertex (i.e. $\mathrm{deg}(x)=2$) and $B_G(x,r)$ contains no vertex then the formula follows by a similar argument.
\end{proof}

We are now prepared to prove the main technical lemma.

\begin{proof}[Proof of Lemma \ref{lem:vertex_cardinalities}]
Suppose that $L^{\mathrm{K}}_{\mathrm{h}}(\mG,\mH) = 0$. Lemma \ref{lem:equality_of_pushforwards} implies that we can choose an $r$ which is less than half the length of the shortest edge in $G$ or $H$ such that $(h_\mG^r)_\# \mu_G = (h_\mH^r)_\# \mu_H$.

We first show that the number of degree-1 points of $G$ can be recovered  from $(h_\mG^r)_\# \mu_G$ by considering
\begin{linenomath}\begin{equation*}
(h_\mG^r)_\# \mu_G ((0,2r)) = \mu_G \left(\{x \in G \mid h_\mG(x,r) < 2r\}\right). 
\end{equation*}\end{linenomath}
By our choice of $r$, an $r$-ball around any point $x \in G$ contains at most one vertex. By Lemma \ref{lem:degree_formula_for_local_distribution}, $h_\mG(x,r) < 2r$ is only possible if $B_G(x,r)$ contains a degree-1 vertex; that is, $h_\mG(x,r) < 2r$ if and only if $d_G(x,\mN_1(\mG)) < r$. It follows that
\begin{linenomath}\begin{equation*}
(h_\mG^r)_\# \mu_G ((0,2r)) =\mu_G \left(\{x \in G \mid d_G(x,\mN_1(\mG)) < r\}\right) = | \mN_1(\mG)| \cdot r.
\end{equation*}\end{linenomath}
Therefore $|\mN_1(\mG)| = |\mN_1(\mH)|$.

A similar strategy works for vertices of higher degree. Let $k_G$ denote the maximum vertex degree of $G$. We claim that for each $k\geq 3$, the quantity $\left(h_\mG^r\right)_\# \mu_G(((k-1)\cdot r, k \cdot r))$ is given by
\begin{equation}\label{eqn:vertex_counting_formula}
\frac{r}{k-2}\cdot  | \mN_k(\mG)| + \frac{r}{(k+1)-2} \cdot |\mN_{k+1}(\mG)| + \cdots + \frac{r}{k_{G}-2} \cdot | \mN_{k_G}(\mG)|,
\end{equation}
when $k \leq k_G$ and that it is equal to zero otherwise. Assuming that the claim holds, the number of vertices of each degree of $G$ can be recovered recursively from $(h_\mG^r)_\# \mu_G$, and this completes the proof. 

It remains to derive \eqref{eqn:vertex_counting_formula}. By definition, 
\begin{linenomath}\begin{equation*}
\left(h_\mG^r\right)_\# \mu_G (((k-1)\cdot r, k \cdot r)) = \mu_G \left(\left\{x \in G \mid (k-1) \cdot r <  \mu_G(B_G(x,r)) < k\cdot r\right\}\right). 
\end{equation*}\end{linenomath}
By our choice of $r$ and Lemma \ref{lem:degree_formula_for_local_distribution}, the maximum value of $\mu_G(B_G(x,r))$ is $r \cdot k_G$, when $x$ is a vertex of degree $k_G$. It follows that $\left(h_\mG^r\right)_\# \mu_G((k-1)\cdot r, k \cdot r) = 0$ when $k \geq k_G + 1$. We then consider $k$ with $3 \leq k \leq k_G$. A  point $x \in G$ satisfies $(k-1) \cdot r < \mu_G(B_G(x,r)) < k\cdot r$ if and only if it satisfies one of the following mutually exclusive conditions, indexed by $\ell = k, k+1, \ldots, k_G$:
\begin{equation}\tag{$C_\ell$}\label{eqn:condition_ell}
\frac{\ell-k}{\ell-2} r < d_G(x, \mN_\ell(\mG)) < \frac{\ell-(k-1)}{\ell-2} r.
\end{equation}
To see this, note that if $\mu_G(B_G(x,r)) > 2r$, then $x$ must lie within distance $r$ from a (unique) vertex $v$ of degree $\ell \geq 3$. Applying Lemma \ref{lem:degree_formula_for_local_distribution} and our assumed bounds on $\mu_G(B_G(x,r))$, we have
\begin{linenomath}\begin{equation*}
(k-1) \cdot r < r + d_G(x,v) + (\ell - 1)(r - d_G(x,v)) <  kr
\end{equation*}\end{linenomath}
and solving for $d_G(x,v)$ shows that condition \eqref{eqn:condition_ell} must hold. The set of points satisfying \eqref{eqn:condition_ell} has measure
\begin{linenomath}\begin{equation*}
| \mN_\ell(\mG)| \cdot \left(\frac{\ell-(k-1)}{\ell-2} r - \frac{\ell-k}{\ell-2} r\right) = | \mN_\ell(\mG)| \cdot \frac{r}{\ell-2}.
\end{equation*}\end{linenomath}
Adding up these measures for $\ell = k, k+1,\ldots,k_G$, we obtain \eqref{eqn:vertex_counting_formula}.
\end{proof}

\subsection{Injectivity of the Local Distance Distribution}\label{sec:continuous_mappings_trees}

In general, $L^{\mathrm{K}}_{\mathrm{h}}$ can only distinguish metric graphs up to homotopy equivalence, but if we restrict to the category whose morphisms are continuous measure-preserving maps then we have the following strengthening, which says that local distance distributions are able to distinguish metric graphs up to isomorphism in this category.

\begin{theorem}[Global Injectivity of Local Distance Distributions for Metric Graphs (\globalInj)]\label{thm:continuous_maps_graphs}
Let $\mathcal{G}$ and $\mathcal{H}$ be metric graphs. Suppose that there exists a continuous measure-preserving map $\phi:G \rightarrow H$ such that
$h_\mathcal{G}(x,r) = h_\mathcal{H}(\phi(x),r)$
for all $x \in G$ and $r \geq 0$. Then $\mG \approx \mH$.
\end{theorem}

We prove the theorem through a sequence of lemmas. Each lemma statement assumes the setup of the theorem.

\begin{lemma}\label{lem:continuous_map_surjective}
The map $\phi$ is surjective.
\end{lemma}

\begin{proof}
Suppose that $y \in H$ does not lie in the image of $\phi$. Since $\phi$ is continuous, it takes $G$ to a compact subset of $H$, which is necessarily closed. Then there is some open neighborhood $U$ of $y$ not contained in the image of $\phi$. This open neighborhood has positive measure, but $\phi^{-1}(U)$ is empty, contradicting the measure-preserving assumption. 
\end{proof}

We define the \emph{cycle graph} $\mathcal{C} = (C,d_{C},\mu_{C})$ to be the metric graph which is isomorphic to the circle of of length one with its geodesic metric and arclength measure. An underlying combinatorial graph for $\mathcal{C}$ is given by $(V = \{v\}, E = \{(v,v)\})$: in other words $(V,E)$ consists of a single point and a self loop. All points in the cycle graph have degree-2; the cycle graph is the unique metric graph with this property.

\begin{lemma}\label{lem:continuous_map_cycle_graph}
Theorem \ref{thm:homotopy_type_metric_graphs} holds in the case that $\mG \approx \mathcal{C}$. 
\end{lemma}

\begin{proof}
Since all points in $G$ are degree-2, it follows from Lemma \ref{lem:degree_formula_for_local_distribution} that all points in $H$ are also degree-2, hence that $\mG \approx \mH$. (We show below in Section \ref{sec:sphere_characterization_metric_graphs} that the cycle graph is characterized by its \emph{global} distance distribution.)
\end{proof}

The main technicality of the proof of Theorem \ref{thm:continuous_maps_graphs} is to establish the following.

\begin{lemma}\label{lem:continuous_map_injective}
If $\mG \not \approx \mathcal{C}$ then $\phi$ is injective.
\end{lemma}

We will prove the lemma after finishing the proof of the theorem.

\begin{proof}[Proof of Theorem \ref{thm:continuous_maps_graphs}]
The case where $\mG \approx \mathcal{C}$ is handled by Lemma \ref{lem:continuous_map_cycle_graph}, so assume $\mG \not \approx \mathcal{C}$. Then Lemmas \ref{lem:continuous_map_surjective} and \ref{lem:continuous_map_injective} imply that $\phi$ is a bijection. Since $G$ and $H$ are uncountable, separable, nonatomic mm-spaces, $\phi^{-1}$ is also measure-preserving \cite[Proposition A.4]{lovasz2012large}

It remains to show that $\phi$ is an isometry. Let $x,x' \in G$ and $J \subset G$ be the image of a path realizing $d_G(x,x')$, as in \ref{eqn:geodesic_distance_formula}. Then
\begin{linenomath}\begin{equation*}
d_H(\phi(x),\phi(x')) \leq \mu_H\left(\phi(J)\right) = \mu_G \left(\phi^{-1} \left( \phi(J)\right)\right) = \mu_G(J) = d_G(x,x'),
\end{equation*}\end{linenomath} 
where the first inequality follows by \eqref{eqn:measurable_subset_inequality}. Therefore $\phi$ is a $1$-Lipschitz map. Running the same argument on $\phi^{-1}$, we see that it is also $1$-Lipschitz. It follows that $\phi$ is an isometry and this completes the proof.
\end{proof}

\subsubsection*{Proof of Lemma \ref{lem:continuous_map_injective}}

\begin{figure}
\begin{center}
    \includegraphics[width = 0.9\textwidth]{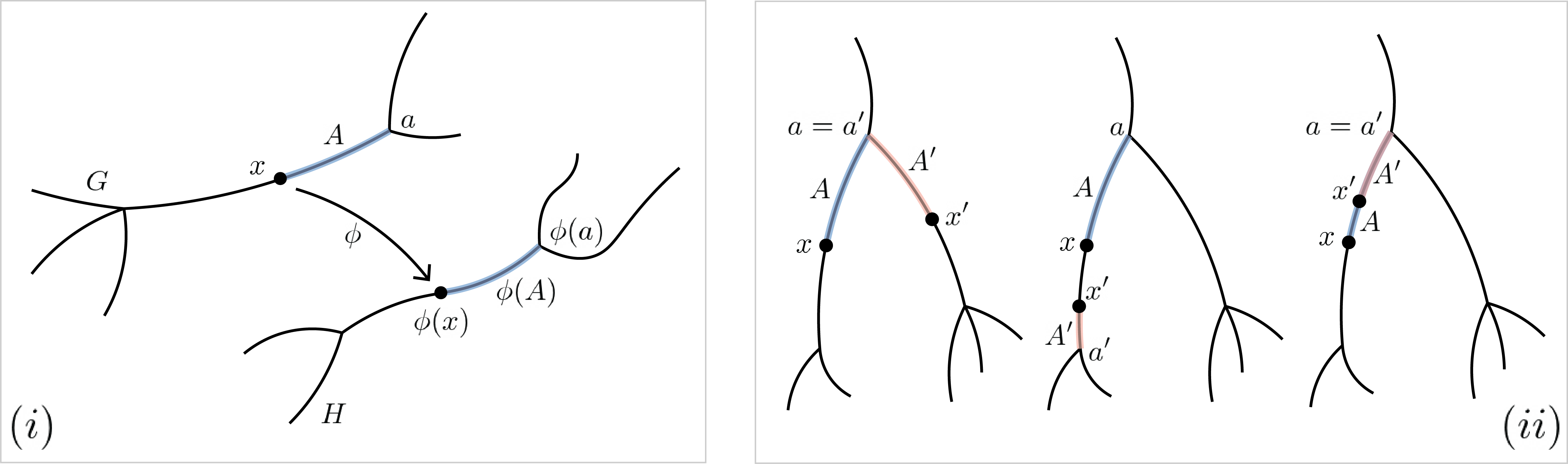}
\end{center}
\caption{Illustrations for the proof of Lemma \ref{lem:continuous_map_injective}. See text for details.}\label{fig:continuous_injectivity}
\end{figure}

The proof of Lemma \ref{lem:continuous_map_injective} will follow from several auxilliary results. For the rest of the subsection, we assume that $\mG$ is not a cycle graph. In particular, $G$ contains a vertex. For a point $x \in \mN_2(\mG)$, let $I \in \mE(\mG)$ be such that $x \in I$, let $\{a,b\} = \overline{I} \cap \mN(\mG)$ be the endpoints of $I$---these endpoints exist by the assumption that $\mG$ is not a cycle graph. The set $I \setminus \{x\}$ consists of  two connected components, each homeomorphic to an open interval. Let the closure of the shorter of these two components be denoted $A$ (if the components are the same length, choose one arbitrarily), and assume without loss of generality that $a \in A$.

\begin{remark}
As a technical point, note that it is possible that $a = b$: even though $\mG \not \approx \mathcal{C}$, it is still possible for $G$ to contain a self-loop, attached at vertex $a$. In this case, we still have $\mathrm{deg}(a) \neq 2$, and the following arguments go through.
\end{remark}

\begin{lemma}\label{lem:interval_image}
Let $x \in \mN_2(\mG)$, $a$ and $A$ as above, and let $y = \phi(x)$. Then $\phi|_A$ is a homeomorphism onto its image $\phi(A)$ (which is then homeomorphic to a closed interval). Moreover, the endpoints of $\phi(A)$ are $y \in \mN_2(\mH)$ and $\phi(a) \in \mN(\mH)$, with $\mathrm{deg}(\phi(a)) = \mathrm{deg}(a)$. For all $z \in A\setminus\{a\}$, $\phi(z) \in \mN_2(\mH)$. (See Figure \ref{fig:continuous_injectivity} (i).)
\end{lemma}

\begin{proof}
For any $z \in A \setminus \{a\}$, $z \in \mN_2(\mG)$ and it follows from Lemma \ref{lem:degree_formula_for_local_distribution} that $\phi(z) \in \mN_2(\mH)$, since $h_\mG(z,r) = h_\mH(\phi(z),r)$. By the same reasoning, $\phi(a)$ is a vertex of $H$ with the same degree as $a$. Next, we observe that $\phi|_A$ is injective: if $\phi(z_1) = \phi(z_2)$ for some $z_1, z_2 \in A$, then 
\begin{linenomath}\begin{equation*}
h_\mathcal{H}(\phi(z_1),r) = h_\mH (\phi(z_2),r ) \; \forall \, r \Rightarrow h_\mathcal{G}(z_1,r) = h_\mG (z_2,r) \; \forall \, r,
\end{equation*}\end{linenomath}
but each point in $A$ has a \emph{unique} local distance distribution. Indeed, this follows because  Lemma \ref{lem:degree_formula_for_local_distribution} says that $h_\mG(z,r) = 2r$ for all $r$ such that $B_G(z,r)$ doesn't contain a vertex; since $a$ is the closest vertex to any $z \in A$ (by our choice of $A$ as the shortest segment joining $x$ to  a vertex), we have $d_G(z,a) = \inf\{r \geq 0 \mid h_{\mG}(z,r) \neq 2r\}$. Now observe that, since $A$ is isometric to a closed interval,  distinct $z_1,z_2 \in A$ have $d_G(z_1,a) \neq d_G(z_2,a)$ (i.e., one of $z_1$ or $z_2$ is closer to $a$ than the other). This implies that $z_1 = z_2$. Then $\phi|_A$ is a bijective map from the compact space $A$ to the Hausdorff space $\phi(A)$ and it follows that it is a homeomorphism \cite[Theorem 26.6]{munkres2014topology}.
\end{proof}

\begin{lemma}\label{lem:interval_measure}
With the above setup, $\mu_G(A) = \mu_H(\phi(A))$.
\end{lemma}

\begin{proof}
Let $\alpha := \mu_G(A)$ and $\gamma := \mu_H(\phi(A))$. The measure-preserving assumption on $\phi$ implies
\begin{linenomath}\begin{equation*}
\gamma = \mu_H\left(\phi(A)\right) = \mu_G\left(\phi^{-1}\left(\phi(A)\right)\right) \geq \mu_G(A) = \alpha.
\end{equation*}\end{linenomath}
On the other hand, we can show that $\gamma > \alpha$ gives a contradiction.  

To derive the contradiction, we first show that $B_H(y,\alpha)$ contains no vertices. Observe that, since $A$ is the shortest edge segment joining $x$ to a vertex, it must be that $B_G(x,\alpha)$ contains no vertices. By Lemma \ref{lem:degree_formula_for_local_distribution}, $h_\mG(x,r) = 2r$ for all $r < \alpha$ and it follows that $h_\mH(y,r) = 2r$ for all $r < \alpha$. Now, to obtain a contradiction, suppose that the nearest vertex $v$ to $y$ is at distance $\epsilon := d_H(y,v) < \alpha$. If $v$ were the unique vertex at distance $\epsilon$ from $y$, Lemma \ref{lem:degree_formula_for_local_distribution} would imply that for all sufficiently small radii $r > 0$,
\begin{linenomath}\begin{equation*}
h_\mH(y,\epsilon + r) = \epsilon + r + \epsilon + (\mathrm{deg}(v)-1)(\epsilon + r -\epsilon) = 2\epsilon + \mathrm{deg}(v)r,
\end{equation*}\end{linenomath}
which is not identically equal to $2r$ (as a function of $r$), giving a contradiction. Since $y$ is degree-2 and there are no vertices at distance less than $\epsilon$ from $y$, the only remaining option is that there are \emph{exactly} two vertices at distance $\epsilon$ from $y$; in particular, with $a$ denoting the endpoint of $A$ as above, it must be that $d_H(y,\phi(a)) = \epsilon$. Then, for any $y' \in \phi(A)$, there is $r < \epsilon$ such that $h_\mH(y',r) \neq 2r$. On the other hand, continuity allows us to choose $x' \in A$ close enough to $x$ such that $h_\mG(x',r) = 2r$ for all $r < \epsilon$, and this gives a contradiction to Lemma \ref{lem:degree_formula_for_local_distribution} by setting $y' = \phi(x')$. 

We now show that $\gamma > \alpha$ gives a contradiction. In this case, choose via continuity a point $z \in A \setminus \{x\}$ such that $0 < \delta:=d_H(y,\phi(z)) < \gamma - \alpha$. Since $A$ is the shortest edge segment between $x$ and a vertex, it must be that $d_G(z,a) < \alpha$, so that the ball $B_G(z,\alpha)$ contains a vertex. By taking $z$ sufficiently close to $x$, we can assume that this vertex is unique. It then follows from Lemma \ref{lem:degree_formula_for_local_distribution} that $h_\mG(z,r) \neq 2r$ for some $r < \alpha$. We will show that $B_H(\phi(z),\alpha)$ does not contain a vertex, which implies that $h_\mH(\phi(z),r) = 2r$ for all $r < \alpha$ and gives a contradiction. Indeed, let $J$ be a path from $\phi(z)$ to a vertex. If $y \in J$, then $J$ has length at least $d_H(y,\mN(\mH)) + \delta$. On the other hand, if $y \not \in J$ then $J$ has length at least $\gamma - \delta$ (the shortest such path is from $\phi(z)$ to $\phi(a)$). That is,
\begin{linenomath}\begin{equation*}
d_H(\phi(z),\mN(\mH)) = \min \{d_H(y,\mN(\mH)) + \delta, \gamma - \delta\}.
\end{equation*}\end{linenomath}
Since $B_H(y,\alpha)$ contains no vertices, by the previous paragraph, we have $d_H(y,\mN(\mH)) > \alpha$. Since $\gamma - \delta > \alpha$, we deduce that $d_H(\phi(z),\mN(\mH)) > \alpha$ and this gives us the desired contradiction.
\end{proof}

\begin{lemma}\label{lem:continuous_maps_injective_degree_2}
If $x,x' \in \mN_2(\mG)$ satisfy $\phi(x) = \phi(x')$ then $x = x'$.
\end{lemma}

\begin{proof}
Suppose, to obtain a contradiction, that there exist distinct points $x,x' \in \mN_2(\mG)$ with $\phi(x) = \phi(x') =: y$. Let $a \in \mN(\mG)$ and $A \subset G$ be as above. Similarly, let $I' \in \mE(\mG)$ with $x' \in I'$, let $a'$ and $b'$ be the endpoints of $I'$, let $A'$ be the closure of the shorter of the two components of $I' \setminus \{x'\}$ and assume that $a' \in A'$. Lemmas \ref{lem:interval_image} and \ref{lem:interval_measure} then imply that $\phi|_{A'}$ is a homeomorphism onto its image $\phi(A')$, which is a closed segment in $H$ with endpoints $y$ and $\phi(a') \in \mN(\mH)$ such that $\mu_H(\phi(A')) = \mu_G(A')$.

We will derive our contradiction in two separate cases: either $\phi(A) = \phi(A')$ or not. 

\smallskip
\noindent {\bf Case 1: $\phi(A) = \phi(A')$.} Let $\alpha:= \mu_G(A)$; in this case, we also have $\mu_{G}(A') = \alpha$, by Lemma \ref{lem:interval_measure}. Then
\begin{linenomath}\begin{align*}
\alpha = \mu_H \left(\phi(A)\right) = \mu_G \left(\phi^{-1} \left(\phi(A)\right)\right) &\geq \mu_G\left(A \cup A'\right) \\
&= \mu_G(A) + \mu_G(A') - \mu_G \left(A \cap A'\right) = 2\alpha - \mu_G \left(A \cap A'\right) > \alpha,
\end{align*}\end{linenomath}
which gives a contradiction once we justify the last inequality. We claim that $\mu_G(A \cap A') < \alpha$. To see this, consider three cases (illustrated in Figure \ref{fig:continuous_injectivity}, (ii)). If $I \neq I'$ (i.e., $x$ and $x'$ do not belong to the same edge), then $A \cap A'$ contains at most one point (this happens if $a = a'$), so we have $\mu_G(A \cap A') = 0$. If $I = I'$, but $A$ and $A'$ are disjoint, then the claim obviously follows. Finally, suppose that $I = I'$ and $A \cap A' \neq \emptyset$. Since $A$ (respectively, $A'$) is the shortest segment connecting $x$ (respectively, $x'$) to a vertex, it must be that $A \subset A'$ or $A' \subset A$---assume without loss of generality that $A' \subset A$. Since $x \neq x'$, $A \setminus A'$ is positive measure, but this is impossible under the assumption that $\mu_G(A) = \mu_G(A')$ and this case is ruled out. This yields a contradiction in the case that $\phi(A) = \phi(A')$.

\smallskip
\noindent {\bf Case 2: $\phi(A) \neq \phi(A')$.} Each of $\phi(A)$ and $\phi(A')$ is a closed segment in $H$ with $y$ as an endpoint and a vertex as the other endpoint. Since $\phi(A) \neq \phi(A')$, it must be that $\phi(A) \cup \phi(A')$ is equal to the closure of the edge containing $y$. Moreover, $\phi(A) \cap \phi(A')$ is measure zero: it is either equal to $\{y\}$ or to $\{y,\phi(a)\}$ (in the case that the edge containing $y$ forms a loop and $\phi(a) = \phi(a')$). We also have that $A \cap A'$ is measure zero---this follows by a case analysis similar to the previous paragraph: if $I \neq I'$ then $A \cap A'$ contains at most one point (if $a = a'$), if $I = I'$ and $A$ and $A'$ are disjoint then the claim follows, and the case that $I = I'$ and $A \cap A' \neq \emptyset$ is ruled out because this would imply that (without loss of generality) $A' \subset A$, hence $\phi(A) \cap \phi(A') = \phi(A)$ does not have measure zero and this contradicts the previous sentence. 

By continuity, there exists $\epsilon > 0$ such that $B:=B_G(x,\epsilon)$ and $B':=B_G(x',\epsilon)$ both get mapped into $\phi(A) \cup \phi(A')$. Indeed, choose a small open set $U$ around $y$ which is contained in the edge $\phi(A) \cup \phi(A')$. Then $x,x' \in \phi^{-1}(U)$, so we can choose $\epsilon$ such that $\phi(B),\phi(B') \subset U$. Since $x \neq x'$, we can also choose $\epsilon$ small enough so that $B \cap B' = \emptyset$. Let $\alpha' := \mu_G(A')$ and assume without loss of generality that $\epsilon < \min\{\alpha,\alpha'\}$. We also have $A \cap B' = \emptyset$. This follows once again by a case analysis: if $I \neq I'$ then $\epsilon < \min\{\alpha,\alpha'\}$ implies $a' \not \in B'$, so that $A \cap B' = \emptyset$; if $I = I'$ but $A$ and $A'$ do not overlap, then the claim follows by taking $\epsilon < d_G(x,x')/2$ and the case that $I=I'$ and $A \cap A' \neq \emptyset$ has already been ruled out. Similarly, $A' \cap B = \emptyset$. Then
\begin{linenomath}\begin{align}
    \alpha + \alpha' &= \mu_H \left(\phi(A) \cup \phi(A')\right) = \mu_G \left(\phi^{-1} \left(\phi(A) \cup \phi(A') \right)\right) \nonumber \\
    &\geq \mu_G(A \cup B \cup A' \cup B') \nonumber \\
    &= \mu_G(A \cup B) + \mu_G(A' \cup B') - \mu_G\left((A \cup B) \cap (A' \cup B')\right) \nonumber \\
    &= \alpha + \epsilon + \alpha' + \epsilon - \mu_G\left((A \cup B) \cap (A' \cup B')\right) \label{eqn:injectivity_proof1}\\
    &= \alpha + \alpha' + 2\epsilon. \label{eqn:injectivity_proof2}
\end{align}\end{linenomath}
This gives a contradiction once we justify the last two steps. Since $A$ is the shortest segment joining $x$ to a vertex and $\epsilon < \alpha = \mu_G(A)$, $B$ is contained entirely in the edge of $x$ and therefore has measure $2\epsilon$. Half of this mass overlaps with the segment $A$, so that $\mu_G(A \cup B) = \alpha + \epsilon$. Likewise, $\mu_G(A' \cup B') = \alpha' + \epsilon$, and this justifies \eqref{eqn:injectivity_proof1}. By the work above, $A \cap A'$ is measure zero and $B \cap B'$, $A \cap B'$ and $A' \cap B$ are all empty. It follows that $(A \cup B) \cap (A' \cup B')$ has measure zero, and this justifies \eqref{eqn:injectivity_proof2}. We have therefore derived a contradiction in this case and it must be that $x = x'$. 
\end{proof}

We are now prepared to prove the main lemma.

\begin{proof}[Proof of Lemma \ref{lem:continuous_map_injective}]
To obtain a contradiction, suppose that $x,x'$ are distinct points in $G$ which map to a common point $y \in H$. It follows from Lemma \ref{lem:degree_formula_for_local_distribution} that $\mathrm{deg}(x)=\mathrm{deg}(x')=\mathrm{deg}(y)$. The case where this common degree is 2 is handled by Lemma \ref{lem:continuous_maps_injective_degree_2}, so suppose that $x$, $x'$ and $y$ are all vertices of the same degree. We claim that, by continuity, we can reduce to the degree-2 case. Consider an open ball $B_H(y,r)$, with $r$ small enough so that $B_H(y,r) \setminus \{y\}$ is a union of open intervals $I_1,\ldots,I_{\mathrm{deg}(y)}$---we identify each $I_i$ with $(0,r)$, so that $B_H(y,r)$ is the union of $\mathrm{deg}(y)$ half open intervals $[0,r)$, glued together along $0$. Choose $s > 0$ small enough so that $B_G(x,s) \subset \phi^{-1}(B_H(y,r))$ and $B_G(x,s) \setminus \{x\}$ is a union of open intervals $J_1,\ldots,J_{\mathrm{deg}(x)}$---we similarly identify each $J_j$ with $(0,s)$ so that $B_G(x,s)$ is the union of $\mathrm{deg}(x) = \mathrm{deg}(y)$ half open intervals $[0,s)$, glued together at $0$. Since $\phi$ is continuous, the image of each $J_j$ is connected, and since $\phi$ preserves degree (once again, by Lemma \ref{lem:degree_formula_for_local_distribution}), $\phi(J_j)$ it is contained in some edge interval $I_i$. Moreover, $\phi(x) = y$ implies that the image of $J_j$ contains some interval $(0,r_j) \subset (0,r) \approx I_i$. If there are distinct edge intervals $J_j$ and $J_k$ such that $\phi(J_j)$ and $\phi(J_k)$ map into the same $I_i$, then we are done---for any $y_0 \in I_i$ with $d_H(y_0,y) < \min\{r_j,r_k\}$, $y_0 \in \phi(J_j) \cap \phi(J_k)$, so there exist degree-2 points $x_0 \in J_j$ and $x_0' \in J_k$ such that $\phi(x_0) = \phi(x_0') = y_0$. Supposing otherwise (i.e., $\phi(J_j) \cap \phi(J_k) = \emptyset$ for all $j \neq k$), the same arguments apply to a small ball $B_G(x',s)$ (shrinking $s$, if necessary). Let $J_1',\ldots,J_{\mathrm{deg}(x)}'$ denote the edge intervals in the corresponding decomposition of $B_G(x',s) \setminus \{x'\}$. By the pigenonhole principle, there must be some $J_j$ and $J_k'$ such that $\phi(J_j)$ and $\phi(J_k')$ are subsets of a common edge interval $I_i$. Once again, this implies that there exist degree-2 points $x_0 \in J_j$, $x_0 \in J_k'$ and $y_0 \in I_i$ such that $\phi(x_0) = \phi(x_0') = y_0$. We have therefore established that the argument reduces to the case of Lemma \ref{lem:continuous_maps_injective_degree_2}.
\end{proof}

\subsection{Sphere Characterization for Metric Graphs}\label{sec:sphere_characterization_metric_graphs}

In this subsection we address the sphere characterization inverse problem (\spheres) in the setting of metric graphs. In this category, the appropriate notion of a sphere is the \emph{cycle graph} $\mathcal{C}$---see the definition above, prior to Lemma \ref{lem:continuous_map_cycle_graph}.

\begin{theorem}[Sphere Characterization for Metric Graphs (\spheres)]\label{thm:sphere_characterization_metric_graphs}
If a metric graph $\mG$ has the same global distance distribution as the cycle graph then it is isomorphic to the cycle graph.
\end{theorem}

\begin{remark}
Sturm shows in \cite[Proposition 8.5]{sturm2012space} that (using our terminology) the cycle graph is distinguished from all other ``balanced" length spaces by its \emph{local} distance distribution. The term \emph{balanced} means that the local distance distribution of the length space is basepoint independent. The theorem (and its corollary) above  considers a more restrictive category (metric graphs versus length spaces), but does not assume the balanced condition and moreover uses the much weaker assumption of equality of \emph{global} distance distributions.
\end{remark}

The proof of the theorem will use some extra terminology and preliminary results. We define a combinatorial invariant of $\mG$ by
\begin{linenomath}\begin{equation*}
q(\mG) := \left\{\begin{array}{cc}
\sum_{v \in \mN(\mG)} \mathrm{deg}(v)^2 - 4|\mE(\mG)| & \mbox{ if $\mathcal{G} \not \approx \mathcal{C}$} \\
0 & \mbox{ if $\mathcal{G} \approx \mathcal{C}$}.
\end{array}\right.
\end{equation*}\end{linenomath}
The necessity of the cases in the definition comes from the fact that $\mathcal{C}$ is the unique metric graph with $\mN(\mG) = \emptyset$.

\begin{lemma}\label{lem:distance_distribution_graph}
Let $\mG$ be a metric graph with shortest edge length $r_0$. For $r < \frac{r_0}{2}$, the global distance distribution of $\mG$ is given by
\begin{linenomath}\begin{equation*}
H_\mG(r) = 2r + \frac{r^2}{2} \cdot q(\mG).
\end{equation*}\end{linenomath}
\end{lemma}

\begin{proof}
The proof is a computation. The formula is straightforward if $\mG$ is a cycle graph, so assume not. If $x \in G$ lies within distance $r$ of a vertex $v \in \mN(\mG)$, then this vertex is unique, so Lemma \ref{lem:degree_formula_for_local_distribution} implies
\begin{linenomath}\begin{equation*}
h_\mG(x,r) = r + d_G(x,v) + (r-d_G(x,v))(\mathrm{deg}(v) - 1) = 2 d_G(x,v) + (r - d_G(x,v)) \mathrm{deg}(v)
\end{equation*}\end{linenomath}
and otherwise $h_\mG(x,r) = 2r$. Therefore
\begin{linenomath}\begin{align}
H_\mG(r) &= \int_G h_\mG(x,r) \,  \mu_G(dx) \nonumber \\
&= \int_{d_G(x,\mN(\mG)) > r} h_\mG(x,r) \,  \mu_G(dx) + \sum_{v \in \mN(\mG)} \int_{d_G(x,v) \leq r} h_\mG(x,r) \,  \mu_G(dx) \nonumber \\
&= \int_{d_G(x,\mN(\mG)) > r} 2r  \,  \mu_G(dx) + \sum_{v \in \mN(\mG)} \int_{d_G(x,v) \leq r} 2 d_G(x,v) + (r - d_G(x,v)) \mathrm{deg}(v)  \,  \mu_G(dx) \nonumber \\
&= 2r\left(1 - r \sum_{v \in \mN(\mG)} \mathrm{deg}(v)\right) + \sum_{v \in \mN(\mG)} \int_0^r \left(  2 \rho + (r - \rho) \mathrm{deg}(v)  \right)\mathrm{deg}(v) \, d\rho \label{eqn:distance_distribution_graph_1}  \\
&= 2r - 4 r^2 |\mE(\mG)| +  \sum_{v \in \mN(\mG)} \left(r^2 \mathrm{deg}(v) + \frac{r^2}{2} \mathrm{deg}(v)^2\right) \label{eqn:distance_distribution_graph_2} \\
&= 2r + \frac{r^2}{2} \left(\sum_{v \in \mN(\mG)} \mathrm{deg}(v)^2 - 4|\mE(\mG)| \right), \label{eqn:distance_distribution_graph_3}
\end{align}\end{linenomath}
where the first term in \eqref{eqn:distance_distribution_graph_1} comes from the total length of the set $\{x \in G \mid d_G(x,\mN(\mG)) \leq r\}$ and the second term comes from the change of variables $\rho = d_G(x,v)$ for each fixed $v$; this allows us to rewrite the integral in the second term as an integral over each segment emanating from the vertex $v$. The first term in \eqref{eqn:distance_distribution_graph_2} comes from \eqref{eqn:sum_of_degrees}. Finally, \eqref{eqn:distance_distribution_graph_3} follows by applying \eqref{eqn:sum_of_degrees} again and simplifying.
\end{proof}

We now focus on the case of metric trees (i.e., contractible metric graphs). There are two families of metric trees which play a special role. A \emph{line graph} is a metric tree which is isomorphic to a closed interval. A \emph{$Y$-graph} is a metric tree which is shaped like a ``Y"; that is, which can be constructed from the underlying combinatorial graph with data $V = \{u,v,w,x\}$ and  $E = \{(u,x),(v,x),(w,x)\}$.

\begin{lemma}\label{lem:metric_tree_q}
For a metric tree $\mT$, $q(\mT) = 0$ if and only if $\mT$ is a $Y$-graph.
\end{lemma}

We prove the lemma below, but let us first finish the proof of the theorem.

\begin{proof}[Proof of Theorem \ref{thm:sphere_characterization_metric_graphs}]

The proof follows by checking two cases, which depend on the concept of Betti number \eqref{eqn:betti_1}. First suppose that $\beta_1(\mG) \geq 1$. If $H_\mG = H_{\mathcal{C}}$ then Lemma \ref{lem:distance_distribution_graph} implies that $q(\mG) = q(\mathcal{C}) = 0$. If it were the case that $\mG \not \approx \mathcal{C}$ then we would have
\begin{linenomath}\begin{align}
0 &= q(\mG) = \sum_{v \in \mN(\mG)} \mathrm{deg}(v)^2 - 4|\mE(\mG)| \geq \frac{\left(\sum_{v \in \mN(\mG)} \mathrm{deg}(v)\right)^2}{|\mN(\mG)|} - 4|\mE(\mG)| \label{eqn:betti_geq_1_1} \\
&= \frac{4 |\mE(\mG)|^2}{|\mN(\mG)|} - 4|\mE(\mG)| = 4|\mE(\mG)|\left(\frac{|\mE(\mG)|}{|\mN(\mG)|} - 1\right) \geq 0, \label{eqn:betti_geq_1_2}
\end{align}\end{linenomath}
where \eqref{eqn:betti_geq_1_1} follows by Cauchy-Schwartz and \eqref{eqn:betti_geq_1_2} follows from \eqref{eqn:sum_of_degrees} and the assumption that $\beta_1(\mG) \geq 1$, hence $|\mE(\mG)| \geq |\mN(\mG)|$. Equality in Cauchy-Schwartz is then forced, and we conclude that the sequence of vertex degrees for $G$ must be constant---say, $\mathrm{deg}(v) = k$ for all $v \in \mN(\mG)$. Then 
\begin{linenomath}\begin{equation*}
0 = q(\mG) = \sum_{v \in \mN(\mG)} \mathrm{deg}(v)^2 - 4|\mE(\mG)| = |\mN(\mG)|k^2 - 2 |\mN(\mG)|k,
\end{equation*}\end{linenomath}
where we have once again applied \eqref{eqn:sum_of_degrees}. Solving for $k$ yields $k=2$, a contradiction. Therefore $\mG \approx \mathcal{C}$.

Now suppose that $H_\mG = H_{\mathcal{C}}$ and $\beta_1(\mG) = 0$. Then $q(\mG) = q(\mathcal{C}) = 0$ and $\mathcal{G}$ is a metric tree, so Lemma \ref{lem:metric_tree_q} implies that $\mG$ is a $Y$-graph. Since $H_\mG = H_\mathcal{C}$, $G$ and $C$ must have the same diameter---namely, $\mathrm{diam}(G) = \mathrm{diam}(C) = \frac{1}{2}$. We claim that there is no $Y$-graph with diameter $\frac{1}{2}$. Let $\mN_1(\mG) = \{u,v,w\}$ and $\mN_3(\mG) = \{x\}$. Then 
\begin{linenomath}\begin{equation*}
\frac{1}{2} = \mathrm{diam}(G) \geq d_G(u,x) + \max\{d_G(v,x),d_G(w,x)\} \geq d_G(u,x) + \frac{d_G(v,w)}{2}.
\end{equation*}\end{linenomath}
Now we use the assumption that $\mG$ has unit total length, meaning that $d_G(u,x) + d(v,x) + d(w,x) = 1$. Putting this together with the fact that $d_G(v,w) = d_G(v,x) + d_G(w,x)$ and the computation above, we have
\begin{linenomath}\begin{equation*}
1 \geq 2 d_G(u,x) + d_G(v,w) = d_G(u,x) + d_G(u,x) + d_G(v,x) + d_G(w,x) = d_G(u,x) + 1.
\end{equation*}\end{linenomath}
This implies $d_G(u,x) = 0$, which is a contradiction.
\end{proof}

\begin{proof}[Proof of Lemma \ref{lem:metric_tree_q}]
One can show by direct computation that if $\mT$ is a line graph then $q(\mT) = -2$ and that if $\mT$ is a $Y$-graph then $q(\mT) = 0$. We will next show that $q(\mT) \geq -2$ for any metric tree $\mT$. We prove this claim by induction on $|\mN(\mT)|$. The base case is $n=2$, where the only metric trees are line graphs and we are done. Suppose that the claim holds for all metric trees with $n-1$ or fewer vertices and let $\mT$ be a metric tree with $n > 2$ vertices. Choose a $v \in \mN_1(\mT)$ (i.e., a vertex of degree-1) and let $I \subset T$ be the edge with $v$ as an endpoint (i.e., $v \in \overline{I}$). Let $w$ denote the other endpoint of $I$; note that $\mathrm{deg}_T(w) \geq 3$, since $\mT$ is not a line graph (by the assumption that $n > 2$). Let $\mT'$ be the metric tree whose underlying set is $T' = T \setminus (I \cup \{v\})$, with metric renormalized so that $T'$ has unit total edge length---intuitively, $T'$ is obtained from $T$ by pruning off the edge $I$.

We have
\begin{linenomath}\begin{equation}\label{eqn:pruning_formula_0}
q(\mT) = q(\mT') + 2 \mathrm{deg}_T(w) - 4,
\end{equation}\end{linenomath}
where we use $\mathrm{deg}_T(w)$ to emphasize that the degree of $w$ is computed with respect to $T$ in this formula. Indeed, if $\mathrm{deg}_T(w) > 3$, then $w \in \mN(\mT')$ and 
\begin{equation}\label{eqn:pruning_formula}
q(\mT)  = \sum_{u \in \mN(\mT) \setminus \{v,w\}} \mathrm{deg}_T(u)^2 + \mathrm{deg}_T(v)^2 + \mathrm{deg}_T(w)^2 - 4 |\mE(\mT)|.
\end{equation}
The claim follows by observing that $|\mE(\mT)| = |\mE(\mT')| + 1$, that $\mathrm{deg}_T(v) = 1$, that $\mathrm{deg}_T(u) = \mathrm{deg}_{T'}(u)$ for all $u \in \mN(\mT') \setminus \{w\}$, and finally that $\mathrm{deg}_T(w) = \mathrm{deg}_{T'}(w) + 1$, hence
\begin{linenomath}\begin{equation*}
\mathrm{deg}_T(w)^2 = \mathrm{deg}_{T'}(w)^2 + 2\mathrm{deg}_T(w) - 1.
\end{equation*}\end{linenomath}
Plugging this into \eqref{eqn:pruning_formula} and doing some algebra yields \eqref{eqn:pruning_formula_0}. On the other hand, if $\mathrm{deg}_T(w) = 3$, so that $w \not \in \mN(\mT')$, then \eqref{eqn:pruning_formula} reduces to the desired formula more directly by observing that $|\mE(\mT)| = |\mE(\mT')| + 2$.

The new tree $\mT'$ has at most $(n-1)$-vertices, so the inductive hypothesis gives
\begin{equation}\label{eqn:q_computation_Y_line}
q(\mT) = q(\mT') + 2 \mathrm{deg}_T(w) - 4 \geq -2 + 2 \cdot 3 - 4 = 0.
\end{equation}
This proves the claim, but also gives us more. Indeed, it shows that if $\mT$ is not a line graph then $q(\mT) \geq 0$. Moreover, if $\mT$ is not a $Y$-graph, then $\mT'$ is not a line graph, so that $q(\mT') \geq 0$ and the inequality in \eqref{eqn:q_computation_Y_line} becomes strict. This proves the lemma.
 
\end{proof}

\subsection{Local Injectivity for Metric Trees}\label{sec:local_and_generic_trees}

The last of our main results considers the question of whether $L^{\mathrm{K}}_{\mathrm{h}}$ characterizes metric graphs up to isomorphism in the local sense (\local). We are able to obtain such a result when restricting to the subcategory of metric trees, where we have some extra tools at hand coming from the computational topology literature.

\begin{theorem}[Local Characterization for Metric Trees (\local)]\label{thm:main_theorem_mtrees}
For every metric tree $\mT$, there exists $\epsilon_\mT > 0$ such that if a metric tree $\mS$ satisfies $\dgh(\mT,\mS) < \epsilon_\mT$ and $L^{\mathrm{K}}_{\mathrm{h}}(\mT,\mS) = 0$ then $\mS$ is isomorphic to $\mT$.
\end{theorem}

Let $\mT$ be a metric tree and fix some point $x \in T$. Let $\partial B_{T}(x,r):=\{y \in T \mid d_T(x,y) = r\}$. Observe that if $\partial B_T(x,r)$ does not contain a vertex, then for sufficiently small $\epsilon > 0$, 
\begin{linenomath}\begin{equation*}
h_\mT(x,r + \epsilon) = h_\mT(x,r) + \epsilon \cdot |\partial B_T(x,r)|.
\end{equation*}\end{linenomath}
It follows that, for fixed $x$, $h_\mT(x,r)$ is piecewise linear with a finite collection of radii at which the slope can change; denote these radii as $r_1^x < r_2^x < \cdots < r_{M_x}^x$. These radii form a subset of $\{d_T(x,v) \mid v \in \mN(\mT)\}$.

In what follows, we borrow terminology from graph theory and refer to a vertex of degree-1 $v \in \mN_1(\mT)$ as a \emph{leaf}. We use the term \emph{leaf edge} to refer to an edge in $\mathcal{T}$ containing a leaf as an endpoint---that is, an edge $I \subset T$ with $\overline{I} \cap \mN_1(\mT) \neq \emptyset$. Let 
\begin{linenomath}\begin{equation*}
\{r_1,\ldots,r_M\} = \bigcup_{x \in \mN(\mT)} \{r^x_1,\ldots,r^x_{M_x}\},
\end{equation*}\end{linenomath}
let $\{\ell_1,\ldots,\ell_N\}$ denote the set of all lengths of leaf edges of $\mT$, and finally define 
\begin{linenomath}\begin{equation}\label{eqn:sigma_set}
\Sigma_\mT := \{\lambda_1 r_1 + \cdots + \lambda_M r_M + \xi_1 \ell_1 + \cdots + \xi_N \ell_N \mid \lambda_j, \xi_k \in \{0,1\} \}.
\end{equation}\end{linenomath}

The proof of Theorem \ref{thm:main_theorem_mtrees} hinges on the following technical lemma, based on ideas from \cite{agarwal2015computing}.

\begin{lemma}\label{lem:gromov_hausdorff_bound}
For metric trees $\mT$ and $\mS$ with $L^{\mathrm{K}}_{\mathrm{h}}(\mT,\mS) = 0$, there is a number $\delta$ in the set
\begin{linenomath}\begin{equation*}
\Delta := \left\{\frac{1}{2}|A - B| \mid A,B \in \Sigma_\mT \right\} \cup \left\{|A - B| \mid A,B \in \Sigma_\mT \right\}
\end{equation*}\end{linenomath}
such that $\frac{1}{2} d_\mathrm{GH}(T,S) \leq \delta \leq 14 d_{\mathrm{GH}}(T,S)$.
\end{lemma}

We will defer the proof of the lemma until after we have proved the theorem. Observe that the set $\Delta$ only depends on $\Sigma_\mT$---to prove the lemma we will show that $L^{\mathrm{K}}_{\mathrm{h}}(\mT,\mS) = 0$ implies $\Sigma_\mT = \Sigma_\mS$.

\begin{proof}[Proof of Theorem \ref{thm:main_theorem_mtrees}]
Let $\mT$ and $\mS$ be metric trees with $L^{\mathrm{K}}_{\mathrm{h}}(\mT,\mS) = 0$. Let $\epsilon_\mT = \frac{1}{14} \min \left(\Delta \setminus \{0\}\right)$. If $L^{\mathrm{K}}_{\mathrm{h}}(\mT,\mS) = 0$ and $\dgh(\mT,\mS) < \epsilon_\mT$, then Lemma \ref{lem:gromov_hausdorff_bound} implies that there is some $\delta \in \Delta$ such that
\begin{linenomath}\begin{equation*}
\delta \leq 14\, \dgh(\mT,\mS) < \min \left(\Delta \setminus \{0\}\right).
\end{equation*}\end{linenomath}
This forces $\delta = 0$. Applying the lower bound of Lemma \ref{lem:gromov_hausdorff_bound}, we conclude that $\dgh(\mT,\mS) = 0$. This implies that there exists an isometry $\phi:T \to S$ and we claim that this isometry must be measure-preserving. Indeed, the measures $\mu_T$ and $\mu_S$ each coincide with 1-dimensional Hausdorff measure (this is because Lebesgue measure on edges coincides with 1-dimensional Hausdorff measure---see \cite[Section 1.7]{burago2001course}), so the claim follows by \cite[Proposition 1.7.8]{burago2001course}. Therefore $\mT \approx \mS$.
\end{proof}

\subsubsection*{Proof of Lemma \ref{lem:gromov_hausdorff_bound}}

The proof uses the following necessary condition for vanishing $L^{\mathrm{K}}_{\mathrm{h}}$.

\begin{proposition}\label{prop:vertex_sets_equal}
Suppose that $\mG$ and $\mH$ are metric graphs with $L^{\mathrm{K}}_{\mathrm{h}}(\mG,\mH) = 0$. Then for every $v \in \mN(\mG)$ there exists $w \in \mN(\mH)$ such that $h_\mG(v,r) = h_\mH(w,r)$ for all $r \geq 0$. By symmetry, for every $w \in \mN(\mH)$ there exists $v \in \mN(\mG)$ such that $h_\mH(w,\cdot) = h_\mG(v,\cdot)$. It follows that the sets of functions $\{h_\mG(v,\cdot) \mid v \in \mN(\mG)\}$ and $\{h_\mH(w,\cdot) \mid w \in \mN(\mH)\}$ are equal.
\end{proposition}

\begin{remark}
As a technical point, note that there may be distinct vertices $v,v' \in \mN(\mG)$ such that $h_\mG(v,\cdot) = h_\mG(v',\cdot)$. The proposition is only keeping track of the sets of unique functions (i.e., without multiplicity).
\end{remark}

\begin{proof}
Suppose there is some $v \in \mN(\mG)$ such that $h_\mG(v,\cdot) \neq h_\mH(w,\cdot)$ for all $w \in \mN(\mH)$. If this is the case, then $h_\mG(v,\cdot) \neq h_\mH(y,\cdot)$ for any point $y$ in $H$, as $h_\mG(v,r)$ will differ from any $h_\mH(y,r)$ corresponding to a non-vertex $y$ for small values of $r$---see Lemma \ref{lem:degree_formula_for_local_distribution}. It follows that the continuous function $H \rightarrow \R_{\geq 0}$ on the compact space $H$ defined by $y \mapsto c_{\mG,\mH}(v,y)$, where $c_{\mG,\mH}$ is the cost function appearing in the definition of $L^{\mathrm{K}}_{\mathrm{h}}$ (see \eqref{eqn:local_distribution_cost_function}), achieves its minimum $m_v > 0$. Let $U$ denote the open set $\{x \in G \mid c_{\mG,\mG}(v,x) < m_v/2\}$. We claim that for any $x \in U$, 
\begin{linenomath}\begin{equation*}
\min_{y \in H} c_{\mG,\mH}(x,y) \geq \frac{m_v}{2}.
\end{equation*}\end{linenomath}
If this is not the case then there is  $x \in G$ such that $c_{\mG,\mG}(v,x) < m_v/2$ and $c_{\mG,\mH}(x,y) < m_v/2$ for some $y \in H$. It is not hard to see that the cost function satisfies a triangle inequality-like relation for all $v,x \in G$ and $y \in H$:
\begin{linenomath}\begin{equation*}
c_{\mG,\mH}(v,y) \leq c_{\mG,\mG}(v,x) + c_{\mG,\mH}(x,y).
\end{equation*}\end{linenomath}
Applying this to our particular points $v,x$ and $y$, it follows that 
\begin{linenomath}\begin{equation*}
c_{\mG,\mH}(v,y) \leq c_{\mG,\mG}(v,x) + c_{\mG,\mH}(x,y) < m_v,
\end{equation*}\end{linenomath}
which is a contradiction. 

Now we observe that for any coupling $\mu \in \mathcal{U}(\mu_G,\mu_H)$,
\begin{linenomath}\begin{equation*}
\int_{G\times H} c_{\mG,\mH}(x,y) \,\mu (dx \times dy) \geq \int_{U \times H} c_{\mG,\mH}(x,y) \, \mu(dx \times dy) \geq \frac{m_v}{2} \cdot \mu(U \times H) =  \frac{m_v}{2} \mu_G(U).
\end{equation*}\end{linenomath}
It follows that $L^{\mathrm{K}}_{\mathrm{h}}(\mG,\mH) > 0$. 
\end{proof}

Let $\mT$ be a metric tree, fix $x \in T$ and let $\{r^x_1,\ldots,r^x_{M_x}\}$ be defined as above. Observe that for $r \not \in \{r^x_1,\ldots,r^x_{M_x}\}$ (so that the slope of $h_\mT(x,\cdot)$ does not change at $r$), $\partial B_T(x,r)$ can only contain a vertex if it, in particular, contains a leaf. Indeed, as $\partial B_T(x,r)$ passes over a leaf vertex, it is possible for the slope of $h_\mT(x,r)$ to decrease, but as $\partial B_T(x,r)$ passes over any vertex of degree $\geq 3$, the slope of $h_\mT(x,r)$ increases. It follows that if $\partial B_T(x,r)$ contains only vertices of degree $\geq 3$, then the slope of $h_\mT(x,r+\epsilon)$ is strictly greater than that of $h_\mT(x,r)$ for all $\epsilon > 0$; this is a contradiction.

The discussion of the previous paragraph allows us to distinguish a finite list of candidate values $r$ where $\partial B_T(x,r)$ can contain a vertex, for any $x \in \mathcal{V}(T)$, as we show in the next lemma.

\begin{lemma}\label{lem:vertex_radii}
Let $\mT$ be a metric tree, let $x \in \mathcal{V}(T)$ and let $r_1^x,\ldots,r^x_{M_x}$ be as defined above. Let $\ell_1,\ldots,\ell_N$ denote the lengths of all leaf edges of $\mT$. Then $\partial B_T(x,r)$ can only contain a vertex if $r$ lies in the set
\begin{equation}\label{eqn:radii_containing_vertices}
\Sigma_\mT(x) := \{\lambda_1 r^x_1 + \cdots + \lambda_M r^x_{M_x} + \xi_1 \ell_1 + \cdots + \xi_N \ell_N \mid \lambda_j, \xi_k \in \{0,1\},\mbox{ at most one $\lambda_j \neq 0$}\}.
\end{equation}
\end{lemma}

\begin{proof}
Suppose that $\partial B_T(x,r)$ contains a vertex and that $r \not \in \{r^x_1,\ldots,r^x_M\}$. By the above discussion, $\partial B_T(x,r)$ must contain a leaf $v_1$. Let $J_1$ denote the unique path joining $x$ to $v_1$ (i.e., realizing the distance $d_T(x,v_1)$ from \ref{eqn:geodesic_distance_formula}---such paths are unique because of the tree topology of $T$). Let $x_1$ denote the vertex lying on $J_1$ which immediately precedes $v_1$, let $s_1=d_T(x,x_1)$ and let $\ell_{j_1} = d_T(x_1,v_1)$. Since $v_1$ is a leaf, $\ell_{j_1}$ is the length of its leaf edge. There are several cases to consider. If $x_1 = x$, then we are clearly finished, as $r = \ell_{j_1}$. If $s_1 = r^x_{j_1} \in \{r^x_1,\ldots,r^x_{M_x}\}$, then we are finished because this implies $r=r^x_{j_1}+\ell_{j_1} \in \Sigma_\mT(x)$. We therefore assume that we are in neither of these situations and iterate the process. That is, we have a vertex $x_1 \in \partial B_T(x,s_1)$, where $s_1 \not \in \{r^x_1,\ldots,r^x_{M_x}\}$, and it follows that $\partial B_T(x,s_1)$ contains a leaf $v_2$. Let $J_2$ denote the path from $x$ to $v_2$, let $x_2$ denote the vertex on $J_2$ preceding $v_2$, let $s_2 = d_T(x,x_2)$ and let $\ell_{j_2} = d_T(x_2,v_2)$. The algorithm terminates at this stage if $x_2=x$ (in which case $r = \ell_2 + \ell_1$) or if $s_2 = r^x_{j_2} \in \{r^x_1,\ldots,r^x_{M_x}\}$ (in which case $r = r^x_{j_2} + \ell_{j_2} + \ell_{j_1}$) and otherwise iterates again. The algorithm must eventually terminate, since the distances $s_j$ in each step are strictly decreasing.
\end{proof}

 Note that the set $\Sigma_\mT$ from \eqref{eqn:sigma_set}  contains the union of all $\Sigma_\mT(x)$ for $x \in \mN(\mT)$. In turn, $\Sigma_\mT$ contains the set of distances $\{d_T(v,w) \mid v,w \in \mN(\mT)\}$. In general, these containments are strict.

\begin{lemma}\label{lem:sigma_sets_equal}
Let $\mT$ and $\mS$ be metric trees with $L^{\mathrm{K}}_{\mathrm{h}}(\mT,\mS) = 0$. Then $\Sigma_\mT = \Sigma_\mS$.
\end{lemma}

\begin{proof}
We claim that the set of functions $\{h_\mT(v,\cdot) \mid v \in \mN(\mT)\}$ determines the set $\Sigma_\mT$---the result then follows by Proposition \ref{prop:vertex_sets_equal}. To prove the claim, observe that the set $\{h_\mT(v,\cdot) \mid v \in \mN(\mT)\}$ allows us to extract the radii $r_1^x,\ldots,r_{M_x}^x$ from Lemma \ref{lem:vertex_radii} for each $x \in \mN(\mT)$. We can therefore determine the set $\{r_1,\ldots,r_M\}$. Moreover, we can also extract the set of lengths $\{\ell_1,\ldots,\ell_N\}$ of all leaf edges of $\mT$. Indeed, a function $h_\mT(x,\cdot)$ is the function of a leaf if and only if $h_\mT(x,r) = r$ for all sufficiently small $r \geq 0$. Next, use the fact that the length of the leaf edge of a leaf $x$ is the first radius where the slope of $h_\mT(x,r)$ is not equal to one; that is, it is given by
\begin{linenomath}\begin{equation*}
\sup \{r_0 \geq 0 \mid h_\mT(x,r) = r \mbox{ for all $r < r_0$}\}.
\end{equation*}\end{linenomath}
\end{proof}

\begin{proof}[Proof of Lemma \ref{lem:gromov_hausdorff_bound}]
The paper \cite{agarwal2015computing} defines a distance on the space of metric trees, denoted for the course of this proof as $D(\mT,\mS)$, which is based on the notion of \emph{interleaving distance between merge trees} \cite{morozov2013interleaving}. We omit the details of their construction, but note that they prove $\frac{1}{2} d_\mathrm{GH}(T,S) \leq D(\mT,\mS) \leq 14 d_{\mathrm{GH}}(T,S)$---this is a combination of Lemmas 4.1 and 4.2 from \cite{agarwal2015computing}.

Moreover, Lemma 5.1 of \cite{agarwal2015computing} shows that $D(\mT,\mS)$ always lies in a set whose values are predetermined by the structures of $\mT$ and $\mS$, as we now describe. 
For $t \in \mN(\mT)$, let
\begin{linenomath}\begin{equation*}
\Lambda_{\mT}(t) := \left\{\frac{1}{2} \left|d_T(u,t) - d_T(v,t)\right| \mid u,v \in \mN(\mT)\right\}
\end{equation*}\end{linenomath}
Lemma \ref{lem:vertex_radii} implies that for any vertex $u \in \mN(\mT)$, the value $d_T(u,t)$ lies in the set $\Sigma_\mT(t)$. It follows that 
\begin{linenomath}\begin{equation*}
\Lambda_{\mT}(t) \subset \left\{\frac{1}{2}\left|A-A'\right| \mid A, A' \in \Sigma_\mT(t) \right\} \subset \left\{\frac{1}{2}\left|A-A'\right| \mid A, A' \in \Sigma_\mT \right\}
\end{equation*}\end{linenomath}
For $s \in \mN(\mS)$, we define the sets
\begin{align*}
    \Lambda_{\mS}(s) &:= \left\{\frac{1}{2} \left|d_S(u,s) - d_S(v,s)\right| \mid u,v \in \mN(\mS)\right\} \\
    \Lambda_{\mT,\mS}(t,s) &:= \left\{\left|d_T(u,t) - d_T(v,s)\right| \mid u \in \mN(\mT), \, v \in \mN(\mS)\right\}.
\end{align*}
By Lemma \ref{lem:sigma_sets_equal}, we have $\Sigma_\mT = \Sigma_\mS$, so that
\begin{linenomath}\begin{align*}
\Lambda_{\mS}(s) &\subset \left\{\frac{1}{2}\left|B-B'\right| \mid B, B' \in \Sigma_\mS(s) \right\} \subset \left\{\frac{1}{2}\left|B-B'\right| \mid B, B' \in \Sigma_\mT \right\}, \\
\Lambda_{\mT,\mS}(t,s) &\subset \left\{\left|A-B\right| \mid A \in \Sigma_\mT(t), \, B \in \Sigma_{\mS}(s) \right\} \subset \left\{\left|A-B\right| \mid A,B \in \Sigma_\mT \right\}
\end{align*}\end{linenomath}
Lemma 5.1 of \cite{agarwal2015computing} says that
\begin{linenomath}\begin{equation*}
D(\mT,\mS) \in \Lambda_{\mT}(t) \cup \Lambda_{\mS}(s) \cup \Lambda_{\mT,\mS}(t,s)
\end{equation*}\end{linenomath}
for some $t \in \mN(\mT)$ and $s \in \mN(\mS)$. The lemma then follows from the work above, as $\Lambda_{\mT}(t) \cup \Lambda_{\mS}(s) \cup \Lambda_{\mT,\mS}(t,s) \subset \Delta$ for all $t,s$. In particular, we take $\delta = D(\mT,\mS)$.
\end{proof}

\section{Discussion}\label{sec:discussion}

In this paper we formalized several inverse problems on recovering mm-spaces within a fixed subcategory from invariants defined in terms of distributions of distances: (\globalInj) on global injectivity of the invariants, (\homotopy) on the ability of the invariants to detect homotopy type, (\spheres) on the ability of the invariants to distinguish spheres, and (\local) on local injectivity of the invariants. We solved, at least in part, several of these problems in categories of plane curves, embedded manifolds, Riemannian manifolds, metric graphs and metric trees. Our results are summarized in Table \ref{tab:summary}. In each category, for each inverse problem, we mark \cmark if positive progress has been made and \xmark \ if we have a counterexample and, in each case, refer to the relevant result. Of course, these questions are not completely independent---e.g., global injectivity implies local injectivity---and this is reflected in the Table. We mark with $\circ$ if the problem is completely open in the category. The table should be read with the following disclaimer in mind: even if an entry is marked with \cmark or \xmark, this does not mean that the problem has been completely solved in full generality (far from it, in some cases). Future work will be to fill in the remaining gaps in Table \ref{tab:summary} and to extend these results to new categories of mm-spaces. Starting points for obtaining results in these directions are interspersed throughout Sections \ref{sec:plane_curves} and \ref{sec:metric_trees} as \emph{Questions}.

\begin{table}
    \centering
    \caption{Summary of inverse problem  results.}
    \label{tab:summary}
    \begin{tabular}{lccccccc}
    \toprule
    \multirow{2}{*}{Class of Spaces} & 
    \multicolumn{3}{c}{Local Distribution} &
    \multicolumn{4}{c}{Global Distribution} \\
    \cmidrule(r){2-4} \cmidrule(r){5-8}
    {} & (\globalInj) & (\homotopy) & (\local) & (\globalInj) & (\homotopy) & (\spheres) & (\local) \\
    \midrule
    Plane Curves & \cmark & n/a & \cmark & \xmark & n/a & \cmark & \xmark \\
    {} & Prop \ref{prop:local_distribution_curves} & -- & Prop \ref{prop:local_distribution_curves} & Prop \ref{prop:curve_histogram_conjecture_false} & -- & Thm \ref{thm:sphere_distributions}  & Exmp \ref{exmp:global_distribution_circle_densities} \\
    Emb. Manifolds & $\circ$ & $\circ$ & $\circ$ & \xmark & $\circ$ & \cmark & $\circ$ \\
    {} & -- & -- & -- & Exmp \ref{exmp:surfaces_same_distribution} & -- & Thm \ref{thm:sphere_distributions} & -- \\
    Riem. Manifolds & \cmark & \cmark & \cmark & \xmark & \cmark &  \cmark & $\circ$ \\
    {} & Prop \ref{prop:local_dist_riem_surf} & Thm \ref{thm:surface_characterization} & Prop \ref{prop:local_dist_riem_surf} & Exmp \ref{exmp:flat_tori} & Thm \ref{thm:surface_characterization} & Prop \ref{prop:sphere_distributions_riemannian} & -- \\
    Metric Graphs & \cmark & \cmark & \cmark & \xmark & $\circ$ & \cmark & \xmark \\
    {} & Thm \ref{thm:continuous_maps_graphs} & Thm \ref{thm:homotopy_type_metric_graphs} & Thm \ref{thm:continuous_maps_graphs} & Exmp \ref{exmp:non_injectivity_local_dist_c_trees} & -- & Thm \ref{thm:sphere_characterization_metric_graphs} & Exmp \ref{exmp:global_distribution_circle_densities}  \\
    Metric Trees & \xmark & n/a & \cmark & \xmark & n/a & n/a & $\circ$ \\
    {} & Exmp \ref{exmp:non_injectivity_local_dist_c_trees} & -- & Thm \ref{thm:main_theorem_mtrees} & Exmp \ref{exmp:non_injectivity_local_dist_c_trees} & -- & -- & -- \\
    \bottomrule
    \end{tabular}
\end{table}

\subsubsection*{Acknowledgements}

We acknowledge funding from these sources: NSF CCF 1526513, NSF DMS 1723003, NSF CCF 1740761, NSF DMS 1547357. We also thank Zane Smith and Peter Olver for useful conversations on some of the counterexamples presented in the paper.

\printbibliography

\end{document}